\documentclass[10pt,a4paper,reqno]{amsart}
\usepackage[margin=30mm]{geometry}

\usepackage[title]{appendix}
\usepackage{amssymb}
\usepackage{amsmath}
\usepackage{amsthm}
\usepackage{tikz-cd}
\usepackage{mathtools}
\usepackage{framed}
\usepackage{comment}
\definecolor{darkblue}{rgb}{0,0,0.6}
\usepackage[breaklinks, ocgcolorlinks,colorlinks=true, citecolor=darkblue, filecolor=darkblue, linkcolor=darkblue, urlcolor=darkblue]{hyperref}
\usepackage[capitalize]{cleveref}

\newcommand{\C}{\mathbb{C}}

\newcommand{\Q}{\mathbb{Q}}
\newcommand{\R}{\mathbb{R}}

\newcommand{\Z}{\mathbb{Z}}
\renewcommand{\H}{\mathbb{H}}
\newcommand{\F}{\mathbb{F}}

\renewcommand{\l}{\Lambda}
\newcommand{\wh}{\widehat}
\newcommand{\wt}{\widetilde}
\newcommand{\ol}{\overline}

\newcommand{\nsurj}{\twoheadrightarrow\hspace{-4.2mm}\parbox{1mm}{/}\hspace{3.2mm}}

\newcommand{\EE}{\mathscr{E}}

\DeclareMathOperator{\Aut}{Aut}
\DeclareMathOperator{\BPG}{BPG}

\DeclareMathOperator{\Hom}{Hom}
\DeclareMathOperator{\id}{id}

\DeclareMathOperator{\End}{End}

\DeclareMathOperator{\Res}{Res}

\DeclareMathOperator{\PSL}{PSL}

\DeclareMathOperator{\LF}{LF}

\DeclareMathOperator{\IM}{Im}

\DeclareMathOperator{\Ker}{Ker}

\DeclareMathOperator{\Wh}{Wh}

\DeclareMathOperator{\Syl}{Syl}

\DeclareMathOperator{\GL}{GL}
\DeclareMathOperator{\SL}{SL}

\DeclareMathOperator{\Cls}{Cls}

\DeclareMathOperator{\Inn}{Inn}
\DeclareMathOperator{\Out}{Out}
\DeclareMathOperator{\SF}{SF}
\DeclareMathOperator{\Picent}{Picent}

\DeclareMathOperator{\B}{\mathcal{B}}
\DeclareMathOperator{\PC}{PC}
\DeclareMathOperator{\SFC}{SFC}
\DeclareMathOperator{\G}{\Gamma}
\DeclareMathOperator{\MNEC}{MNEC}
\DeclareMathOperator{\pc}{PC}

\DeclareMathOperator{\Quot}{Quot}

\DeclareMathOperator{\QF}{QF}
\DeclareMathOperator{\FS}{FS}

\newtheorem{thm}{Theorem}[section]

\newtheorem*{thm*}{Theorem}
\newtheorem{prop}[thm]{Proposition}
\newtheorem*{prop*}{Proposition}
\newtheorem{lemma}[thm]{Lemma}
\newtheorem{corollary}[thm]{Corollary}
\newtheorem*{corollary*}{Corollary}

\newtheorem{question}[thm]{Question}
\newtheorem*{question*}{Question}

\makeatletter
\newtheorem*{rep@theorem}{\rep@title}
\newcommand{\newreptheorem}[2]{%
	\newenvironment{rep#1}[1]{%
		\def\rep@title{#2 \ref{##1}}%
		\begin{rep@theorem}}%
		{\end{rep@theorem}}}
\makeatother

\newreptheorem{theorem}{Theorem}
\newreptheorem{lemma}{Lemma}
\newreptheorem{proposition}{Proposition}
\newreptheorem{corollary}{Corollary}

\newtheorem{thmx}{Theorem}

\newtheorem{algorithm}[thm]{Algorithm}

\theoremstyle{definition}

\newtheorem*{theorem*}{Theorem}

\theoremstyle{remark}
\newtheorem{remark}[thm]{Remark}
\newtheorem*{remark*}{Remark}

\usepackage{rotating}

\usepackage{color}
\usepackage{enumerate}

\usepackage{cite}
\usepackage[nodisplayskipstretch]{setspace}
\usepackage{placeins}

\renewcommand{\l}{\Lambda}

\newenvironment{clist}[1]
{\begin{enumerate}[\normalfont #1]}
{\end{enumerate}}

\usepackage{mathrsfs}
\usepackage{mathabx}

\makeatletter
\@namedef{subjclassname@2020}{%
  \textup{2020} Mathematics Subject Classification}
\makeatother

\usetikzlibrary{decorations.markings}

\usetikzlibrary{patterns}

\usepackage[latin1]{inputenc}
\usepackage{tikz}
\usetikzlibrary{shapes,arrows}
\usetikzlibrary{calc,positioning}

\newcommand\dhxrightarrow[2][]{%
  \mathrel{\ooalign{$\xrightarrow[#1\mkern4mu]{#2\mkern4mu}$\cr%
  \hidewidth$\rightarrow\mkern4mu$}}
}

\usepackage{cancel}

\makeatletter
\newcounter{savesection}
\newcounter{apdxsection}
\renewcommand\appendix{\par
  \setcounter{savesection}{\value{section}}%
  \setcounter{section}{\value{apdxsection}}%
  \setcounter{subsection}{0}%
  \gdef\thesection{\@Alph\c@section}}
\newcommand\unappendix{\par
  \setcounter{apdxsection}{\value{section}}%
  \setcounter{section}{\value{savesection}}%
  \setcounter{subsection}{0}%
  \gdef\thesection{\@arabic\c@section}}
\makeatother

\usepackage{longtable}
\usepackage{cellspace} %
\usetikzlibrary{backgrounds}

\usepackage{caption}
\usepackage{bm}

\begin{document}

\title[Cancellation for projective modules over integral group rings]{The cancellation property for projective modules over integral group rings}

\author{John Nicholson}
\address{School of Mathematics and Statistics, University of Glasgow, United Kingdom}
\email{john.nicholson@glasgow.ac.uk}



\begin{abstract}
We obtain a partial classification of the finite groups $G$ for which the integral group ring $\Z G$ has projective cancellation, i.e. for which $P \oplus \Z G \cong Q \oplus \Z G$ implies $P \cong Q$ for projective $\Z G$-modules $P$ and $Q$. In particular, we determine when projective cancellation holds for a finite group with no exceptional binary polyhedral quotients.
To do this, we prove a cancellation theorem based on a relative version of the Eichler condition.
We then use a group theoretic argument to precisely determine the class of groups not covered by this result. 
The final classification is then obtained by applying results of Swan, Chen and Bley-Hofmann-Johnston which show failure of projective cancellation for certain groups.
\end{abstract}

\maketitle

\vspace{-5mm}
\section{Introduction}

Let $G$ be a finite group, let $\Z G$ denote the integral group ring and recall that a \textit{$\Z G$-lattice} is a $\Z G$-module whose underlying abelian group is $\Z^n$ for some $n$.  This can equivalently be viewed as an integral representation $G \to \GL_n(\Z)$.
Consider the following \textit{cancellation problem}: for $\Z G$-lattices $M$, $N$, when does $M \oplus \Z G \cong N \oplus \Z G$ imply $M \cong N$?
This is grounded in fundamental problems in both topology and number theory.
Its resolution, for certain $\Z G$-lattices, is essential in the classification of finite 2-complexes up to homotopy equivalence \cite{Jo03-book,Ni19}, closed 4-manifolds up to homeomorphism \cite{HK88,HK88b}, and in determining when a number field has a normal integral basis \cite{Ta81}.
These applications have so far only been obtained for particular finite groups $G$, and the apparent intractability of the cancellation problem over arbitrary finite groups presents an obstacle to extending such results to a more general setting.

We say that $\Z G$, or the group $G$, has \textit{projective cancellation} (PC) if $M \oplus \Z G \cong N \oplus \Z G$ implies $M \cong N$ for all finitely generated projective $\Z G$-modules $M$ and $N$.
This is equivalent to asking that finitely generated projective $\Z G$-modules are determined by their image in the K-group $K_0(\Z G)$. 
This is the simplest case of the cancellation problem for $\Z G$-lattices and is of particular significance since projective $\Z G$-modules arise naturally in many of the applications (e.g. \cite{Ta81}) and PC often implies cancellation for other $\Z G$-lattices (e.g. \cite{Ni20b,Ni19}).

This was first studied by Swan \cite{Sw60-II} who showed that PC fails for the quaternion group $Q_{32}$ of order 32 \cite{Sw62}.
Jacobinski \cite{Ja68} showed that $\Z G$ has PC if $G$ satisfies the \textit{Eichler condition}, which asks that the Wedderburn decomposition of $\R G$ has no copies of the quaternions $\H = M_1(\H)$. This holds if and only if $G$ has no quotient which is a \textit{binary polyhedral group}, i.e. a quaternion group $Q_{4n}$ for $n \ge 2$, or one of the exceptional groups $\widetilde{T}$, $\widetilde{O}$, $\widetilde{I}$ which are the binary tetrahedral, octahedral and icosahedral groups.
Fr\"{o}hlich \cite{Fr75} showed that, if $G$ has a quotient $H$ and $\Z G$ has PC, then $\Z H$ has PC.
Finally Swan \cite{Sw83} showed that, if $G$ is a binary polyhedral group, then PC fails for $Q_{4n}$ for $n \ge 6$ and holds otherwise. This leaves open the case where $G$ has a quotient of the form 
$Q_8$, $Q_{12}$, $Q_{16}$, $Q_{20}$, $\wt T$, $\wt O$ or $\wt I$
but no quotient of the form $Q_{4n}$ for $n \ge 6$.

In this article, we will introduce a method for dealing with this remaining family of groups. 
We also consider \textit{stably free cancellation} (SFC) which asks that cancellation holds for all $M$, $N$ free, i.e. stably free $\Z G$-modules are free. 
The following is a special case of the results we obtain. 

\begin{thmx} \label{thm:main-classification}
Let $G$ be a finite group such that $G \nsurj \wt T$, $\wt O$ or $\wt I$. 
Then the following are equivalent:
\begin{clist}{(i)}
\item $\Z G$ has projective cancellation
\item $\Z G$ has stably free cancellation
\item $G$ has no quotients of the form $Q_{4n}$ for $n \ge 6$, $Q_{4n} \times C_2$ for $2 \le n \le 5$, $G_{(32,14)}$, $G_{(36,7)}$, $G_{(64,14)}$ or $G_{(100,7)}$.
\end{clist}
\end{thmx}

Here $G_{(n,m)}$ denotes the $m$th group of order $n$ in GAP's Small Groups library \cite{BEO02, GAP4}, $C_n$ denotes the cyclic group of order $n$, and $G \nsurj H$ means that $H$ is not a quotient of $G$.

\begin{remark} \label{remark:after-main}
(a) The previous results summarised above determine PC and SFC for finite groups $G$ such that $G \nsurj Q_8$, $Q_{12}$, $Q_{16}$, $Q_{20}$, $\wt T$, $\wt O$, $\wt I$. For such groups, $\Z G$ has PC if and only if $\Z G$ has SFC if and only if $G \nsurj Q_{4n}$ for all $n \ge 6$. \cref{thm:main-classification} can be viewed as an extension of this result whereby we exclude just three binary polyhedral quotients rather than seven.

(b) A more general statement holds, though is more complicated to state.
Firstly, we can replace the condition $G \nsurj \wt T$, $\wt O$ or $\wt I$ with the much weaker condition that $G \nsurj \wt T \times C_2$, $\wt I \times C_2$, $G_{(96,66)}$, $G_{(384,18129)}$, $G_{(1152,155476)}$ or $Q_8 \rtimes \wt T^n$ (see \cref{s:groups-exceptional} for a definition) for $n \ge 1$. We would then need to add $\wt O \times C_2$ to the list in (iii). This follows from \cref{thm:MNEC-calculation-main} and \cref{fig:G_2}.
Secondly, we have that (i) $\Rightarrow$ (ii) $\Rightarrow$ (iii) for an arbitrary finite group $G$.

(c) If $G$ is a finite 2-group, then $G$ has no quotient of the form $\wt T$, $\wt O$ or $\wt I$ since $3 \mid |\wt T|, |\wt O|, |\wt I|$. For such groups, \cref{thm:main-classification} implies that $\Z G$ has PC if and only if $\Z G$ has SFC if and only if $G \nsurj Q_8 \times C_2$, $Q_{16} \times C_2$, $G_{(32,14)}$, $G_{(64,14)}$ or $Q_{2^n}$ for $n \ge 5$.
This determines PC and SFC for almost all finite groups subject to thevconjecture that almost all finite groups are $2$-groups.
For example, more than 99\% of groups of order at most 2000 are of order 1024 \cite{BEO02} (see also \cite{Bu22}).
\end{remark}

We will also establish the following. This applies to different groups to \cref{thm:main-classification}. For example, $\wt T \times C_2^2$ is covered by this result whilst $G_{(36,7)}$ is not. 

\begin{thmx}  \label{thm:main-prescribed}
Let $G$ be a finite group such that $G \twoheadrightarrow C_2^2$. Then the following are equivalent:
\begin{clist}{(i)}
\item
$\Z G$ has projective cancellation
\item
$\Z G$ has stably free cancellation
\item
$G$ has no quotients of the form $Q_{4n}$ for $n \ge 6$ even, $Q_{4n} \times C_2$ for $n=2,4$ or $n \ge 3$ odd, $\wt T \times C_2^2$, $\wt O \times C_2$, $\wt I \times C_2^2$, $G_{(32,14)}$ or $G_{(64,14)}$.
\end{clist}
\end{thmx}

\begin{remark} \label{remark:after-prescribed}
(a) If $G \twoheadrightarrow Q_8$ or $Q_{16}$, then $G \twoheadrightarrow C_2^2$ and so \cref{thm:main-prescribed} completely determines whether such a group $G$ has PC or SFC respectively (without the assumption that $G \nsurj \wt T$, $\wt O$, $\wt I$). If $G \twoheadrightarrow Q_8$, then $\Z G$ has PC (or SFC) if and only if $G \nsurj Q_8 \times C_2$ or $G_{(32,14)}$. If $G \twoheadrightarrow Q_{16}$, then $\Z G$ has PC (or SFC) if and only if $G \nsurj Q_{16} \times C_2, G_{(32,14)}$ or $G_{(64,14)}$.

(b) We have that (i) $\Rightarrow$ (ii) $\Rightarrow$ (iii) for an arbitrary finite group $G$.

(c) If $G$ is a non-cyclic finite 2-group, then $G \twoheadrightarrow C_2^2$ (see \cite[p253]{Mi26}). In particular, as in \cref{remark:after-main} (c), \cref{thm:main-prescribed} conjecturally determines PC and SFC for almost all finite groups.
\end{remark}

Knowing that PC and SFC coincide in each of these cases is of interest since there are currently no known examples of a group ring $\Z G$ where this is not the case.
Note that there exists a $\Z$-order in a semisimple $\Q$-algebra which has SFC but not PC \cite[Theorem 1.3]{SV19}.
 We therefore ask:

\begin{question} \label{question:PC=SFC}
Does there exist a finite group $G$ such that $\Z G$ has $\SFC$ but not $\PC$?
\end{question}

In the remainder of the introduction, we will introduce the method behind Theorems \ref{thm:main-classification} and \ref{thm:main-prescribed} as well as more general classification results which are discussed in \cref{ss:discussion}.
The method is based on two key new ideas. The first is a relative version of Jacobinski's theorem for quotients $G \twoheadrightarrow H$ satisfying the \textit{relative Eichler condition} (\cref{ss:cancellation-thm}).
The second is the notion of \textit{Eichler simple groups} and the \textit{Fundamental Lemma} which makes it possible to determine the groups which are not covered by cancellation theorems based on the relative Eichler condition (\cref{ss:group-theory}).
The outcome of this approach is that, provided a given list of finite groups can be shown to fail PC (resp. SFC), we would obtain a complete classification of the finite groups with PC (resp. SFC).
Our main results are then obtained by using results of Swan \cite{Sw83}, Chen \cite{Ch86} and Bley-Hofmann-Johnston \cite{BHJ24} which show failure of SFC for a number of these groups (\cref{ss:discussion}).

\newcommand{\depth}{1}
\addtocontents{toc}{\protect\setcounter{tocdepth}{\depth}}

\subsection*{Conventions}

From now on, $G$ and $H$ will denote finite groups. All groups will be considered up to group isomorphism. 
We often state properties of $\Z G$ as properties of $G$ (e.g. `$G$ has PC'). 
We write $G \twoheadrightarrow H$ (resp. $G \nsurj H$) to denote the statement that $H$ is a quotient of $G$ (resp. $H$ is not a quotient of $G$).
For a ring $R$, all $R$-modules will be taken to be finitely generated left $R$-modules.

\renewcommand{\depth}{2}
\addtocontents{toc}{\protect\setcounter{tocdepth}{\depth}}

\subsection{The relative Eichler condition and cancellation lifting}
 \label{ss:cancellation-thm}

Let $m_{\H}(G)$ denote the number of copies of the quaternions $\H = M_1(\H)$ in the Wedderburn decomposition of $\R G$, so that $G$ satisfies the Eichler condition if and only if $m_{\H}(G)=0$.
We say that a pair of finite groups $(G,H)$ satisfies the \textit{relative Eichler condition} if $G$ has a quotient $H$ and $m_{\H}(G) = m_{\H}(H)$. When this holds, we say $H$ is an \textit{Eichler quotient} of $G$ and that $G$ is an \textit{Eichler cover} of $H$. Otherwise, we say $H$ is a \textit{non-Eichler quotient} of $G$ and that $G$ is a \textit{non-Eichler cover} of $H$.

We say that $\Z H$ satisfies \textit{PC lifting} (resp. \textit{SFC lifting}) if $\Z G$ has PC (resp. SFC) whenever $G$ is an Eichler cover of $H$.
Note that this property entails that $\Z H$ itself has PC (resp. SFC).

We will establish the following relative version of Jacobinski's theorem \cite{Ja68}. 
This is a special case of a more general result which applies to $\Z$-orders in semisimple $\Q$-algebras (\cref{thm:canc-general}).

\begin{thmx} \label{thm:main-group-rings}
Let $H$ be a finite group such that:
\begin{clist}{(i)}
\item The map $\Z H^\times \to K_1(\Z H)$ is surjective
\item Every finitely generated projective $\Z H$-module is left isomorphic to an ideal in $\Z H$ which is generated by central elements.
\end{clist}
If $\Z H$ has $\PC$, then $\Z H$ has $\PC$ lifting, i.e. if $G$ has an Eichler quotient $H$, then $\Z G$ has $\PC$.
\end{thmx}

This generalises \cite[Theorem A]{Ni18}, which applied in the case of SFC and did not require condition (ii). The proof is based on the decomposition of $\Z G$ as a pullback of rings $\Z H \times_{(\Z/n\Z)[H]} \Lambda$ (see \cref{lemma:G-H-Square}). Condition (ii) implies that projective $\Z H$-modules are locally free as $(\Z H, \Z H)$-bimodules (\cref{subsection:picard}).
The main technical tool is \cref{lemma:lifting-two-sided-ideals} which shows that, modulo the action of locally free $(\Z G,\Z G)$-bimodules, an arbitrary projective $\Z G$-module $P$ can be assumed to map to a free module over $\Z H$ under extension of scalars.
Condition (i) then implies that the class $[P]$ has cancellation via a Mayer-Vietoris argument (\cref{subsection:fibre-squares}). 

We would next like to use \cref{thm:main-group-rings} to establish PC for many of the groups which remain after the results of Jacobinski, Fr\"{o}hlich and Swan. 
Swan showed that two additional families of groups have PC: $H \times K$ where 
$H \in \{Q_8, Q_{12}, Q_{16}, Q_{20}, \wt T, \wt O, \wt I \}$
and $K \nsurj C_2$ \cite[Theorem 15.5]{Sw83}, and $\wt T^n \times \wt I^m$ for any $n, m \ge 0$  \cite[Theorem 13.7]{Sw83}.
The following recovers PC in both cases.

\begin{thmx} \label{thm:main-cancellation}
Let $G$ be a finite group which has an Eichler quotient $H$ of the form
\begin{equation} 
\label{eq:BPGs+TxI} C_1, \, Q_8, \, Q_{12}, \, Q_{16}, \, Q_{20}, \, \widetilde{T}, \, \widetilde{O}, \, \widetilde{I} \quad \text{or} \quad \widetilde{T}^n \times \widetilde{I}^m \text{ for $n, m \ge 0$}.	
\end{equation}
Then $\Z G$ has $\PC$. That is, for each group $H$ listed above, $\Z H$ satisfies $\PC$ lifting.
\end{thmx}

The proof is in two parts. Firstly, we establish PC lifting for $H \in \{Q_8, Q_{12}, Q_{16}, Q_{20}\}$ using \cref{thm:main-group-rings}. We verify (i) using the results of \cite{MOV83} (see also \cite{Ni18}) and we verify (ii) by explicitly constructing projective ideals generated by central elements over each group (see \cref{thm:q8-20=two-sided}). Secondly, the case $H=\wt O$ or $\wt T^n \times \wt I^m$ for $n,m \ge 0$ is dealt with using a cancellation theorem of Swan \cite[Theorem 13.1]{Sw83}. This result was previously used by Swan to establish PC for the groups themselves but, by verifying additional hypothesis, we show that it can be used to also establish PC lifting (see \cref{prop:canc-2}).
The following question remains open.

\begin{question} \label{question:PC-lifting}
Does every group with $\PC$ satisfy $\PC$ lifting?
That is, if $G$ has an Eichler quotient $H$ such that $\Z H$ has $\PC$, then does $\Z G$ always have $\PC$?
\end{question}

\subsection{Eichler simple groups the Fundamental Lemma}
\label{ss:group-theory}

For a class of finite groups $S$, we say that $G$ is \textit{$S$-Eichler} if $G$ has an Eichler quotient $H$ for some $H \in S$. 
For example, \cref{thm:main-cancellation} says that $S$-Eichler implies PC where $S$ is the list of groups in \cref{eq:BPGs+TxI}.
We say that a finite group $G$ is \textit{Eichler simple} if it has no proper Eichler quotients. 
Examples include $C_1$, binary polyhedral groups, $\wt T^n \times \wt I^m$ for $n,m \ge 0$, $H \times C_2$ for $H$ a binary polyhedral group, and the groups listed in \cref{thm:main-classification} (iii) such as $G_{(32,14)}$. We say that a class of Eichler simple groups $S$ is \textit{closed under quotients} if $G \in S$ and $G \twoheadrightarrow H$ for $H$ an Eichler simple group implies that $H \in S$.
We will show:

\begin{lemma}[Fundamental Lemma] \label{lemma:fundamental-lemma}
Let $S$ be a class of Eichler simple groups which is closed under quotients. 
Then a finite group $G$ is $S$-Eichler if and only if $G$ has no quotient in $\MNEC(S)$.	
\end{lemma}

Here $\MNEC(S)$ denotes the class of \textit{minimal non-Eichler covers} of the set $S$ (see \cref{ss:MNEC-groups} for a definition).
We use this to establish the following, which is our main group theoretic result.

\begin{thmx} \label{thm:main-group-theory}
Let $G$ be a finite group. Then the following are equivalent:
\begin{clist}{(i)}
\item 
$G$ has an Eichler quotient $H$ of the form:
\[ C_1, \, Q_8, \, Q_{12}, \, Q_{16}, \, Q_{20}, \, \widetilde{T}, \, \widetilde{O}, \, \widetilde{I} \quad \text{or} \quad \widetilde{T}^n \times \widetilde{I}^m \text{ for $n, m \ge 0$}\]
\item 
$G$ has no quotient $H$ of the form:
\begin{clist}{(a)}
\item
$Q_{4n}$ for $n \ge 6$
\item 
$Q_8 \rtimes \widetilde{T}^n$ for $n \ge 1$
\item
$Q_8 \times C_2$, \, $Q_{12} \times C_2$, \, $Q_{16} \times C_2$, \, $Q_{20} \times C_2$, \, $\widetilde{T} \times C_2$, \, $\widetilde{O} \times C_2$, \, $\widetilde{I} \times C_2$, \, $G_{(32,14)}$, $G_{(36,7)}$, $G_{(64,14)}$, $G_{(96,66)}$, $G_{(100,7)}$, $G_{(384, 18129)}$, $G_{(1152,155476)}$.
\end{clist}
\end{clist}
\end{thmx}

The semidirect product $Q_8 \rtimes \widetilde{T}^n$ is defined in \cref{s:groups-exceptional}.
The proof is based on applying the Fundamental Lemma to the list of groups $S$ in \cref{eq:BPGs+TxI}.
We show that the groups in $S$ are all Eichler simple and that $S$ is closed under quotients. It remains to show that $\MNEC(S)$ consists of the groups in (ii) (a)-(c).
The calculation is broken into two parts. Firstly, we compute $\MNEC(H)$ for $H \in \{C_1, Q_8, Q_{12}, Q_{16}, Q_{20}, \widetilde{T}, \widetilde{O}, \widetilde{I}\}$ by means of an algorithm (Algorithm \ref{alg:MNEC(G)-original}) which we implement in GAP \cite{GAP4} and Magma \cite{magma} (\cref{s:MNEC-algorithms}).
Secondly, we calculate $\MNEC(\wt T^n \times \wt I^m)$ using a range of methods such as classifying extensions using group cohomology (\cref{s:groups-exceptional}).

\subsection{Consequences for the classification of finite groups with PC}
\label{ss:discussion}

An immediate consequence of Theorems \ref{thm:main-cancellation} and \ref{thm:main-group-theory} is that we obtain a potential strategy to completely classify the finite groups $G$ for which $\Z G$ has $\PC$.
In particular, by \cref{thm:main-cancellation}, the groups in \cref{thm:main-group-theory} (i) all have $\PC$. Hence, if the groups listed in \cref{thm:main-group-theory} (ii) (a)-(c) all failed PC, then the finite groups with PC would be precisely those which satisfy the hypotheses in \cref{thm:main-group-theory}.

Previously, PC was known to fail for $Q_{4n}$ for $n \ge 6$ and $Q_8 \times C_2$ by Swan \cite{Sw83} and for $Q_{12} \times C_2$, $Q_{16} \times C_2$, $Q_{20} \times C_2$, $G_{(36,7)}$ and $G_{(100,7)}$ by Chen \cite{Ch86}.
In March 2022, we sent the list of groups (a)-(c) to Bley-Hofmann-Johnston with the hope that the eight remaining groups in (c) and at least the smallest group $Q_8 \rtimes \wt T$ in (b) would fail PC. 
They showed the following \cite{BHJ24}:

\begin{thm}[Bley-Hofmann-Johnston] \label{thm:BHJ-part-1} 
$\SFC$ fails for $G_{(32,14)}$, $G_{(64,14)}$, $\wt T \times C_2^2$, $\widetilde{O} \times C_2$ and $\widetilde{I} \times C_2^2$.
$\SFC$ holds for $\wt T \times C_2$, $G_{(96,66)}$, $Q_8 \rtimes \wt T$, $\wt I \times C_2$, as well as $G_{(192,183)}$, $\wt T \times Q_{12}$, $\wt T \times Q_{20}$.
\end{thm}

This determines SFC for all groups listed in \cref{thm:main-group-theory} (ii) except  $G_{(384, 18129)}$, $G_{(1152,155476)}$ and $Q_8 \rtimes \wt T^n$ for any $n \ge 2$. The proof made use of computer calculations which become harder for larger groups.
We will now deduce \cref{thm:main-classification} from Theorems \ref{thm:main-cancellation}, \ref{thm:main-group-theory} and \ref{thm:BHJ-part-1}.

\begin{proof}[Proof of \cref{thm:main-classification}]
Suppose $G \nsurj \wt T$, $\wt O$ or $\wt I$. If $G$ has an Eichler quotient in \cref{eq:BPGs+TxI}, then \cref{thm:main-cancellation} implies $G$ has $\PC$. If not, then \cref{thm:main-group-theory} implies $G$ has a quotient $H$ which is one of the groups in \cref{thm:main-group-theory} (ii) (a)-(c).
The groups $Q_8 \rtimes \wt T^n$ for $n \ge 1$, $\wt T \times C_2$, $\wt O \times C_2$, $\wt I \times C_2$, $G_{(96,66)}$, $G_{(384, 18129)}$, $G_{(1152,155476)}$ all have quotients in $\wt T$, $\wt O$ or $\wt I$ (see \cref{fig:G_2}).
By \cite{Sw83}, \cite{Ch86} and \cref{thm:BHJ-part-1}, $\SFC$ fails for all remaining groups $H$ and so fails for $G$. The result follows.
\end{proof}

We conclude by discussing the finite groups for which PC (resp. SFC) remains open.
If $G$ has an Eichler quotient which is one the groups in \cref{eq:BPGs+TxI}, then $\Z G$ has PC. If not, $G$ has a quotient which is one of the groups in \cref{thm:main-group-theory} (ii). By results of Swan, Chen and Bley-Hofmann-Johnston, $G$ fails SFC if it has a quotient which is one of the groups in:
\begin{clist}{\hspace{3.2mm}(a)}
\item[(a)\hspace{1mm}]
$Q_{4n}$ for $n \ge 6$
\item[(c)$'$]
$Q_8 \times C_2$, $Q_{12} \times C_2$, $Q_{16} \times C_2$, $Q_{20} \times C_2$, $\widetilde{O} \times C_2$, $G_{(32,14)}$, $G_{(36,7)}$, $G_{(64,14)}$, $G_{(100,7)}$.
\end{clist}
This leaves open the case where $G$ has a quotient of the form $\widetilde{T} \times C_2$, $\wt I \times C_2$, $G_{(96,66)}$, $G_{(384, 18129)}$, $G_{(1152,155476)}$ or $Q_8 \rtimes \wt T^n$ for $n \ge 1$, but no quotient of the form (a) or (c)$'$.

Whilst $\wt T \times C_2$, $\wt I \times C_2$, $G_{(96,66)}$ and $Q_8 \rtimes \wt T$ all have SFC (by \cref{thm:BHJ-part-1}), the possibility remains that they fail PC.
However, using \cref{thm:BHJ-part-1}, we will show:

\begin{thm} \label{thm:TxC2}
$\wt T \times C_2$ has $\PC$. In particular, there exists a finite group with PC which does not satisfy the hypotheses in \cref{thm:main-group-theory}.
\end{thm}

This demonstrates that the classification of finite groups with PC suggested by \cref{thm:main-group-theory} is not complete.
We can attempt to remedy this via the following, which is an analogue of \cref{thm:main-group-theory}.

\begin{thm} \label{thm:TxC2-group-theory}
Let $G$ be a finite group such that $G \twoheadrightarrow \wt T \times C_2$. Then the following are equivalent:
\begin{clist}{(i)}	
\item
	$G$ has an Eichler quotient $\wt T \times C_2$
\item	
	$G$ has no quotient of the form 
	$Q_{4n}$ for $n \ge 6$,
	$\wt T \times C_2^2$, $\wt T \times Q_{12}$, $(Q_8 \rtimes \wt T) \times C_2$, $\wt T \times Q_{20}$, $\wt T \times \wt O$ or $\wt T^2 \times C_2$.
\end{clist}	
\end{thm}

If we could show that $\wt T \times C_2$ had PC lifting, then we would have completed the classification of finite groups with PC provided the following groups all fail PC: $G_{(96,66)}$, $\wt I \times C_2$, $G_{(384, 18129)}$, $G_{(1152,155476)}$, $\wt T \times Q_{12}$, $(Q_8 \rtimes \wt T) \times C_2$, $\wt T \times Q_{20}$, $\wt T \times \wt O$, $\wt T^2 \times C_2$ and $Q_8 \rtimes \wt T^n$ for $n \ge 1$.
This gives an iterative process whereby, if one of these groups has PC, we could prove a result such as \cref{thm:TxC2-group-theory} above and replace the group with several others.
Provided that PC can be determined for the infinite family $Q_8 \rtimes \wt T^n$, and all groups have PC lifting, this has the potential to reduce the complete classification of the finite groups with PC to finite computation.
At present, we know that $\wt T \times C_2$ has $\SFC$ lifting (\cref{prop:TxC2-lifting}) but PC lifting remains unknown (\cref{question:PC-lifting-TxC2}). By \cref{thm:BHJ-part-1}, $\wt T \times Q_{12}$ and $\wt T \times Q_{20}$ have SFC but the possibility remains that they fail PC.

In \cref{s:results}, we establish a number of additional partial classifications. Using Theorem \ref{thm:main-group-theory}, we study groups with $|G|$ bounded and $m_{\H}(G)$ bounded respectively (\cref{ss:classification-partial}). We prove \cref{thm:main-prescribed}, which involves establishing an analogue of Theorem \ref{thm:main-group-theory} for groups which have a $C_2^2$ quotient (\cref{ss:classification-prescribed-quotient}). 
Finally, we reflect on the applications of projective cancellation to topology given in \cite{Ni19,Ni20b} in the case where $G$ has periodic cohomology (\cref{ss:periodic-cohomology}). Our results lead to new and simpler proofs of the results obtained in those articles.

\subsection*{Organisation of the paper} The paper will be structured into three parts.
In \cref{p:PC-lifting}, we develop the theory of cancellation for locally free $\Z G$-modules and establish Theorems \ref{thm:main-group-rings} and \ref{thm:main-cancellation}.

In \cref{p:Group-theory}, we introduce and study the class of Eichler simple groups $\EE$. We prove the Fundamental Lemma (\cref{lemma:fundamental-lemma}).
We give algorithms for computing $\MNEC(H)$ for a group $H$.

In \cref{p:Computations}, the main computations are carried out. We calculate $\MNEC(\wt T^n \times \wt I^m)$ (without a computer program). The results of our computer calculations are presented in a table in Appendix \ref{s:tables}. 
Using both parts, we complete the proof of Theorem \ref{thm:main-group-theory}.

\subsection*{Acknowledgements}

Work on this article started while the author was a PhD student at University College London and Theorems \ref{thm:main-group-rings} and \ref{thm:main-cancellation} appeared in the author's PhD thesis (see \cite[Theorems A, B]{Ni21b}).
I would like to thank my PhD supervisor F. E. A. Johnson for his guidance and many interesting conversations. 
I would like to thank Werner Bley, Henri Johnston and Tommy Hofmann for many interesting conversations and for taking up the project (in parallel) of determining SFC for the groups in \cref{thm:main-group-theory} (a)-(c).
I would additionally like to thank Derek Holt for advice on group theoretic computations, and Alex Bartel, Henri Johnston and Mark Powell for helpful comments on the manuscript.
Lastly, I would like to thank Andy Thomas, for his assistance with using the Maths NextGen Compute Cluster, where much of the computations were run.
This work was supported by EPSRC grant EP/N509577/1, the Heilbronn Institute for Mathematical Research, and a Rankin-Sneddon Research Fellowship from the University of Glasgow.

\renewcommand{\depth}{2}
\addtocontents{toc}{\protect\setcounter{tocdepth}{\depth}}

\setcounter{tocdepth}{2}
\tableofcontents


\part{Cancellation and the relative Eichler condition}
\label{p:PC-lifting}


The goal of this part will be to establish Theorems \ref{thm:main-group-rings} and \ref{thm:main-cancellation} from the introduction.
Our approach will be to prove a general relative cancellation for $\Z$-orders in finite-dimensional semisimple $\Q$-algebras (\cref{thm:canc-general}). 
In \cref{s:prelim-LF-modules}, we develop the theory of locally free modules over orders. Much of this is taken from Swan \cite{Sw80,Sw83} and Fr\"{o}hlich \cite{Fr73}, but we offer slight variations on these results throughout (particularly 
\cref{prop:LF=>mod-n-free}, 
\cref{cor:lf-bimodule}, \cref{lemma:coset=K_R} and \cref{thm:coset=K_R}) which we believe have not previously appeared in this form. 
In \cref{s:main-cancellation-thm}, we then establish \cref{thm:canc-general}, our main cancellation theorem for orders. In 
\cref{s:specialisation-to-ZG}, we specialise to the case $\l=\Z G$ and prove Theorems \ref{thm:main-group-rings} and \ref{thm:main-cancellation}.

If $M$ is an $R$-module and $f:R \to S$ is a ring homomorphism, we will write the $S$-module which is image of $M$ under extension of scalars as $S \otimes M$, $S \otimes_R M$ or $f_\#(M)$.

\section{Locally free modules and fibre squares} \label{s:prelim-LF-modules}

We will now give a brief summary of the theory of locally free modules over orders in finite-dimensional semisimple $\Q$-algebras. 
Throughout this section, let $A$ be a finite-dimensional semisimple $\Q$-algebra and let $\Lambda$ be a $\Z$-order in $A$, i.e. a subring of $A$ which is finitely generated as an abelian group under addition and is such that $\Q \cdot \Lambda = A$. For example, we can take $\Lambda = \Z G$ and $A = \Q G$ for a finite group $G$.

\subsection{Preliminaries on locally free modules and fibre squares} \label{ss:LF-modules}

For a prime $p$, let $\Z_p$ denote the $p$-adic integers and let $\Lambda_p = \Z_p \otimes_\Z \Lambda$. We say a $\Lambda$-module $M$ is \textit{locally projective} if $M_p = \Z_p \otimes_\Z M$ is a projective $\Lambda_p$-module for all primes $p$. 
The following is well-known.

\begin{prop}[\hspace{-1mm}\text{\cite[Lemma 2.1]{Sw80}}] \label{prop:proj=locally-proj}
Let $M$ be a $\Lambda$-module. Then $M$ is projective if and only if $M$ is locally projective.
\end{prop}

Similarly, we say that $M$ is \textit{locally free (of rank $n$)} if there exists $n \ge 1$ for which $M_p$ is a free $\Lambda_p$-module of rank $n$ for all $p$ prime.
In the special case where $A = \Q G$ and $\Lambda = \Z G$ for a finite group $G$, we have the following refinement of Proposition \ref{prop:proj=locally-proj}.

\begin{prop}[\hspace{-1mm}\text{\cite[p156]{Sw80}}] \label{prop:free=locally-free}
Let $M$ be a $\Z G$-module. Then $M$ is projective if and only if $M$ is locally free.
\end{prop}

Define the \textit{locally free class group} $C(\Lambda)$ to be the set of equivalence classes of locally free $\Lambda$-modules up to the relation $P \sim Q$ if $P \oplus \Lambda^i \cong Q \oplus \Lambda^j$ for some $i,j \ge 0$. By abuse of notation, we write $[P]$ to denote both the class $[P] \in C(\Lambda)$ and, where convenient, the set of isomorphism classes of locally free modules $P_0$ where  $[P_0]=[P]$. 

We also define the \textit{class set} $\Cls(\Lambda)$ as the collection of isomorphism classes of rank one locally free $\Lambda$-modules, which is a finite set by the Jordan-Zassenhaus theorem \cite[Section 24]{CR81}. This is often written as $\LF_1(\l)$ (see \cite{Sw83}). 
This comes with the stable class map
\[ [ \,\cdot \,]_{\Lambda} : \Cls(\Lambda) \to C(\Lambda), \quad P \mapsto [P].\] 
This map is always surjective due to the following.
 This was proven by A. Fr\"{o}hlich in \cite{Fr75} using idelic methods, generalising the case $\Lambda = \Z G$ first obtained by Swan \cite[Theorem A]{Sw60-II}. However, it is worth noting that the first part follows already from the cancellation theorems of Bass and Serre \cite[Section 2]{Sw80}.

\begin{prop}[\hspace{-1mm}\text{\cite[p115]{Fr75}}] \label{prop:LF-rank1}
Let $M$ is a locally free $\Lambda$-module. Then there exists a left ideal $I \subseteq \l$ which is a rank one locally free $\l$-module and such that $M \cong I \oplus \l^i$ for some $i \ge 0$. 
\end{prop}

We say that $\Lambda$ has \textit{locally free cancellation} (LFC) if $P \oplus \Lambda \cong Q \oplus \Lambda$ implies $P \cong Q$ for all locally free $\Lambda$-modules $P$ and $Q$. By \cref{prop:LF-rank1}, we have that $\Lambda$ has LFC if and only if $[\,\cdot\,]_{\Lambda}$ is bijective, i.e $|\Cls(\l)| = |C(\Lambda)|$. When $\l=\Z G$, this coincides with projective cancellation (PC).

More generally, we say that a class $[P] \in C(\Lambda)$ has cancellation if $P_1 \oplus \Lambda \cong P_2 \oplus \Lambda$ implies $P_1 \cong P_2$ for all $P_1, P_2 \in [P]$.
We say that $\Lambda$ has \textit{stably free cancellation (SFC)} when $[\l]$ has cancellation, i.e. when every stably free $\l$-module is free.
We write $\Cls^{[P]}(\Lambda) = [\,\cdot\,]_{\Lambda}^{-1}([P])$ and $\SF(\Lambda) = \Cls^{[\Lambda]}(\Lambda)$ so that, by \cref{prop:LF-rank1}, $[P] \in C(\l)$ has cancellation if and only if $|\Cls^{[P]}(\l)| = 1$.

Recall that $[P] \in C(\Lambda)$ can be represented as a graded tree with vertices the isomorphism classes of non-zero modules $P_0 \in [P]$, edges between each $P_0 \in [P]$ and $P_0 \oplus \Lambda \in [P]$ and with grading from the rank of each locally free $\Lambda$-module. 
By \cref{prop:LF-rank1}, this takes the following simple form where the set of minimal vertices corresponds to $\Cls^{[P]} (\Lambda)$ (see \cref{figure:fork}).

\begin{figure}[h] \vspace{-2mm}  \center
\begin{tikzpicture}
\draw[fill=black] (0,0) circle (2pt);
\draw[fill=black] (1,0) circle (2pt);
\draw[fill=black] (2,0) circle (2pt);
\draw[fill=black] (3,0) circle (2pt);
\draw[fill=black] (4,0) circle (2pt);
\draw[fill=black] (2,1) circle (2pt);
\draw[fill=black] (2,2) circle (2pt);
\draw[fill=black] (2,3) circle (2pt);
\node (a) at (2,3.6) {$\vdots$};
\draw[thick] (0,0) -- (2,1) (1,0) -- (2,1) (2,0) -- (2,1) (3,0) -- (2,1) (4,0) -- (2,1) -- (2,2) -- (2,3);
\end{tikzpicture}
\caption{Tree structure for $[P] \in C(\Lambda)$}
\label{figure:fork}
\end{figure}
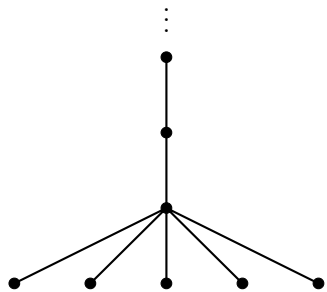

The following is a consequence of a general cancellation theorem of Jacobinski. The proof depends on deep results of Eichler on strong approximation \cite{Ei37}.

\begin{thm}[\hspace{-1mm}\text{\cite[Theorem 4.1]{Ja68}}] \label{thm:jacobinski}
Suppose $\l$ satisfies the Eichler condition. Then $\l$ has LFC. In particular, $\Cls^{[P]}(\l) = \{ P\}$ for all $P \in \Cls(\l)$.
\end{thm}

We will now observe that locally free $\Lambda$-modules cannot be detected on $A$ or on any finite ring quotients of $\Lambda$. For example, if $\l = \Z G$, then $M \in \Cls(\Z G)$ has $\Q \otimes M \cong \Q G$ and $\F_p \otimes M \cong \F_p G$. More generally, we have the following.

\begin{prop} \label{prop:LF=>mod-n-free}
Let $M$ be a locally free $\Lambda$-module. 
Then $\Q \otimes M$ is a free $A$-module and, if $f : \Lambda \to \bar{\Lambda}$ is a surjective ring homomorphism for a finite ring $\bar{\Lambda}$, then $\bar{\Lambda} \otimes M$ is a free $\bar{\Lambda}$-module.
\end{prop}

The result over $\Q$ is a consequence of the Noether-Deuring theorem and can be found in \cite[p.~407]{FRU74}. 
The result over finite rings is known \cite[Remark 4.9]{RU74} though we were not able to locate a proof in the literature except the case $\l = \Z G$ \cite[Theorem 7.1]{Sw60-II}, so we include one below.

\begin{proof}
We will start by considering the case $\bar{\l}=\Z/n\Z$.
First note that $\Lambda/n \cong (\Z/n\Z) \otimes_\Z \Lambda$. In particular, if $M$ is a $\Lambda$-module, then $M/n \cong (\Z/n\Z) \otimes_\Z M$ can be viewed as a $\Lambda/n$-module.

By Proposition \ref{prop:LF-rank1}, it suffices to consider the case where $M$ is locally free of rank one. By Proposition \ref{prop:proj=locally-proj}, we have that $\Z_p \otimes_\Z M \cong \Lambda_p$ for all $p$ prime. Since $\Z \hookrightarrow \Z_p$ induces an isomorphism $\Z/p\Z \to \Z_p/p\Z_p$, we have that
\begin{align*} M /p &\cong (\Z_p/p\Z_p) \otimes_\Z M \cong (\Z_p/p\Z_p) \otimes_{\Z_p} (\Z_p \otimes_\Z M)  \\ 
&\cong (\Z_p/p\Z_p) \otimes_{\Z_p} \Lambda_p \cong (\Z_p/p\Z_p) \otimes_\Z \Lambda \cong  \Lambda / p. \end{align*}
Let $a_p \in M$ be such that $[a_p] \in M/p$ maps to $1 \in \Lambda / p$ under this isomorphism. As in the proof of Hensel's lemma, we can check that $\Lambda/(p^k) \to M/(p^k)$, $1 \mapsto [a_p]$ is an isomorphism for all $k \ge 1$

If $n = p_1^{k_1} \cdots p_r^{k_r}$ is a factorisation into distinct primes, then the Chinese remainder theorem implies that there exists $\alpha_i \in \Z$ such that $\alpha_i \equiv 1$ mod $p_i$ and $\alpha_i \equiv 0$ mod $p_j$ for $i \ne j$. By the argument above, there exists $a_i \in M$ such that the map $1 \mapsto [a_i]$ gives an isomorphism $\Lambda/p^{k_i} \to M/p^{k_i}$. If $a = \sum_{i=1}^r \alpha_i a_i$, then $\Lambda/n \to M/n$, $1 \mapsto a$ is an isomorphism, as required.

Now let $\bar{\l}$ be an arbitrary finite ring.
Note that $\bar{\Lambda} = \Lambda/I$ for a two-sided ideal $I \subseteq \Lambda$ and, since $\Lambda$ is a $\Z$-order in $A$, we have that $\Z \subseteq \Lambda$. Since $\bar{\Lambda}$ is finite, $I \subseteq \Lambda$ must have finite index and so $I \cap \Z = (n) \subseteq \Z$ for some $n \ge 1$. In particular, this implies that $n \Lambda \subseteq I$ and so there is a composition $\Lambda \to \Lambda/n \to \Lambda/I$.
Hence, if $M$ is a locally free $\Lambda$-module, then $\bar{\Lambda} \otimes M \cong \bar{\Lambda} \otimes (M/n)$. Since $M/n$ is a free $\Lambda/n$-module by the case $\bar{\l}=\Z/n\Z$ above, $\bar{\Lambda} \otimes M$ must be a free $\bar{\Lambda}$-module.
\end{proof}

Let $\l$, $\l_1$ and $\Lambda_2$ be $\Z$-orders in finite-dimensional semisimple $\Q$-algebras $A$, $A_1$, $A_2$ respectively, let $\widebar{\l}$ be a finite ring and suppose there is a Milnor square:
\[
\mathcal{R} = 
\begin{tikzcd}
  \l \arrow[r, "i_2"] \arrow[d, "i_1"] & \l_2 \arrow[d,"j_2"] \\
  \l_1 \arrow[r,"j_1"] & \widebar{\l} 
\end{tikzcd}
\]
Since $\widebar{\l}$ is a finite ring, we have that $\Q \otimes \widebar{\l} = 0$. Since $\Q$ is a flat module, tensoring the above diagram with $\Q$ gives another pullback diagram which implies that the map
\[ \Q \otimes (i_1,i_2): \Q \otimes \l \to (\Q \otimes \l_1) \times (\Q \otimes \l_2)\]
is an isomorphism, i.e. $A \cong A_1 \times A_2$.

For a ring $R$, recall that
$K_1(R)=\GL(R)^{\text{ab}}$
where $\GL(R) = \bigcup_n \GL_n(R)$ with respect to the inclusions $\GL_n(R) \hookrightarrow \GL_{n+1}(R)$.
The following is a consequence of the Mayer-Vietoris sequence for $\mathcal{R}$ combined with \cref{prop:LF=>mod-n-free}. This is proven in \cite[6.2]{RU74}. 

\begin{prop} \label{prop:MV-special}
If $\mathcal{R}$ is as above, then there is an exact sequence
\[ K_1(\l) \xrightarrow[]{} K_1(\l_1) \times K_1(\l_2) \xrightarrow[]{} K_1(\widebar{\l}) \xrightarrow[]{\partial} C(\l) \xrightarrow[]{} C(\l_1) \times C(\l_2) \xrightarrow[]{} 0 \]	
where $\partial$ is the snake map and all other maps are functorial.	
\end{prop}

We will now give general conditions under which we can relate cancellation over two orders $\l_1$ and $\l_2$ when there is a map $f: \l_1 \to \l_2$.
The following was shown by Swan and generalises an earlier result of Fr\"{o}hlich \cite[VIII]{Fr75}.

\begin{thm}[\hspace{-1mm}\text{\cite[Theorem A10]{Sw83}}] \label{thm:LF-COR}
Let $f: \l_1 \to \l_2$ be a map of $\Z$-orders in a semisimple $\Q$-algebra $A$ such that the induced map $f_* : \Q \otimes_\Z \l_1 \twoheadrightarrow \Q \otimes_\Z \l_2$ is surjective.
Then the diagram
\[
\begin{tikzcd}
	\Cls(\l_1) \ar[r,"f_\#"] \ar[d,"\text{$[\,\cdot \,]_{\l_1}$}"] & \Cls(\l_2) \ar[d,"\text{$[\,\cdot \,]_{\l_2}$}"] \\
	C(\l_1) \ar[r,"f_\#"] &  C(\l_2)
\end{tikzcd}
\]
is a weak pullback square with all maps surjective.
\end{thm}

In particular, if $P_1 \in \Cls(\l_1)$ and $P_2 = f_\#(P_1) \in \Cls(\l_2)$, then this implies that the map
\[ f_\# : \Cls^{[P_1]}(\l_1) \to \Cls^{[P_2]}(\l_2)\]
is surjective. Hence, if $[P_1]$ has cancellation, then $[P_2]$ has cancellation.

Let $G$ be a finite group with quotient $H$. The situation of \cref{thm:LF-COR} arises when $\l_1 = \Z G$, $\l_2 = \Z H$ and $f : \Z G \to \Z H$ is induced by the quotient map and is itself surjective.
This implies that the properties PC (i.e. LFC) and SFC are closed under quotients of groups.

\subsection{Central Picard groups} \label{subsection:picard}

We will now consider the question of when a locally free $\l$-module can be represented by a two-sided ideal $I \subseteq \l$ and so has the additional structure of a bimodule. This was first considered by Fr\"{o}hlich \cite{Fr73} and Fr\"{o}hlich-Reiner-Ullom \cite{FRU74}.

Recall that, for a ring $R$, an $(R,R)$-bimodule $M$ is \textit{invertible} if there exists an $(R,R)$-bimodule $N$ and bimodules isomorphisms
\[ f: M \otimes_R N \to R, \quad g: N \otimes_R M \to R\]
such that the following diagrams commute:
\[
\begin{tikzcd}
	M \otimes_R N \otimes_R M \ar[r,"f \otimes \id"] \ar[d,"\id \otimes g"] & R \otimes_R M \ar[d] \\
	M \otimes_R R \ar[r] & M
\end{tikzcd}
\quad \quad
\begin{tikzcd}
	N \otimes_R M \otimes_R N \ar[r,"g \otimes \id"] \ar[d,"\id \otimes f"] & R \otimes_R N \ar[d] \\
	N \otimes_R R \ar[r] & N
\end{tikzcd}
\]
The \textit{central Picard group} $\Picent(R)$ is the group of $(R,R)$-bimodule isomorphism classes of $(R,R)$-bimodules $M$ for which $x m = m x$ for all $m \in M$ and central elements $x \in Z(R)$. 

In the special case where $R = \l$ is an order in a semisimple $\Q$-algebra, the central Picard group has the following basic properties. 

Firstly, let $I(\l)$ denote the multiplicative group of two-sided ideals $I \subseteq \l$ which are invertible in the sense that there exists a fractional two-sided ideal $J \subseteq \Q \cdot \l$ for which $I \cdot J = \l$. If $I \in I(\l)$, then it follows that $I$ is invertible as a $(\l,\l)$-bimodule and $x m = m x$ for all $m \in I$, $x \in Z(\l)$. The following can be found in \cite[Corollary 55.18]{CR87}.

\begin{thm} \label{thm:bimodules-are-ideals}
There is an isomorphism of abelian groups:
\[ \Picent(\l) \cong I(\l)/\{\l a : a \in (\Q \cdot C)^\times\}\]
where $C = Z(\l)$ is the centre of $\l$.
\end{thm}

In particular, every $M \in \Picent(\l)$ is bimodule isomorphic to an invertible two-sided ideal. If $I$, $J \in I(\l)$, then $I \cong J$ as $(\l,\l)$-bimodules if and only if there exists $a \in (\Q \cdot C)^\times$ such that $I = J a$.

The following is a consequence of combining this with \cite[Proposition 55.29]{CR87}.

\begin{prop} \label{prop:two-sided=>invertible}
Let $I \subseteq \l$ be a two-sided ideal such that $I \in \Cls(\l)$. Then there exists $J \subseteq \l$ two-sided with $J \in \Cls(\l)$ such that $I \otimes_{\l} J \cong \l \cong J \otimes_{\l} I$ as $(\l,\l)$-bimodules. In particular, $I \in I(\l)$.
\end{prop}

In particular this shows that, if $I \subseteq \l$ be a two-sided ideal such that $I \in \Cls(\l)$, then  $I$ induces a bijection $I \otimes_{\l} - : \Cls(\l) \to \Cls(\l)$.
We can therefore consider an even stronger notion of local freeness than for left modules.
We say that a $(\l,\l)$-bimodule $M$ is \textit{locally free as a bimodule} if there exists $i \ge 1$ such that, for all $p$ prime, $M_p \cong \l_p^i$ are isomorphic as $(\l_p,\l_p)$-bimodules. 
We will now need the following two closely related results. 

\begin{prop}[\hspace{-1mm}\text{\cite[Proposition 55.16]{CR87}}] \label{prop:picent-centre}
Let $R$ be a commutative Noetherian local ring and let $\l$ be a commutative finitely generated $R$-algebra. Then $\Picent(\l) = 1$.
\end{prop}

In particular, if $\l$ is an $\Z$-order in a finite-dimensional semisimple $\Q$-algebra $A$ and $C = Z(\l)$, then $C_p$ is a commutative finitely generated $\Z_p$-algebra and $C_{(p)}$ is a commutative finitely generated $\Z_{(p)}$-algebra. Since $\Z_p$ and $\Z_{(p)}$ are both Noetherian, this implies that $\Picent(C_p) = 1$ and $\Picent(C_{(p)})=1$.

The following was shown by Fr\"{o}hlich (see also \cite[Theorem 55.25]{CR87}). 
Note that, since $\tau' \circ \tau$ factors $\Picent(C_p)$, the fact that $\tau' \circ \tau = 0$ follows from $\Picent(C_p)=1$.

\begin{thm}[\hspace{-1mm}\text{\cite[Theorem 6]{Fr73}}] \label{thm:picent-SES}
For all but finitely many primes $p$, we have $\Picent(\l_p)=1$  and there is an exact sequence
\[ 1 \to \Picent(C) \xrightarrow[]{\tau} \Picent(\l) \xrightarrow[]{\tau'} \prod_{p} \Picent(\l_p) \to 1\]
where $C = Z(\l)$ is the centre of $\l$ and $\tau(M) = \l \otimes_C M$ for $M \in \Picent(C)$.
\end{thm}

This leads to the following three equivalent characterisations of locally free bimodules. This is presumably well-known, though we were not able to locate an equivalent statement in the literature.

\begin{corollary} \label{cor:lf-bimodule}
Let $I \subseteq \l$ be a two-sided ideal such that $I \in \Cls(\l)$. 
The following are equivalent:
\begin{clist}{(i)}
\item $I$ is generated by central elements
\item $I$ is locally free as a bimodule, i.e. for all $p$ prime, $I_p \cong \l_p$ are isomorphic as $(\l_p,\l_p)$-bimodules. 
\item For all $p$ prime, $I_{(p)} \cong \l_{(p)}$ are isomorphic as $(\l_{(p)},\l_{(p)})$-bimodules. 
\end{clist}
\end{corollary}

\begin{proof}
By \cref{prop:two-sided=>invertible}, $I$ represents a class $[I] \in \Picent(\l)$. The equivalence of (i) and (ii) now follows from \cref{thm:bimodules-are-ideals} and \cref{thm:picent-SES}.

Now, (iii) implies (ii) since $\l_p = \Z_p \otimes_{\Z} \l \cong \Z_p \otimes_{\Z_{(p)}} \l_{(p)}$. 
In order to show that (i) implies (iii), suppose that $I$ is generated by central elements and, for $p$ prime, let $\tau'' : \Picent(\l) \to \Picent(\l_{(p)})$ be the induced map. Then there is a commutative diagram:
\[
\begin{tikzcd}
  \Picent(C) \arrow[r, "\tau"] \arrow[d] & \Picent(\l) \arrow[d,"\tau''"] \\
  \Picent(C_{(p)}) \arrow[r] & \Picent(\l_{(p)}) 
\end{tikzcd}
\]
where all maps are the induced maps.
By \cref{prop:picent-centre}, we have that $\Picent(C_{(p)}) = 1$ and so $\tau'' \circ \tau = 0$. Since $I$ is generated by central elements, $[I] \in \IM(\tau)$ and so 
\[ [I_{(p)}] = \tau''(I) = 0 \in \Picent(\l_{(p)})\] 
which implies that $I_{(p)} \cong \l_{(p)}$ as $(\l_{(p)},\l_{(p)})$-bimodules.
\end{proof}

\subsection{Cancellation over fibre squares} \label{subsection:fibre-squares}

For each splitting $A \cong A_1 \times A_2$ of $\Q$-algebras, let $\l_1$, $\l_2$ be the projections onto $A_1$, $A_2$ which are $\Z$-orders in $A_1$ and $A_2$ respectively. If $\l_1 = \l/I_1$, $\l_2 = \l/I_2$ and $\widebar{\l} = \l/(I_1 + I_2)$ then, by \cite[Example 42.3]{CR87}, there is a pullback diagram
\[
\mathcal{R} = 
\begin{tikzcd}
  \l \arrow[r, "i_2"] \arrow[d, "i_1"] & \l_2 \arrow[d,"j_2"] \\
  \l_1 \arrow[r,"j_1"] & \widebar{\l}
\end{tikzcd}
\]
since $I_1 \cap I_2 = \{0\}$. Since $\Q \otimes (i_1,i_2)$ induces the isomorphism $A \to A_1 \times A_2$, we must have that $\Q \otimes \widebar{\l} = 0$ which implies that $\widebar{\l}$ is a finite ring.
We will write $\mathcal{R} = \mathcal{R}(\l,A_1,A_2)$ to denote the diagram induced by the splitting $A \cong A_1 \times A_2$. Here $A_1$, $A_2$ are assumed to be subrings of $A$.

Consider the maps
\[ \Cls_{\mathcal{R}} : \Cls(\l) \to \Cls(\l_1) \times \Cls(\l_2), \quad C_{\mathcal{R}} : C(\l) \to C(\l_1) \times C(\l_2)\] which are both induced by the extension of scalars maps $((i_1)_\#,(i_2)_\#)$.

\begin{lemma} \label{lemma:Sw-Fr}
Let $P \in \Cls(\l)$ and let $P_k = (i_k)_{\#}(P) \in \Cls(\l_k)$ for $k=1,2$. Then $\Cls_{\mathcal{R}}$ restricts to a surjection
\[ \Cls_{\mathcal{R}} : \Cls^{[P]}(\l) \twoheadrightarrow \Cls^{[P_1]}(\l_1) \times \Cls^{[P_2]}(\l_2)\]
and $[\,\cdot\,]_\l$ restricts to a surjection
\[ [\,\cdot\,]_\l : \Cls_{\mathcal{R}}^{-1}(P_1,P_2) \twoheadrightarrow C_{\mathcal{R}}^{-1}([P_1],[P_2]).\]
\end{lemma}

\begin{proof}
Since $(i,j) : \l \to \l_1 \times \l_2$ is a map of $\Z$-orders in $A$ such that $\Q \otimes (i_1,i_2)$ is an isomorphism, \cref{thm:LF-COR} implies that the diagram 
\[
\begin{tikzcd}
	\Cls(\l) \ar[r,"\Cls_{\mathcal{R}}"] \ar[d,"\text{$[\,\cdot \,]_\l$}"] & \Cls(\l_1) \times \Cls(\l_2) \ar[d,"\text{$[\,\cdot \,]_{\l_1} \times [\,\cdot \,]_{\l_2}$}"] \\
	C(\l) \ar[r,"C_{\mathcal{R}}"] & C(\l_1) \times C(\l_2)
\end{tikzcd}
\]
is a weak pullback in that $\Cls(\l)$ maps onto the pullback of the lower right corner. Hence the fibres of $[\,\cdot\,]_\l$ and $\Cls_{\mathcal{R}}$ map onto the fibres of $[\,\cdot \,]_{\l_1} \times [\,\cdot \,]_{\l_2}$ and $C_{\mathcal{R}}$ respectively, as required.
\end{proof}

In order to determine when $[\,\cdot\,]$ is bijective, i.e. when $\l$ has LFC, it is therefore useful to give the following explicit forms for the fibres of $\Cls_{\mathcal{R}}$ and $C_{\mathcal{R}}$ respectively.

For a $\l$-module $M$, let $\End_{\l}(M)$ be the ring of $\l$-module homomorphisms $f: M \to M$ and let $\Aut_{\l}(M) = \End_{\l}(M)^\times$ be the subset of $\l$-module isomorphisms $f: M \to M$. We will write $\End(M)$ and $\Aut(M)$ when $\l$ is clear from the context.
Since $j_1$ and $j_2$ are both surjective in our construction above, $\mathcal{R}$ is a Milnor square. If $P_i \in \Cls(\l_i)$, then we also have that $(j_k)_\#(P_k)$ is free by \cref{prop:LF=>mod-n-free} (ii). The following are now a consequence of \cite[Theorems 2.1-2.3]{Mi71} and \cref{prop:MV-special} respectively.

\begin{prop} \label{prop:milnor-square}
For $k=1,2$, let $P_k \in \Cls(\l_k)$. Then there is an isomorphism of $\widebar{\l}$-modules $\varphi_k : (j_k)_\#(P_k) \to \widebar{\l}$ and, for any such $\varphi_k$, there is a one-to-one correspondence
\[ \Aut(P_1) \backslash \widebar{\l}^\times \hspace{-1mm} \slash \Aut(P_2) \leftrightarrow \Cls_{\mathcal{R}}^{-1}(P_1,P_2) \]
where the maps $\Aut(P_k) \to \widebar{\l}^\times$ are defined via the map $\varphi_k(1 \otimes -) : P_k \to \widebar{\l}$.
\end{prop}

\begin{prop} \label{prop:K_1-square}
For $k=1,2$, let $P_k \in \Cls(\l_k)$. Then there is a one-to-one correspondence
\[ \frac{K_1(\widebar{\l})}{K_1(\l_1) \times K_1(\l_2)} \leftrightarrow C_{\mathcal{R}}^{-1}([P_1],[P_2]).\]
\end{prop}

For $i = 1, 2$, Morita equivalence gives us maps
\[\psi_i: \Aut(P_i) = \End(P_i)^\times \to K_1(\End(P_i)) \cong K_1(\l_i).\]
By \cite[Corollary A17]{Sw83}, these maps fit into a commutative diagram
\[
\begin{tikzcd}
\Aut(P_i) \ar[r] \ar[d,"h_{\#}"] & K_1(\End(P_i)) \ar[r,"\cong"] \ar[dr,swap,"K_1(h_{\#})"] & K_1(\l_i) \ar[d,"K_1(h)"] \\
\widebar{\l}^\times \ar[rr] & & K_1(\widebar{\l})
\end{tikzcd}
\]
for $h= j_1$ or $j_2$ respectively, and so define a map
\[  \Psi_{P_1,P_2}: \Aut(P_1) \backslash\, \widebar{\l}^\times \slash \Aut(P_2) \twoheadrightarrow \frac{K_1(\widebar{\l})}{K_1(\l_1) \times K_1(\l_2)} \]
which can be shown to coincide with the map induced by $[\,\cdot\,]$ under the equivalences given in Propositions \ref{prop:milnor-square} and \ref{prop:K_1-square}.
Hence, by the discussion earlier, we have that $\l$ has LFC if and only if $\Psi_{P_1,P_2}$ is a bijection for all $P_1$, $P_2$ such that $P_1 = (i_1)_\#(P)$ and $P_2 = (i_2)_\#(P)$ for some $P \in \Cls(\l)$. 

Now consider the constant $K_{\mathcal{R}} = |\frac{K_1(\widebar{\l})}{K_1(\l_1) \times K_1(\l_2)}|$ associated to $\mathcal{R}$. It follows from Lemma \ref{lemma:Sw-Fr} that $[\,\cdot\,]_\l$ is surjective and so $|\Aut(P_1)\backslash \widebar{\l}^\times \slash \Aut(P_2)| \ge K_{\mathcal{R}}$.

\begin{lemma} \label{lemma:coset=K_R}
Let $P_1 \in \Cls(\l_1)$ and $P_2 \in \Cls(\l_2)$ and suppose that
\[ |\Aut(P_1)\backslash \, \widebar{\l}^\times \slash \Aut(P_2)| = K_{\mathcal{R}}.\]
Then $| \Cls_{\mathcal{R}}^{-1}(P_1,P_2) \cap \Cls^{[\widetilde{P}]}(\l)|=1$ for all $[\widetilde{P}] \in C_{\mathcal{R}}^{-1}([P_1],[P_2])$.
\end{lemma}

\begin{proof}
By Propositions \ref{prop:milnor-square} and \ref{prop:K_1-square}, we have that 
\[ |\Aut(P_1)\backslash \, \widebar{\l}^\times \slash \Aut(P_2)| = |\Cls_{\mathcal{R}}^{-1}(P_1,P_2)|, \quad K_{\mathcal{R}} = |C_{\mathcal{R}}^{-1}([P_1],[P_2])|\]
and so $|\Cls_{\mathcal{R}}^{-1}(P_1,P_2)| = |C_{\mathcal{R}}^{-1}([P_1],[P_2])|$ by our hypothesis. By Lemma \ref{lemma:Sw-Fr}, this implies that the map
\[ [\,\cdot\,]_\l : \Cls_{\mathcal{R}}^{-1}(P_1,P_2) \to C_{\mathcal{R}}^{-1}([P_1],[P_2])\]
is a bijection. The result follows since $[\,\cdot\,]_\l^{-1}([\widetilde{P}]) = \Cls_{\mathcal{R}}^{-1}(P_1,P_2) \cap \Cls^{[\widetilde{P}]}(\l)$.
\end{proof}

We are now ready to prove the following, which is the main result of this section.

\begin{thm} \label{thm:coset=K_R}
Let $\mathcal{R}$ be as above, let $P \in \Cls(\l)$ and let $P_k = (i_k)_{\#}(P) \in \Cls(\l_k)$ for $k=1,2$. 
Suppose that $|\Cls_{\mathcal{R}}^{-1}(\widetilde{P}_1,\widetilde{P}_2)| = K_{\mathcal{R}}$ for all $\widetilde{P}_1 \in \Cls^{[P_1]}(\l_1)$ and $\widetilde{P}_2 \in \Cls^{[P_2]}(\l_2)$.
Then $\Cls_{\mathcal{R}}$ induces a bijection
\[ \Cls^{[P]}(\l) \cong \Cls^{[P_1]}(\l_1) \times \Cls^{[P_2]}(\l_2).\]	
\end{thm}

Note that, subject to the hypothesis, this implies that $[P] \in C(\l)$ has cancellation if and only if $[P_1] \in C(\l_1)$ and $[P_2] \in C(\l_2)$ cancellation.

\begin{proof}
Recall that, by Lemma \ref{lemma:Sw-Fr}, there is a surjection
\[ \Cls_{\mathcal{R}} \mid_{\Cls^{[P]}(\l)} : \Cls^{[P]}(\l) \to \Cls^{[P_1]}(\l_1) \times \Cls^{[P_2]}(\l_2)\]
which has fibres $(\Cls_{\mathcal{R}} \mid_{\Cls^{[P]}(\l)})^{-1}(\widetilde{P}_1,\widetilde{P}_2) = \Cls_{\mathcal{R}}^{-1}(\widetilde{P}_1,\widetilde{P}_2) \cap \Cls^{[P]}(\l)$. By Lemma \ref{lemma:coset=K_R}, we have that $|\Cls_{\mathcal{R}}^{-1}(\widetilde{P}_1,\widetilde{P}_2) \cap \Cls^{[P]}(\l)|=1$ for all $\widetilde{P}_1 \in \Cls^{[P_1]}(\l_1)$ and $\widetilde{P}_2 \in \Cls^{[P_2]}(\l_2)$ and this implies that $\Cls_{\mathcal{R}} \mid_{\Cls^{[P]}(\l)}$ is a bijection.
\end{proof}

\section{Main cancellation theorem for orders in semisimple $\Q$-algebras} \label{s:main-cancellation-thm}

As before, let $\l$ be a $\Z$-order in a finite-dimensional semisimple $\Q$-algebra $A$, and let $\mathcal{R} = \mathcal{R}(\l,A_1,A_2)$ denote the fibre square corresponding to a splitting $A \cong A_1 \times A_2$ of $\Q$-algebras.

The aim of this section will be to prove the following which is our main cancellation theorem for orders in semisimple $\Q$-algebras.
For a ring $R$, the map $R^\times \to K_1(R)$ is taken to be the composition $R^\times = \GL_1(R) \hookrightarrow \GL(R) \twoheadrightarrow K_1(R)$.

\begin{thm} \label{thm:canc-general}
Let $P \in \Cls(\Lambda)$ and let $P_1 = (i_1)_\#(P) \in \Cls(\Lambda_1)$. Suppose the following conditions are satisfied by $\mathcal{R}$:
\begin{clist}{(i)}
\item $\l_2$ satisfies the Eichler condition
\item The map $\Lambda_1^\times \to K_1(\Lambda_1)$ is surjective
\item Every $\wt P_1 \in \Cls^{[P_1]}(\Lambda_1)$ is represented by a two-sided ideal $I \subseteq \l_1$ which is generated by central elements.
\end{clist}
Then the map $(i_1)_\# : \Cls^{[P]}(\Lambda) \to \Cls^{[P_1]}(\Lambda_1)$ is a bijection. In particular, $[P] \in C(\l)$ has cancellation if and only if $[P_1] \in C(\l_1)$ has cancellation.
\end{thm}

By \cref{cor:lf-bimodule}, a two-sided ideal $I \subseteq \l_1$ is locally free as a $(\l_1,\l_1)$-bimodule if and only if it is generated by central elements. In particular, hypothesis (iii) is satisfied if $I$ is generated by central elements.
The proof will be given in \cref{subsection:proof} and will depend on  Lemmas \ref{lemma:eichler}, \ref{lemma:unit-rep-double-cosets} and \ref{lemma:lifting-two-sided-ideals} which roughly correspond to the three conditions in \cref{thm:canc-general}.
The first lemma is due to Swan, the second is due to the author \cite{Ni18} and the third has not been previously observed.

\subsection{The Eichler condition} \label{subsection:eichler}

The following is a consequence of \cite[A17 and A18]{Sw83}.

\begin{thm} \label{thm:eichler-quotient}
Let $\l$ be a $\Z$-order in a semisimple $\Q$-algebra, let $P \in \Cls(\l)$, let $\bar{\l}$ be a finite ring and let $f: \l \twoheadrightarrow \widebar{\l}$ be a surjective ring homomorphism. Then $f_\#(\Aut(P)) \le \widebar{\l}^\times$ is a normal subgroup and the map $\Aut(P) \to K_1(\widebar{\l})$ induces an isomorphism
\[ \widebar{\l}^\times \slash \Aut(P) \cong K_1(\widebar{\l}) \slash K_1(\l).\]
\end{thm}

\begin{remark} 
By \cref{thm:coset=K_R}, it can be shown that this implies \cref{thm:jacobinski}.
\end{remark}

We can apply this to the case where $\mathcal{R} = \mathcal{R}(\l,A_1,A_2)$ for a splitting of $\Q$-algebras $A \cong A_1 \times A_2$. 

\begin{lemma} \label{lemma:eichler}
Let $\mathcal{R}$ be as above.
If $\l_2$ satisfies the Eichler condition then, for all $P_k \in \Cls(\l_k)$ for $k=1,2$, there is a bijection 
$\Cls_{\mathcal{R}}^{-1}(P_1,P_2) \cong \Cls_{\mathcal{R}}^{-1}(P_1,\l_2)$.
\end{lemma}

\begin{proof}
By \cref{thm:eichler-quotient}, there are isomorphisms $\widebar{\l}^\times \slash \Aut(P_2) \cong K_1(\widebar{\l}) \slash K_1(\l_2) \cong \widebar{\l}^\times \slash \l^\times$. This implies that there is a bijection
\[ \Aut(P_1) \backslash \, \widebar{\l}^\times \slash \Aut(P_2) \cong \Aut(P_1) \backslash \, \widebar{\l}^\times \slash \l_2^\times \]
which is equivalent to $\Cls_{\mathcal{R}}^{-1}(P_1,P_2) \cong \Cls_{\mathcal{R}}^{-1}(P_1,\l_2)$ by \cref{prop:milnor-square}.
\end{proof}

\subsection{Unit representation for $K_1$} \label{subsection:unit-rep}

The following observation was first made in \cite{Ni18}.

\begin{lemma} \label{lemma:unit-rep-double-cosets}
Let $\mathcal{R}$ be as above and suppose:
\begin{clist}{(i)}
\item $\l_2$ satisfies the Eichler condition
\item The map $\Lambda_1^\times \to K_1(\Lambda_1)$ is surjective.
\end{clist}
Then $|\l_1^\times \backslash \, \widebar{\l}^\times \slash \l_2^\times| = K_{\mathcal{R}}$.
\end{lemma}

\begin{proof}
Since $m_{\H}(\l_2)=0$, \cref{thm:eichler-quotient} implies that the map $\widebar{\l}^\times \to K_1(\widebar{\l})$ induces an isomorphism $\widebar{\l}^\times \slash \l_2^\times \cong K_1(\widebar{\l}) \slash K_1(\l_2)$.
The relevant maps fit into a commutative diagram
\[
\begin{tikzcd}
	\l_1^\times \ar[r] \ar[d] & \widebar{\l}^\times / \l_2^\times \ar[d,"\cong"] \\
	K_1(\l_1) \ar[r] & \displaystyle \frac{K_1(\widebar{\l})}{K_1(\l_2)}
\end{tikzcd}
\]
and so $\IM(\l_1^\times \to \widebar{\l}^\times / \l_2^\times) = \IM(K_1(\l_1) \to K_1(\widebar{\l})/K_1(\l_2))$ since the map $\l_1^\times \to K_1(\l_1)$ is surjective. Hence we have $|\l_1^\times\backslash \, \widebar{\l}^\times \slash \l_2^\times| = K_{\mathcal{R}}$.
\end{proof}

\subsection{Two-sided ideals over orders in semisimple $\Q$-algebras} \label{section:two-sided}

The main result of this section is as follows. This gives a method of constructing locally free two-sided ideals over $\l$ from locally free two-sided ideals over the projections $\l_k$ subject to certain conditions.

\begin{lemma} \label{lemma:lifting-two-sided-ideals}
Let $\mathcal{R}$ be as above and, for $k=1,2$, suppose $I_k \subseteq \Lambda_k$ is a two-sided ideal such that $I_k \in \Cls(\l_k)$ and which is generated by central elements.
Then there exists a two-sided ideal $I \subseteq \Lambda$ with $I \in \Cls(\l)$ such that
\begin{clist}{(i)}
\item For $k=1,2$, there is a $(\Lambda_k,\Lambda)$-bimodule isomorphism
\[(i_k)_\#(I) \cong (I_k)_{i_k}\]
\item For all $P_k \in \Cls(\l_k)$ for $k=1,2$, there is a bijection 
\[I \otimes_{\l} - : \Cls_{\mathcal{R}}^{-1}(P_1,P_2) \to \Cls_{\mathcal{R}}^{-1}(I_1 \otimes_{\l_1} P_1, I_2 \otimes_{\l_2} P_2).\]
\end{clist}
\end{lemma}

\begin{remark}
(i) This actually holds under the weaker hypothesis that $(I_k)_{(p)} \cong (\l_k)_{(p)}$ are isomorphic as $((\l_k)_{(p)},(\l_k)_{(p)})$-bimodules for all primes $p \mid |\widebar{\l}|$. However, since we cannot currently see an application for this in the $\Z G$ case, we will restrict to the centrally generated case for simplicity.

(ii) This can be viewed as a generalisation of a construction of Gustafson-Roggenkamp \cite[p376-377]{GR88} which applies to the case $I_k = \l_k$ for $k=1,2$. Whilst we could similarly consider different $I$ corresponding to units $\alpha \in Z(\widebar{\l})^\times$, we will avoid this for simplicity.
\end{remark}

We will begin by proving the following embedding result, which can be viewed as a generalisation of \cite[Theorem A]{Sw60-II} to bimodules.

\begin{prop} \label{prop:embed-ideal}
Let $I \subseteq \l$ be a two-sided ideal generated by central elements such that $I \in \Cls(\l)$. Then, for all $n \ne 0$, there exists a two-sided ideal $J \subseteq \l$ such that $I \cong J$ as $(\l,\l)$-bimodules and $J \cap \Z$ is coprime to $(n)$.
\end{prop}

Note that, by \cref{cor:lf-bimodule}, the hypothesis is satisfied whenever $I$ is locally free as a bimodule.
In order to prove this, we will need the following two lemmas.

\begin{lemma} \label{lemma:embed-ideal-1}
Let $n \ne 0$ be an integer and let $I \subseteq \l$ be a two-sided ideal such that $I \in \Cls(\l)$ and, for all $p \mid n$ prime, there is a $(\l_{(p)},\l_{(p)})$-bimodule isomorphism $I_{(p)} \cong \l_{(p)}$. Then there is a $(\l/n,\l/n)$-bimodule isomorphism $f: \l/n \to I/n$, $1 \mapsto [a]$ for some $a \in Z(\l) \cap I$.
\end{lemma}

\begin{proof}
For each $p \mid n$ prime, consider the bimodule isomorphism $f: \l_{(p)} \to I_{(p)}$, $1 \mapsto [a_p]$ for some $a_p \in I_{(p)}$. There exists $m \ne 0$ such that $m a_p \in I \subseteq I_{(p)}$ and $f' : \l_{(p)} \to I_{(p)}$, $1 \mapsto [ma_p]$ is still a bimodule isomorphism, and so we can assume that $a_p \in I \subseteq I_{(p)}$.
Since $f$ is a bimodule isomorphism, we have that $a_p \in Z(\l_{(p)})$ and so $a_p \in Z(\l) \cap I$.

Now, $f$ induces a bimodule isomorphism $\l_{(p)}/p \cong I_{(p)}/p$. Since $\Z_{(p)}/p \cong \Z/p$, there are bimodule isomorphisms $\l/p \cong \l_{(p)}/p$ and $I/p \cong I_{(p)}/p$ and so there exists a bimodule isomorphism $f_p : \l/p \to I/p$, $1 \mapsto [a_p]$. It is straightforward to check that the map $\widetilde{f}_{p^i} : \l/p^i \to I/p^i$, $1 \mapsto [a_p]$ is also a bimodule isomorphism for all $i \ge 1$.

In general, suppose $n = p_1^{n_1} \cdots p_k^{n_k}$ for distinct primes $p_i$, and integers $n_i \ge 1$ and $k \ge 1$. By the Chinese remainder theorem, $\Z/n \cong \Z/p_1^{n_1} \times \cdots \times \Z/p_k^{n_k}$. By tensoring with $\l$ or $I$, we see that there are bimodule isomorphisms $\l/n \cong \l/p_1^{n_1} \times \cdots \times \l/p_k^{n_k}$ and $I/n \cong I/p_1^{n_1} \times \cdots \times I/p_k^{n_k}$. Hence, by the bimodule isomorphism constructed above, there is a bimodule isomorphism $f: \l/n \to I/n$, $1 \mapsto [a]$ for some $a \in \Z \cdot \langle a_{p_1}, \cdots, a_{p_k}\rangle \subseteq Z(\l) \cap I$.
\end{proof}

\begin{lemma} \label{lemma:embed-ideal-2}
Let $n \ne 0$ be an integer, let $I \subseteq \l$ be a two-sided ideal such that $I \in \Cls(\l)$, and let $f : \l/n \to I/n$, $1 \mapsto [a]$ be a $(\l/n,\l/n)$-bimodule isomorphism for some $a \in Z(\l) \cap I$. Then $\l a \cong \l$ as a $(\l,\l)$-bimodule and there exists $m \ne 0$ such that $m I \subseteq \l a$ and $(n,m)=1$.
\end{lemma}
	
\begin{proof}
Since $a \in Z(\l)$, $\l a$ is a bimodule and there is a map of bimodules $\varphi : \l \to \l a$, $x \mapsto xa$. To see that $\varphi$ is a bimodule isomorphism, note that it is clearly surjective and is injective since $f$ is a bijection.

Since $f$ is an isomorphism, we have $I = \l a + n I$ as ideals in $\l$ and so there is an equality of finitely generated abelian groups $I/\l a = n \cdot I/\l a$. Hence, as an abelian group, $I/\l a$ is finite of order $m$ where $(n,m)=1$. Since $m \cdot I/\l a =0$, we have that $m I \subseteq \l a$.
\end{proof}	
	
\begin{proof}[Proof of \cref{prop:embed-ideal}]
Let $I \subseteq \l$ be a two-sided ideal such that $I \in \Cls(\l)$ and which is generated by central elements. By \cref{cor:lf-bimodule}, this implies that $\l_{(p)} \cong I_{(p)}$ are isomorphic as bimodules for all $p$. By \cref{lemma:embed-ideal-1}, there is a $(\l/n,\l/n)$-bimodule isomorphism $f: \l/n \to I/n$, $1 \mapsto [a]$ for some $a \in Z(\l) \cap I$. 
By \cref{lemma:embed-ideal-2}, this implies that there is a $(\l,\l)$-bimodule isomorphism $\psi : \l a \to \l$, $x a \mapsto x$ and there exists $m \ne 0$ with $(n,m)=1$ and $m I \subseteq \l a$. 
Let 
\[ J = \psi(m I) \subseteq \psi(\l a) = \l,\] 
which is a two-sided ideal since $\psi$ is a map of bimodules. 
Finally, note that the map $I \to J$, $x \mapsto \psi(mx)$ is a $(\l,\l)$-bimodule isomorphism, and $m = \psi(ma) \in \psi(mI) = J$ implies that $J \cap \Z = (m_0)$ where $m_0 \mid m$ and so $(n,m_0)=1$.
\end{proof}

We will need the following lemma. In the statement of \cref{lemma:lifting-two-sided-ideals}, this shows (i) implies (ii).

\begin{lemma} \label{lemma:two-sided-induces-bijection}
Let $\mathcal{R}$ be as above and suppose $I \subseteq \l$, $I_k \subseteq \l_k$ are two-sided ideals such that $I \in \Cls(\l)$, $I_k \in \Cls(\l_k)$ and $(i_k)_\#(I) \cong (I_k)_{i_k}$ are isomorphic as $(\l_k,\l)$-bimodules for $k = 1,2$. Then, for all $P_k \in \Cls(\l_k)$ for $k=1,2$, there is a bijection 
\[I \otimes_{\l} - : \Cls_{\mathcal{R}}^{-1}(P_1,P_2) \to \Cls_{\mathcal{R}}^{-1}(I_1 \otimes_{\l_1} P_1, I_2 \otimes_{\l_2} P_2).\]
\end{lemma}

\begin{proof}
By \cref{prop:two-sided=>invertible}, there exists a two-sided ideal $J \subseteq \l$ such that $J \in \Cls(\l)$ and 
$I \otimes_{\l} J \cong \l \cong J \otimes_{\l} I$
as $(\l,\l)$-bimodules. In particular, $I$ is invertible as a bimodule and determines a bijection
$I \otimes_{\l} - : \Cls(\l) \to \Cls(\l)$
with inverse $J \otimes_{\l} -$.

Now suppose $P \in \Cls_{\mathcal{R}}^{-1}(P_1,P_2)$, i.e. that $(i_k)_\#(P) \cong P_k$ are isomorphic as left $\l_k$-modules for $k=1,2$. Then:
\begin{align*} (i_k)_\#(I \otimes_{\l} P) &= \l_k \otimes_{\l} (I \otimes_{\l} P) \cong (I_k)_{i_k} \otimes_{\l} P \\ 
&\cong (I_k \otimes_{\l_k} \l_k) \otimes_{\l} P \cong I_k \otimes_{\l_k} P_k \end{align*}
and so 
\[ (I \otimes_{\l} - )(\Cls_{\mathcal{R}}^{-1}(P_1,P_2)) \subseteq \Cls_{\mathcal{R}}^{-1}(I_1 \otimes_{\l_1} P_1, I_2 \otimes_{\l_2} P_2).\] 

Similarly, we can show that 
\[ (J \otimes_{\l} -)(\Cls_{\mathcal{R}}^{-1}(I_1 \otimes_{\l_1} P_1, I_2 \otimes_{\l_2} P_2)) \subseteq \Cls_{\mathcal{R}}^{-1}(P_1,P_2).\] 
Hence $I \otimes_{\l} -$ restricts to the required bijection.	
\end{proof}

Finally, we will now use \cref{prop:embed-ideal} to complete the proof of \cref{lemma:lifting-two-sided-ideals}.

\begin{proof}[Proof of \cref{lemma:lifting-two-sided-ideals}]
By \cref{lemma:two-sided-induces-bijection}, it suffices to prove part (i) only. 
Let $k=1$ or $2$. By \cref{prop:embed-ideal} we can assume, by replacing $I_k$ with a bimodule isomorphic two-sided ideal, that $I_k \cap \Z$ is coprime to $|\widebar{\l}|$. Let $n \in I_k \cap \Z$ be such that $n \ne 0$ and let $m \in \Z$ be such that $nm \equiv 1$ mod $|\widebar{\l}|$, which exists since $(|\widebar{\l}|,n)=1$. Consider the left $\widebar{\l}$-module homomorphisms
\[\psi_k : \widebar{\l} \to \widebar{\l} \otimes_{\l_k} I_k, \quad 1 \mapsto m \otimes n, \quad \quad \varphi_k : \widebar{\l} \otimes_{\l_k} I_k \to \widebar{\l}, \quad x \otimes y \mapsto x j_k(y)\]
where $x \in \widebar{\l}$ and $y \in I_k \subseteq \l_k$. Note that 
\[ \varphi_k(\psi_k(1)) = m j_k(n) = mn = 1 \in \widebar{\l}\] 
and 
\[ \psi_k(\varphi_k(x \otimes y)) = (xj_k(y) m) \otimes n = xm \otimes yn = xmn \otimes y = x \otimes y.\] 
This shows that $\psi_k$ and $\varphi_k$ are mutual inverses and so are both bijections.

Now let 
\[M = \{ (x_1,x_2) \in I_1 \times I_2 : \varphi_1(1 \otimes x_1) = \varphi_2(1 \otimes x_2)\} \subseteq \l_1 \times \l_2,\] 
which is a left $\l$-module under the action $\lambda \cdot (x_1,x_2) = (i_1(\lambda) x_1, i_2(\lambda)x_2)$ for $\lambda \in \l$. This coincides with the standard pullback construction for projective module over a Milnor square $\mathcal{R}$ \cite{Mi71}. 
However, for the $\varphi_k$ chosen above, we further have that 
\[ M = \{ (x_1,x_2) \in I_1 \times I_2 : j_1(x_1) = j_2(x_2)\}\] 
and so $M$ is a $(\l,\l)$-bimodule with action 
\[ \lambda \cdot (x_1,x_2) \cdot \mu = (i_1(\lambda) \cdot x_1 \cdot i_1(\mu), i_2(\lambda) \cdot x_2 \cdot i_2(\mu))\] 
for $\lambda$, $\mu \in \l$.

Note that $M \subseteq \l_1 \times \l_2 \subseteq \Q \cdot (\l_1 \times \l_2) = \Q \cdot \l$ and so there exists $k \in \Z$ with $k \ne 0$ for which $k M \subseteq \l$. Hence $I = k M$ is a two-sided ideal in $\l$ which is bimodule isomorphic to $M$. Now note that $M \in \Cls(\l)$ as a left $\l$-module by \cite[Lemma 4.4]{RU74}, and so $I \in \Cls(\l)$. Finally note that, by the proof of \cite[Theorem 2.3]{Mi71}, the map $f : (i_k)_\#(M) \to (I_k)_{i_k}$ which sends $\lambda_k \otimes (x_1,x_2) \mapsto \lambda_k \cdot x_k$ for $\lambda_k \in \l_k$ is a left $\l_k$-module isomorphism. This is also a right $\l$-module isomorphism since 
\begin{align*} f((\lambda_k \otimes (x_1,x_2)) \cdot \lambda) &= f(\lambda_k \otimes (x_1 \cdot i_1(\lambda) ,x_2 \cdot i_2(\lambda))) \\
&= \lambda_k \cdot x_k \cdot i_k(\lambda) = f(\lambda_k \otimes (x_1,x_2)) \cdot i_k(\lambda)\end{align*}
and so $(i_k)_\#(I) \cong (i_k)_\#(M) \cong (I_k)_{i_k}$ are bimodule isomorphic, as required.
\end{proof}

\subsection{Proof of \cref{thm:canc-general}} \label{subsection:proof}

By \cref{thm:coset=K_R} and \cref{lemma:eichler}, it suffices to show that 
\[ |\Cls_{\mathcal{R}}^{-1}(\widetilde{P}_1,P_2)| = K_{\mathcal{R}}\] 
for all $\widetilde{P}_1$. By \cref{lemma:eichler} (ii), there is a bijection $\Cls_{\mathcal{R}}^{-1}(\widetilde{P}_1,P_2) \cong \Cls_{\mathcal{R}}^{-1}(\widetilde{P}_1,\l_2)$ and, by \cref{lemma:unit-rep-double-cosets}, we have that $|\Cls_{\mathcal{R}}^{-1}(\l_1,\l_2)| = K_{\mathcal{R}}$. Hence it suffices to show that, for all $\widetilde{P}_1$, 
there is a bijection $\Cls_{\mathcal{R}}^{-1}(\widetilde{P}_1,\l_2) \cong \Cls_{\mathcal{R}}^{-1}(\l_1,\l_2)$.

By assumption, there exists a two-sided ideal $I_1 \subseteq \l_1$ such that $I_1 \cong \wt P_1$ as left $\l_1$-modules and such that $(I_1)_p \cong (\l_1)_p$ are isomorphic as bimodules for all primes $p \mid |\widebar{\l}|$. By \cref{lemma:lifting-two-sided-ideals}, there exists a two-sided ideal $I \subseteq \l$ such that $I \in \Cls(\l)$ and $(i_1)_\#(I) \cong (I_1)_{i_1}$ as $(\l_1,\l)$-bimodules and $(i_2)_\#(I) \cong (\l_2)_{i_2}$ as $(\l_2,\l)$-bimodules. By \cref{lemma:two-sided-induces-bijection}, this induces a bijection
\[ I \otimes_{\l} - : \Cls_{\mathcal{R}}^{-1}(\l_1,\l_2) \to \Cls_{\mathcal{R}}^{-1}(I_1,\l_2),\]
and so there a bijections $\Cls_{\mathcal{R}}^{-1}(\wt P_1,\l_2) \cong \Cls_{\mathcal{R}}^{-1}(\l_1,\l_2)$, as required.

\section{Application to integral group rings}
\label{s:specialisation-to-ZG}

The aim of this section will be to specialise \cref{thm:canc-general} to the case of integral group rings $\Z G$. 
We will begin by deducing \cref{thm:main-group-rings} from \cref{thm:canc-general}. We will then prove \cref{thm:main-cancellation} by combining this with an additional cancellation theorem of R. G. Swan which is given in \cref{thm:exc-swan}. 

\subsection{Proof of \cref{thm:main-group-rings}}

We will begin by establishing the following. This was noted in \cite[Section 2]{Ni18} and follows directly from \cite[Example 42.3]{CR87}.

\begin{lemma} \label{lemma:G-H-Square}
Let $G$ be a finite group with a quotient $H = G/N$. Then there is a Milnor square
\[
\mathcal{R}_{G,H} = 
\begin{tikzcd}
  \Z G \arrow[r] \arrow[d] & \l \arrow[d] \\
  \Z H \arrow[r] & (\Z/n \Z)[H]
\end{tikzcd}
\]
where $\l = \Z G/ \Sigma_N$, $\Sigma_N = \sum_{g \in N} g \in \Z G$ and $n = |N|$. Furthermore, $(G,H)$ satisfies the relative Eichler condition if and only if $\l$ satisfies the Eichler condition.
\end{lemma}

The statement of \cref{thm:main-group-rings} now follows immediately from \cref{thm:canc-general}. In particular, the condition that $\l_2$ satisfies the Eichler condition is fulfilled since $H$ is an Eichler quotient of $G$.

\subsection{Proof of \cref{thm:main-cancellation} for quaternionic quotients}
\label{subsection:Q4n=two-sided}

We will now prove the following.

\begin{thm} \label{prop:canc-1}
Let $G$ be a finite group which has an Eichler quotient $H$ of the form
\[ Q_8, \, Q_{12}, \, Q_{16}, \, Q_{20}.\]
Then $\Z G$ has projective cancellation.	 
\end{thm}

We will show that conditions (i), (ii) of \cref{thm:main-group-rings} are satisfied for $H = Q_8$, $Q_{12}$, $Q_{16}$ or $Q_{20}$. In particular, for each of these groups, we will show:
\begin{clist}{(i)}
\item The map $\Z H^\times \to K_1(\Z H)$ is surjective
\item Every finitely generated projective $\Z H$-module is left isomorphic to an ideal in $\Z H$ which is generated by central elements.
\end{clist}

By \cite[Theorems 7.15-7.18]{MOV83} and the fact that $|C(\Z[\zeta_3])| = |C(\Z[\zeta_5])| = 1$, we have that $K_1(\Z H)$ is represented by units for $H \in\{Q_8,Q_{12},Q_{16},Q_{20}\}$. Thus $\mathcal{R}_{G,H}$ satisfies conditions (i). 

\begin{remark}
It also follows from \cite[Theorems 7.15-7.18]{MOV83} that condition (i) fails for $H=Q_{116}$.
\end{remark}

We say that a locally free module $P \in \Cls(\l)$ is \textit{represented by a two-sided ideal $I \subseteq \l$} if $P \cong I$ as left $\l$-modules.
It remains to prove the following, i.e. that $\mathcal{R}_{G,H}$ satisfies condition (ii).

\begin{prop} \label{thm:q8-20=two-sided}
For $2 \le n \le 5$, every $P \in \Cls(\Z Q_{4n})$ is represented by a two-sided ideal $I \subseteq \Z Q_{4n}$ which is generated by central elements.
\end{prop}

We begin by discussing two families of two-sided ideals which will suffice to represent all projective modules $P \in \Cls(\Z Q_{4n})$ in the case $2 \le n \le 5$. Note that, by \cite[Theorem III]{Sw83}, we have that $|\Cls(\Z Q_{4n})| = 2$ for $2 \le n \le 5$.

\subsubsection{Swan modules}

Let $G$ be a finite group, let $N = \sum_{g \in G} g$ denote the group norm and let $r \in \Z$ with $(r,|G|)=1$. The two-sided ideal $(N,r) \subseteq \Z G$ is projective as a left $\Z G$-module and is known as a \textit{Swan module}. If $r \equiv s$ mod $|G|$, then $(N,r) \cong (N,s)$ \cite[Lemma 6.1]{Sw60-I} and so we often write $r \in (\Z/|G|)^\times$. Note that $N, r \in Z(\Z G)$ and so $(N,r)$ is generated by central elements.

By \cite[Theorem VI]{Sw83}, we have that $[(N,3)] \ne 0 \in C(\Z Q_{2^n})$ for $n \ge 3$ where $(N,3)$ is a Swan module. Since $|\Cls(\Z Q_{2^n})| = 2$ for $n =3, 4$, this implies that:
\[ \Cls(\Z Q_8) = \{ \Z Q_8, (N,3)\}, \quad \Cls(\Z Q_{16}) = \{\Z Q_{16}, (N,3)\}.\]
This implies \cref{thm:q8-20=two-sided} for the groups $Q_8$ and $Q_{16}$.

\subsubsection{Two-sided ideals of Beyl and Waller} 
In order to prove \cref{thm:q8-20=two-sided} for the groups $Q_{12}$ and $Q_{20}$, we will now consider a family of projective two-sided ideals in $\Z Q_{4n}$ which were first introduced by R. Beyl and N. Waller in \cite{BW05}.

For $n \ge 2$, define $P_{a,b} = (a+by,1+x) \subseteq \Z Q_{4n}$ for $a, b \in \Z$ such that $(a^2+b^2,2n)=1$ if $n$ is odd and $(a^2-b^2,2n)=1$ if $n$ is even. It follows from \cite[Proposition 2.1]{BW05} that $P_{a,b}$ is a two-sided ideal and is projective as a left $\Z Q_{4n}$-module. For $\alpha \in \Z Q_{4n}$, let $(\alpha) \subseteq \Z Q_{4n}$ denote the two-sided ideal generated by $\alpha$. If $n$ is odd, then there is a Milnor square
\[ \mathcal{R} = 
\begin{tikzcd}
	\Z Q_{4n}/(x^n+1) \ar[r,"i_2"] \ar[d,"i_1"] & \Z[\zeta_{2n},j] \ar[d,"j_2"]\\
	\Z[j] \ar[r,"j_1"] & (\Z/n)[j]
\end{tikzcd}
\quad
\begin{tikzcd}
	x,y \ar[r,mapsto] \ar[d,mapsto] & \zeta_{2n},j \ar[d,mapsto]\\
-1,j \ar[r,mapsto] & -1,j \end{tikzcd}
\]
where $\Z Q_{4n}/(x+1) \cong \Z[j]$ and $\Z Q_{4n} / (x^{n-1}-x^{n-2}+\cdots -1) \cong \Z[\zeta_{2n},j] \subseteq \H_{\R}$. If $n=p$ is an odd prime then, by Propositions \ref{prop:milnor-square} and \ref{prop:K_1-square}, we have the following commutative diagram
\[
\begin{tikzcd}
\Cls_{\mathcal{R}}^{-1}(\Z[\zeta_{2p},j],\Z[j]) \ar[d, leftrightarrow] \ar[r,"\text{$[\,\cdot \,]$}"] & C_{\mathcal{R}}^{-1}([\Z[\zeta_{2p},j]],[\Z[j]]) \ar[d,leftrightarrow] \\
\displaystyle \frac{\F_p[j]^\times}{\Z[j]^\times \times \Z[\zeta_p,j]^\times}
\ar[r,"\varphi"] & 
\displaystyle \frac{K_1(\F_p[j])}{K_1(\Z[j]) \times K_1(\Z[\zeta_{2p},j])}
\end{tikzcd}
\]
where the vertical maps are one-to-one correspondences and $\varphi$ is induced by  $\F_p[j]^\times \to K_1(\F_p[j])$.
It follows from \cite[Proposition 2.2]{BW05} that $P_{a,b} \in \Cls_{\mathcal{R}}^{-1}(\Z[\zeta_{2n},j],\Z[j])$ with corresponding element $[a+bj] \in \frac{\F_p[j]^\times}{\Z[j]^\times \times \Z[\zeta_p,j]^\times}$. 

This allows us to deduce the following, which is an extension of \cite[Theorem 3.11]{BW08} in the case where $n=p$ is an odd prime. 

\begin{lemma} \label{lemma:BW}
Let $p$ be an odd prime with $|C(\Z[\zeta_p])|$ odd and let $P_{a,b} = (a+by,1+x) \subseteq \Z Q_{4p}$ for $(a^2+b^2,4p)=1$. Then:
\begin{clist}{(i)}
\item $P_{a,b}$ is free if and only if $p \mid a$ or $p \mid b$
\item $P_{a,b}$ is stably free if and only if $a^2+b^2$ is a square mod $p$.
\end{clist}
\end{lemma}

\begin{proof}
By \cite[Lemma 7.5]{MOV83}, we have that $\Z[\zeta_p,j]^\times = \langle \Z[\zeta_p]^\times,j \rangle$. Furthermore, the map $\Z[\zeta_p]^\times \to \F_p[j]^\times$ sends $\zeta_p$ to $1$ and so has image $\F_p^\times$ since units of any length are achievable. This implies that:
\[ \frac{\F_p[j]^\times}{\Z[j]^\times \times \Z[\zeta_p,j]^\times} \cong \F_p[j]^\times/ \F_p^\times \cdot \langle j \rangle \]
and so $P_{a,b}$ is free if and only if $[a+bj] = 1 \in \F_p[j]^\times/ \F_p^\times \cdot \langle j \rangle$.

Since $\Z[j]$ is a Euclidean Domain, we have $K_1(\Z[j])=\Z[j]^\times = \{ \pm 1, \pm j \}$ and, since $\F_p[j]$ is a finite and hence semilocal ring, we have $K_1(\F_p[j]) \cong \F_p[j]^\times$.
It follows from \cite[Lemmas 7.5/7.6]{MOV83} that, if $|C(\Z[\zeta_p])|$ odd, then $\IM(K_1(\Z[\zeta_p,j])) = \langle \F_p^\times, \Ker(N)) \rangle$ where $N : \F_p[j]^\times \to \F_p^\times$, $x+yj \mapsto x^2+y^2$ is the norm on $\F_p[j]$. In particular, there is an isomorphism:
\[ N : \frac{K_1(\F_p[j])}{K_1(\Z[j]) \times K_1(\Z[\zeta_{2p},j])} \cong \frac{\F_p[j]^\times}{\langle \Z[j]^\times,  \F_p^\times, \Ker(N) \rangle} \to  \F_p^\times/ N(\F_p^\times) \cong  \F_p^\times/(\F_p^\times)^2. \] 

Hence the map $\varphi$ coincides by the map
\[ N : \F_p[j]^\times/ \F_p^\times \cdot \langle j \rangle \to \F_p^\times/(\F_p^\times)^2 \]
which is induced by $N : \F_p[j]^\times \to \F_p^\times$, $x+yj \mapsto x^2+y^2$. In particular, $P_{a,b}$ is stably free if and only if $[a^2+b^2] = 1 \in \F_p^\times (\F_p^\times)^2$. The result follows by evaluating these conditions.
\end{proof}

If $p=3$ or $5$ then, as noted above, we have that $|C(\Z[\zeta_p])| = 1$. In the case $p=3$, we have $(1^2+2^2,12)=1$ and $3 \nmid 1,2$ and, in the case $p =5$, we have $(1^2+4^2,20)=1$ and $5 \nmid 1,4$. Since $|\Cls(\Z Q_{4p})| = 2$ for $p=3,5$, we have:
\[\Cls(\Z Q_{12}) = \{\Z Q_{12},P_{1,2}\}, \quad \Cls(\Z Q_{20}) = \{\Z Q_{20},P_{1,4}\}.\]

We will now show that the $P_{a,b}$ are generated by central elements. Our strategy will be to introduce a new family a two-sided ideals which are generated by central elements and show that the $P_{a,b}$ can be expressed in this form.

If $s \in \Z$ is odd, then define
\[ z_s = (-1)^{\frac{s-1}{2}}(x^{-\frac{s-1}{2}} - x^{-\frac{s-3}{2}} + \cdots + x^{\frac{s-1}{2}}), \quad \wt N = 1 - x + x^2- \cdots - x^{2n-1}.\] 
If $t \in \Z$, then define $\alpha_{s,t} = z_s + t \cdot \wt N y \in \Z Q_{4n}$.

\begin{lemma} \label{lemma:alpha=central}
Let $n \ge 2$ and let $r,s,t \in \Z$ where $(s,2n) = (r,2n)=1$. Then $(\alpha_{s,t}, r) \subseteq \Z Q_{4n}$ is projective as a $\Z Q_{4n}$-module and $\alpha_{s,t} \in Z(\Z Q_{4n})$.
\end{lemma}

\begin{proof}
It is easy to see that $z_s$, $\wt N \in Z(\Z Q_{4n})$ and $\wt N x = \wt N x^{-1}$. Hence we have $y \alpha_{s,t} = \alpha_{s,t} y$ and 
\[ x \alpha_{s,t} = z_s x +t(\wt N x) y = z_s x + t(\wt N x^{-1}) y = \alpha_{s,t} x,\] 
which implies that $\alpha_{s,t} \in Z(\Z Q_{4n})$.
Since $r \in (\alpha_{s,t}, r)$ and $(r,4n)=1$, we have that $(\alpha_{s,t}, r)$ is a projective $\Z Q_{4n}$-module by \cite[Proposition 7.1]{Sw60-II}.
\end{proof}

\begin{lemma} \label{lemma:Pab-centrally-generated}
Let $n \ge 2$ and let $a, b \in \Z$ be such that $(a^2-(-1)^nb^2,2n)=1$.
\begin{clist}{(i)}
\item If $r = (a^2-(-1)^nb^2)/\gcd(a,b)$, then there exists $a_0, b_0 \in \Z$ such that $a \equiv a_0$ $\text{\normalfont mod } r$, $b \equiv b_0$ $\text{\normalfont mod } r$, $(a_0,2n)=1$ and $2n \mid b_0$
\item $P_{a,b} = (\alpha_{s,t},r) \subseteq \Z Q_{4n}$ where $s = a_0$ and $t = b_0/2n$.
\end{clist}	
In particular, $P_{a,b}$ is generated by central elements.
\end{lemma}

\begin{proof}
Since $(r,2n)=1$, there exists $x, y \in \Z$ such that $rx+2ny=1$. Then $a_0 = a+rx(1-a)$ and $b_0 = 2nyb$ have the required properties.

Now recall that $P_{a,b} = (a+by,1+x)$. If $d = \gcd(a,b)$, then $\frac{1}{d}(a-by) \in \Z Q_{4n}$ and so 
\[ r = \frac{1}{d}(a-by) \cdot (a+by) \in P_{a,b}.\] 
In particular, since $a \equiv a_0$ $\text{\normalfont mod } r$ and $b \equiv b_0$ $\text{mod } r$, this implies that 
\[P_{a,b} = (a+by,1+x,r) = (a_0+b_0y,1+x,r).\]

Let $s = a_0$ and $t = b_0/2n$. If $e : \Z Q_{4n} \to \Z Q_{4n}$ is the function which evaluates at $x = -1$, then
\[ e(\alpha_{s,t}) = e(z_s) + t e(\wt N) y = s+ t(2n)y = a_0+b_0y.\]
Since $\Z Q_{4n} (1+x) = (1+x) \Z Q_{4n}$, this implies that $a_0 + b_0 y + \Z Q_{4n} (1+x) = \alpha_{s,t} + \Z Q_{4n} (1+x)$ and so $P_{a,b} = (\alpha_{s,t},1+x,r)$.

Since $s=a_0$ has $(s,2n)=1$, we can let $\ell \ge 1$ be such that $\ell s \equiv 1$ $\text{mod } 2n$. Similarly to the proof of \cite[Lemma 1.3]{BW08}, we now define 
\[ \bar{z}_s = (-x)^{\frac{s-1}{2}}\sum_{i=0}^{\ell -1} (-x^s)^i\] 
so that $z_s \bar{z}_s \in 1+ \Z Q_{4n}\wt N$. Since $\bar{z}_s \alpha_{s,t} = \bar{z}_s z_s +  \wt N \bar{z}_s ty$ and $\wt N \in Z(\Z Q_{4n})$, it follows that $\bar{z}_s \alpha_{s,t} \in 1+ \Z Q_{4n}\wt N$. Since $\wt N \in Z(\Z Q_{4n})$. We can now left-multiply both sides by $1+x$ to get that $(1+x)\bar{z}_s \alpha_{s,t} = 1+x$, where the $\wt N$ term vanishes since $(1+x) \wt N = 0$. Hence $1+x \in (\alpha_{s,t})$ and so $P_{a,b} = (\alpha_{s,t},r)$. By \cref{lemma:alpha=central}, this implies that $P_{a,b}$ is generated by central elements.
\end{proof}

By \cref{lemma:Pab-centrally-generated}, this implies that $P_{1,2} \subseteq \Z Q_{12}$ and $P_{1,4} \subseteq \Z Q_{20}$ are generated by central elements. This completes the proof of \cref{thm:q8-20=two-sided}, and hence completes the proof of \cref{prop:canc-1}.

\begin{remark}
This argument can also be used to prove \cref{thm:main-cancellation} in the case $H = \widetilde{T}$.
\end{remark}

\subsection{Proof of \cref{thm:main-cancellation} for exceptional quotients}
\label{subsection:cancellation-Swan} \label{subsection:proof-of-A}

The following generalises \cite[Corollary 13.5, Theorem 13.7]{Sw83} which corresponds to the case $G=H$.

\begin{thm} \label{prop:canc-2}
Let $G$ be a finite group which has an Eichler quotient $H$ of the form
\[ \widetilde{O}, \quad \widetilde{T}^n \times \widetilde{I}^m \text{ for $n, m \ge 0$}.\]
Then $\Z G$ has projective cancellation.	
\end{thm}

Let $\l$ be a $\Z$-order in a semisimple separable $\Q$-algebra $A$ which is finite-dimensional over $\Q$. Then we can write $A \cong A_1 \times \cdots \times A_r \times B$, where the $A_i$ are totally definite quaternion algebras with centres $K_i$ and $B$ satisfies the Eichler condition, i.e. $m_{\H}(\R \otimes B)=0$. Let $\Gamma_\l$ be the projection of $\l$ onto $A_1 \times \cdots \times A_r$. Let $\mathcal{R} = \mathcal{R}(\l,A_1 \times \cdots \times A_r, B)$ denote the corresponding fibre square.

Suppose that $\Gamma_\l \subseteq A_1 \times \cdots \times A_r$ is a maximal order. Then $\Gamma_\l = \G_1 \times \cdots \times \G_n$ where $\G_i \subseteq A_i$ is a maximal order for $i=1, \cdots, r$, and $\widebar{\l} = \widebar{\l}_1 \times \cdots \times \widebar{\l}_r$ where $\widebar{\l}_i$ is a quotient of $\G_i$.

Recall that there is a finite extension $K/K_i$ for which $A_i \otimes K \cong M_n(K)$ where $n = [K:K_i]$. If $\varphi: A_i \otimes K \to M_n(K)$ is an isomorphism, the \textit{reduced norm} is the map
$\nu_i : A_i \to K_i$,  $\lambda \mapsto \det(\varphi(\lambda \otimes 1))$. It can be shown that $\nu$ is independent of the choice of $K$ and $\varphi$.
For an order $\Gamma_i \subseteq A_i$, this restricts to a map $\nu_i : \Gamma_i^\times \to \mathcal{O}_{K_i}^\times$.

\begin{thm}[\hspace{-1mm}\text{\cite[Theorem 13.1]{Sw83}}] \label{thm:exc-swan}
Let $\mathcal{R}$ be as above and suppose that $\Gamma_{\l}$ is a maximal order in $A_1 \times \cdots \times A_r$. For each $i=1, \cdots, r$ and maximal $\mathcal{O}_{K_i}$-order $\Gamma_i \subseteq A_i$, suppose that:
\begin{clist}{(i)}
\item $\nu_i(\Gamma_i^\times) = (\mathcal{O}_{K_i}^\times)^+$
\item There is at most one prime $p$ such that $(\widebar{\l}_i)_{(p)} = 0$ and $p$ is ramified in $A_i$. If $p$ exists, then $(\Gamma_i)_0^\times = \Ker(\nu_i : \Gamma_i^\times \to \mathcal{O}_{K_i}^\times)$ has a subgroup of order $p+1$.
\end{clist}
Then $\l$ has LFC if and only if $\Gamma_{\l}$ has LFC.	
\end{thm}

For a finite group $G$, we can write $\Q G \cong A_1 \times \cdots \times A_r \times B$ where the $A_i$ are totally definite quaternion algebras and $B$ satisfies the Eichler condition. As above, let $\Gamma_{\Z G}$ is the projection of $\Z G$ onto $A_1 \times \cdots \times A_r$ and let $r = r(G)$ denote the value of $r$ in the decomposition of $\Q G$ above.

The following was proven in \cite[Proposition 4.11]{Sw83} and \cite[p84]{Sw83}.

\begin{lemma} \label{lemma:exc-1}
If $G = \widetilde{T}$, $\widetilde{O}$ or $\widetilde{I}$, then $r(G)=1$ and $\Gamma_{\Z G}$ is a maximal order in $A_{\Z G}$. Furthermore:
\begin{clist}{(i)}
\item $A_{\Z \widetilde{T}}$ has centre $\Z$, is ramified only at $p=2$ and $(\Gamma_{\Z \widetilde{T}})_0^\times \cong \widetilde{T}$
\item $A_{\Z \widetilde{O}}$ has centre $\Z[\sqrt{2}]$, is finitely unramified and $(\Gamma_{\Z \widetilde{O}})_0^\times \cong \widetilde{O}$
\item $A_{\Z \widetilde{I}}$ has centre $\Z[\frac{1}{2}(1+\sqrt{5})]$, is finitely unramified and $(\Gamma_{\Z \widetilde{I}})_0^\times \cong \widetilde{I}$.
\end{clist}
\end{lemma}

By \cite[Lemmas 13.8, 13.9]{Sw83}, we have that $\Gamma_{\Z[\widetilde{T}^n \times \widetilde{I}^m]} \cong \Gamma_{\Z \widetilde{T}}^n \times \Gamma_{\Z \widetilde{I}}^m$ for $n, m \ge 0$. The following is the key to being able to extending from proving that $\Z H$ has PC where $H = \wt O$, $\wt T^n \times \wt I^m$ for some $n,m \ge 0$, to proving that $\Z G$ has PC where $G$ has an Eichler quotient $H$.

\begin{lemma} \label{lemma:exc-2}
Let $G$ have an Eichler quotient $H$. Then the projection map $\Q G \to \Q H$ induces an isomorphism $\Gamma_{\Z G} \cong \Gamma_{\Z H}$	.
\end{lemma}

\begin{proof}
Let $\Q G \cong A_G \times B_G$ and $\Q H \cong A_H \times B_H$ where the $A_G$, $A_H$ are products of totally definite quaternion algebras and $B_G$, $B_H$ satisfy the Eichler condition. By definition, $\G_{\Z G}$ and $\G_{\Z H}$ are the projections of $\Z G$ and $\Z H$ on to $A_G$ and $A_H$ respectively.

Since $H$ is an Eichler quotient of $G$, we have $\Q G \cong \Q H \times B'$ where $B'$ satisfies the Eichler condition. Hence $A_G \cong A_H$ and $B_G \cong B_H \times B'$, and these isomorphisms commute with the inclusion maps into $\Q G$ and $\Q H$. Consider the projection maps $\rho_G : \Q G \twoheadrightarrow A_G$ and $\rho_H : \Q G \twoheadrightarrow \Q H \twoheadrightarrow A_H$. Then $\rho_G(\Z G) \cong \rho_H(\Z G)$ since $A_G \cong A_H$ are isomorphic as subrings of $\Q G$. Since $\rho_G(\Z G) \cong \G_{\Z G}$ and $\rho_H(\Z G) \cong \G_{\Z H}$, the result follows.
\end{proof}

\begin{proof}[Proof of \cref{prop:canc-2}]
By Lemma \ref{lemma:exc-2} and the discussion above, $\Gamma_{\Z G}$ is of the form $\Gamma_{\Z \widetilde{O}}$ or $\Gamma_{\Z \widetilde{T}}^n \times \Gamma_{\Z \widetilde{I}}^m$ for some $n, m \ge 0$. In particular, $\Gamma_{\Z G}$ is a maximal order whose components are maximal orders in $A_{\Z \widetilde{T}}$, $A_{\Z \widetilde{O}}$ or $A_{\Z \widetilde{I}}$.
If $\Gamma = \Gamma_{\Z \widetilde{T}}$, $\Gamma_{\Z \widetilde{O}}$ or $\Gamma_{\Z \widetilde{I}}$, then \cite[p84]{Sw83} implies that $\Gamma$ has projective cancellation and $|C(\Gamma)|=1$. Hence $\Gamma_{\Z G}$ has projective cancellation also, since it is the product of rings with projective cancellation.
In order to show that $\Z G$ has projective cancellation, it suffices to show that the conditions (i), (ii) of Theorem \ref{thm:exc-swan} hold for maximal orders in $A_{\Z \widetilde{T}}$, $A_{\Z \widetilde{O}}$ or $A_{\Z \widetilde{I}}$.

Firstly note that, if $\Gamma = \Gamma_{\Z \widetilde{T}}$, $\Gamma_{\Z \widetilde{O}}$ or $\Gamma_{\Z \widetilde{I}}$ and $A = A_{\Z \widetilde{T}}$, $A_{\Z \widetilde{O}}$ or $A_{\Z \widetilde{I}}$ respectively, then $|C(\Gamma)|=1$ implies that every maximal order in $A$ is conjugate to $\Gamma$. In particular, it suffices to check (i), (ii) for $\Gamma$ only.
To show (i) holds, note that $((\mathcal{O}_K)^\times)^2 \subseteq \nu(\Gamma^\times) \subseteq ((\mathcal{O}_K)^\times)^+$ where $K$ is the centre of $A$. By Lemma \ref{lemma:exc-1}, we have that 
\[ K \in \{\Q, \Q(\sqrt{2}) = \Q(\zeta_{8}+\zeta_{8}^{-1}), \Q(\sqrt{5}) = \Q(\zeta_{10}+\zeta_{10}^{-1})\}.\] 
In each case, we have $C(\mathcal{O}_K)=1$ and so $(\mathcal{O}_K^\times)^+ = (\mathcal{O}_K^\times)^2$ by, for example, \cite[Corollary B24]{Sw83}. Hence $\nu(\Gamma^\times) = ((\mathcal{O}_K)^\times)^+$.
To show (ii) holds, there is nothing to check in the case $\Gamma = \Gamma_{\Z \widetilde{O}}$ or $\Gamma_{\Z \widetilde{I}}$ since $A$ is finitely unramified by Lemma \ref{lemma:exc-1}. If $\Gamma = \Gamma_{\Z \widetilde{T}}$, then $A$ is ramified only at $p=2$ and $(\Gamma_{\Z \widetilde{T}})_0^\times \cong \widetilde{T}$ contains an element of order $p+1=3$. Hence this condition is satified regardless of whether or not $(\widebar{\l}_i)_{(p)} = 0$ for $R = \Z G$.
\end{proof}

\begin{remark}
This argument can also be used to prove \cref{thm:main-cancellation} in the case $H = Q_{12}$ (see also the remark in the proof of \cite[Corollary 13.5]{Sw83}).
\end{remark}

By combining Theorems \ref{prop:canc-1} and \ref{prop:canc-2}, we have now completed the proof of \cref{thm:main-cancellation}.

\subsection{Projective cancellation for $\wt T \times C_2$} 

We will now prove the following result from the introduction. The proof will use that $\wt T \times C_2$ has SFC by Bley-Hofmann-Johnston \cite{BHJ24}. 

\begin{reptheorem}{thm:TxC2}
	$\wt T \times C_2$ has $\PC$.
\end{reptheorem}

\begin{proof}
Consider the Milnor square
\[
\mathcal{R} =
\begin{tikzcd}
	\Z[\wt T \times C_2] \ar[r] \ar[d] & \ar[d] \Z[\wt T] \\
	\Z[\wt T] \ar[r] & \F_2[\wt T]
\end{tikzcd}
\]
which is induced from \cref{lemma:G-H-Square} (i.e. it has the form $\mathcal{R}_{\wt T \times C_2,\wt T}$) and the ring isomorphism 
\[ \Z[\wt T \times C_2]/\Sigma_{C_2} \cong (\Z[\wt T])[C_2]/\Sigma_{C_2} \cong \Z[\wt T].\]

The pullback square above induces a map
\[ \Cls_{\mathcal{R}} : \Cls(\Z[\wt{T} \times C_2]) \to \Cls(\Z[\wt{T}]) \times \Cls(\Z[\wt{T}])\] 
with fibre $\Cls_{\mathcal{R}}^{-1}(P_1,P_2)$ over each $(P_1,P_2) \in \Cls(\Z[\wt{T}]) \times \Cls(\Z[\wt{T}])$.
Let 
\[ K_{\mathcal{R}} = \left|\frac{K_1(\F_2[\wt T])}{K_1(\Z[\wt T]) \times K_1(\Z[\wt T])}\right|\] 
be the constant associated to $\mathcal{R}$.
Recall from \cref{ss:LF-modules} that $\wt T \times C_2$ has PC if and only if $|\Cls_{\mathcal{R}}^{-1}(P_1,P_2)| = K_{\mathcal{R}}$ for all $P_1$, $P_2$, and $\wt T \times C_2$ has SFC if and only if $|\Cls_{\mathcal{R}}^{-1}(\Z[\wt T],\Z[\wt T])| = K_{\mathcal{R}}$.
It follows from \cite[Theorems I, III \& VI]{Sw83} that $\Cls(\Z[\wt{T}]) = \{\Z[\wt{T}], (N,3)\}$, i.e. each element is represented by an ideal generated by central elements.
By \cref{lemma:lifting-two-sided-ideals}, this implies that 
\[ \Cls_{\mathcal{R}}^{-1}(P_1,P_2) \cong \Cls_{\mathcal{R}}^{-1}(\Z[\wt T],\Z[\wt T])\]
for all $P_1$, $P_2$, and so $|\Cls_{\mathcal{R}}^{-1}(P_1,P_2)| = |\Cls_{\mathcal{R}}^{-1}(\Z[\wt T],\Z[\wt T])|$ for all $P_1$, $P_2$.
It follows that $\wt T \times C_2$ has PC if and only if it has SFC. 
Since $\wt T \times C_2$ has SFC by \cite{BHJ24}, it must have PC.
\end{proof}

We will also establish the following:

\begin{prop} \label{prop:TxC2-lifting}
$\wt T \times C_2$ has $\SFC$ lifting.		
\end{prop}

Before turning to the proof, we will first establish the following lemma.

\begin{lemma} \label{lemma:TxC2}
The map $\Z[\wt T \times C_2]^\times \to K_1(\Z[\wt T \times C_2])$ is surjective.
\end{lemma}

A finite group $G$ is \textit{hyperelementary} if, for some prime $p$, it is of the form $A \rtimes B$ where $B$ is a $p$-group and $A$ is a cyclic group of order prime to $p$. 
The \textit{Whitehead group} is $\Wh(G) := K_1(\Z G)/\pm G$. By Wall \cite{Wa74}, we have $\Wh(G) \cong SK_1(\Z G) \oplus \Z^{r_\R(G)-r_\Q(G)}$ where, for $\F$ a field, $r_\F(G)$ denotes the number of indecomposable factors in the Wedderburn decomposition for $\F G$ (see also \cite[p6]{Ol88}).

\begin{proof}
By hyperelementary induction \cite[Theorem 3]{JM80}, $\Z[\wt T \times C_2]^\times \to K_1(\Z[\wt T \times C_2])$ is surjective if and only if $\Z H^\times \to K_1(\Z H)$ is surjective for all hyperelementary subgroups $H \le \wt T \times C_2$. By the tables in GroupNames \cite{groupnames}, the hyperelementary subgroups of $\wt T \times C_2$ are:
\[ \mathcal{H} = \{C_1,\,C_2,\,C_3,\,C_4,\,C_6,\,C_2^2,\,C_2 \times C_4, \,C_2 \times C_6,\, Q_8, \,Q_8 \times C_2\}. \]

The abelian groups $H \in \mathcal{H}$ each have Sylow $p$-subgroups of the form $C_{p^n}$ or $C_p \times C_{p^n}$ for some $n \ge 0$ and so $SK_1(\Z H)=0$ by \cite[Theorem 14.2]{Ol88}. Since $H$ is abelian, this implies automatically that $\Z H^\times \to K_1(\Z H)$ is surjective (see, for example, \cite[Section 1]{MOV83}).
For $H=Q_8$, the map $\Z Q_8^\times \to K_1(\Z Q_8)$ is surjective by \cite[Theorems 7.15-7.18]{MOV83}.

For $H=Q_8 \times C_2$, we claim that $\Wh(Q_8 \times C_2)=0$. We have
\[ \R[Q_8 \times C_2] \cong \R Q_8^2 \cong \R^8 \oplus \H_{\R}^2, \quad \Q[Q_8 \times C_2] \cong \Q Q_8^2 \cong \Q^8 \oplus \H_{\Q}^2\]
and so $r_{\R}(Q_8 \times C_2)=r_{\Q}(Q_8 \times C_2)=10$. 
It follows from \cite[p17]{Ol88} that $SK_1(\Z[Q_8 \times C_2])=0$.
Hence $\Wh(Q_8 \times C_2)=0$ by Wall's formula, and so $\Z[Q_8 \times C_2]^\times \to K_1(\Z[Q_8 \times C_2])$ is surjective.
\end{proof}

\begin{proof}
Note that \cref{lemma:TxC2} implies that condition (i) of \cref{thm:main-group-rings} holds. Hence $\wt T \times C_2$ has SFC lifting by \cite[Theorem A]{Ni19} (or by examining the  proof of \cref{thm:main-group-rings}).
\end{proof}

The following question remains open.

\begin{question} \label{question:PC-lifting-TxC2}
Does $\wt T \times C_2$ have $\PC$ lifting?	
\end{question}


\part{Group theoretic approach to cancellation properties}
\label{p:Group-theory}


In this part, we will develop the group theory which, combined with the results in \cref{p:PC-lifting}, will lead to our main classification results. This was discussed in \cref{ss:group-theory}. In \cref{s:MNECs}, we introduce Eichler simple groups and minimal non-Eichler covers. In \cref{s:MNEC-algorithms}, we give algorithms which will facilitate their computation.

\section{Eichler simple groups and minimal non-Eichler covers} \label{s:MNECs}

In this section, we will discuss properties of Eichler simple groups and minimal non-Eichler covers.
In \cref{ss:relative-Eichler}, we will discuss the Eichler condition.
In \cref{ss:MNEC-groups}, we introduce minimal non-Eichler covers and establish their basic properties.
In \cref{ss:MNEC-graph}, we introduce Eichler simple groups and establish the Eichler Pushout Lemma (\cref{lemma:EPL}) and the Fundamental Lemma (which is \cref{lemma:fundamental-lemma} in the introduction).
As usual, $G$ and $H$ will always denote finite groups.

\subsection{The Eichler condition}
\label{ss:relative-Eichler}

Recall from the introduction that a pair of finite groups $(G,H)$ satisfies the \textit{relative Eichler condition} if $G$ has a quotient $H$ with $m_{\H}(G)=m_{\H}(H)$. This property is independent of the choice of surjection $f : G \twoheadrightarrow H$ since it is characterised by $m_{\H}(G)$ and $m_{\H}(H)$, which depend only on $G$ and $H$.
If $G$ has a quotient $H$, then $\R H$ is a summand of $\R G$ and so $m_{\H}(G) \ge m_{\H}(H)$.
The following is well known (see, for example, \cite[Proposition 1.3]{Ni18}).

\begin{prop} \label{prop:eichler-to-group-theory}
A finite group $G$ satisfies the Eichler condition if and only if $G$ has no quotient which is a binary polyhedral group.
\end{prop}

We will now state an analogous characterisation of when a pair $(G,H)$ satisfies the relative Eichler condition.
In \cite[Section 1]{Ni20b}, we defined 
\begin{align*}
\B(G) &= \{ f : G \twoheadrightarrow Q \mid \text{$Q$ is a binary polyhedral group}\}/\sim 
\end{align*}
where, if $f_1 : G \twoheadrightarrow Q_1$ and $f_2 : G \twoheadrightarrow Q_2$ for $Q_1$, $Q_2$ binary polyhedral groups, then $f_1 \sim f_2$ if $\ker(f_1)=\ker(f_2)$ are equal as subsets of $G$ (which implies that $Q_1 \cong Q_2$).
The following is \cite[Proposition 3.3]{Ni18}.

\begin{prop} \label{prop:relative-eichler-group}
Let $G$, $H$ be finite groups and let $f: G \twoheadrightarrow H$ be a surjective group homomorphism. Then $(G,H)$ satisfies the relative Eichler condition if and only if every $g \in \B(G)$ factors through $f$, i.e. $f^* : \B(H) \to \B(G)$, $\varphi \mapsto \varphi \circ f$ is bijective.
\end{prop}

We will next establish formulae for $m_{\H}(G)$ in terms of characters of complex representations.
This will be used in \cref{ss:implementation} to compute $m_{\H}(G)$ (see the remarks following \cref{alg:MNEC(G)-original}).
Let $V$ be a simple $\C G$-module, which can be equivalently defined as an irreducible complex representation $\rho_V : G \to \GL_n(\C)$, and let $\chi_V : G \to \C^{\times}$ denote its character. If $\rho : G \to \GL_n(\C)$ is given, we write $V_{\rho}$ and $\chi_{\rho}$ for its corresponding simple $\C G$-module and character.
The \textit{Frobenius-Schur indicator} is defined as
\[  \FS(\chi_V) = \frac{1}{|G|} \sum_{g \in G} \chi_V(g^2) \in \C^{\times}.  \]
It can be shown that $\FS(\chi_V) \in \{1,0,-1\}$. We say $\rho$ is \textit{real} if $\FS(\chi_V) = 1$,  \textit{complex} if $\FS(\chi_V) = 0$ and \textit{quaternionic} if $\FS(\chi_V) = -1$. See, for example, \cite[Section 13.2]{Se77} for further details.

The following is standard but we were not able to locate a suitable reference in the literature, so we include a proof below.

\begin{prop} \label{prop:m_H-formula}
Let $G$ be a finite group. Then $m_{\H}(G)$ is equal to the number of $2$-dimensional irreducible complex representations $\rho : G \to \GL_2(\C)$ such that $\FS(\chi_{\rho}) = -1$.	
\end{prop}

\begin{proof}
By the Artin-Wedderburn theorem, there is a one to one correspondence between:
\begin{clist}{(1)}
\item
Simple components of $\R G$ of the form $M_{n_i}(D_i)$ where $D_i \in \{\,\R,\C,\H\,\}$ is a real division algebra.
\item
Simple $\R G$-modules $V$ such that $\End_{\R G}(V) \cong D_i$ and, using this isomorphism to view $V$ as a $D_i$-module, $\dim_{D_i}(V) = n_i$.
\end{clist}
Hence $m_{\H}(G)$ is the number of simple $\R G$-modules $V$ such that $\End_{\R G}(V) \cong \H$ and $\dim_{\H}(V) = 1$.

Let $i : \R G \hookrightarrow \C G$ and let $i^*$ denote the restriction of scalars map.
By \cite[p108 \& Proposition 39]{Se77}, the simple $\R G$-modules $V$ such that $\End_{\R G}(V) \cong \H$ are precisely the $\R G$-modules of the form $i^*(V')$ where $V'$ is a simple $\C G$-module such that $\FS(\chi_{V'}) = -1$. Since $\dim_{\R}(V') = \dim_{\R}(i^*(V'))$, the condition that $\dim_{\H}(i^*(V'))$ is equivalent to $\dim_{\C}(V')=2$. The result follows.
\end{proof}

\subsection{Minimal non-Eichler covers}
\label{ss:MNEC-groups}

A finite group $G$ is a \textit{minimal non-Eichler cover} (MNEC) of $H$ if it is a non-Eichler cover and, for some quotient map $f: G \twoheadrightarrow H$, $f$ does not factor through a map $f' : G' \twoheadrightarrow H$ where $G' \ne G$ is another non-Eichler cover of $H$. Let $\MNEC(H)$ denote the class of minimal non-Eichler covers of $H$.

\begin{prop} \label{lemma:minimal-BPG}
$\MNEC(C_1) = \{\text{\,$Q_{2^n}$ for $n \ge 3$, $Q_{4p}$ for $p$ an odd prime, $\wt T$, $\wt O$, $\wt I$\,}\}$.
\end{prop}

\begin{proof}
If $G$ is a non-Eichler cover of $C_1$, then $G$ does not satisfy the Eichler condition. By \cref{prop:eichler-to-group-theory}, $G$ has a quotient which is a BPG. This implies that, if $G$ is minimal, then $G$ is a BPG which admits no other BPG as a proper quotient. 
For $n \ge 2$, we have $Q_{8n} \twoheadrightarrow Q_{2^m}$ for some $m \ge 3$. For $n \ge 3$ odd, we have $Q_{4n} \twoheadrightarrow Q_{4p}$ for $p \mid n$ an odd prime. Hence $G$ is one of the groups listed.

Conversely, each of these groups have no proper quotients which are BPGs. To see this, note that the quotients of $Q_{4n}$ are abelian or of the form $Q_{4m}$ and $D_{2m}$ for $m \ge 2$ and $m \mid n$ (see, for example, by \cite[Lemmas 6.5-6.6]{Ni19}). For $\wt T$, $\wt O$, $\wt I$, this follows from GroupNames \cite{groupnames}.
\end{proof}

Note that a group $G$ has a BPG quotient if and only if $G$ has a quotient in $\MNEC(C_1)$, i.e. the list of groups (i)-(iii) above. This is because, by definition, all BPGs have a quotient which is one of these groups. This generalises to the following, which was stated briefly in the introduction. 

\begin{prop} \label{prop:groups1}
Let $G$ be a finite group which has a quotient $H$. Then $G$ is $H$-Eichler if and only if $G$ has no quotient in $\MNEC(H)$.
\end{prop}

\begin{proof}
Suppose $G$ has a quotient $G' \in \MNEC(G)$. Since $G' \in \MNEC(H)$, $G'$ has a non-Eichler quotient $H$, and so $m_{\H}(G') > m_{\H}(H)$. Hence $m_{\H}(G) \ge m_{\H}(G') > m_{\H}(H)$ and so $G$ is not $H$-Eichler.
Conversely, suppose $f : G \twoheadrightarrow H$ is a quotient and $G$ is not $H$-Eichler. By definition, $f$ factors through some $f' : G' \twoheadrightarrow H$ where $G' \in \MNEC(H)$ and so $G$ has a quotient in $\MNEC(H)$.
\end{proof}

For a class of finite groups $S$, define $\MNEC(S)$ to be the class of groups $G$ such that $G \in \MNEC(H) \setminus S$ for some $H \in S$ and, if $G \twoheadrightarrow G'$ and $G' \in \MNEC(H') \setminus S$ for some $H' \in S$, then $G \cong G'$.
We also define $\wt \MNEC(S) = \left(\bigcup_{H \in S} \MNEC(H)\right) \, \setminus \, S$. For a class of groups $S$, define the \textit{quotient filter} $S^{\QF}$ to be the subclass of groups in $S$ with no proper quotients in $S$.
The following is then immediate from the definition.

\begin{prop} \label{prop:MNEC-basic-prop}
	Let $S$ be a class of groups. Then $\MNEC(S) = (\wt \MNEC(S))^{\QF}$.
\end{prop}

\begin{remark}
(i) 
The analogue of \cref{prop:groups1} does not hold for arbitrary sets of groups $S$.
For example, if $S = \{C_1,Q_{24}\}$, then $G = Q_{24}$ is $S$-Eichler but has a quotient $Q_8 \in \MNEC(S)$.

(ii) Whilst $S \cap \MNEC(S) = \emptyset$ by definition, it is possible that $S \cap \MNEC^n(S) \ne \emptyset$ for some $n \ge 2$. 
For example, if $S = \{C_1,Q_8 \times C_2\}$, then $Q_8 \times C_2 \in S \cap \MNEC^2(S)$.
\end{remark}

\subsection{Eichler simple groups} 
\label{ss:MNEC-graph}

Recall from the introduction that a finite group $G$ is \textit{Eichler simple} if it has no proper Eichler quotients. We let $\EE$ denote the class of Eichler simple groups, and say a subset $S \subseteq \EE$ is \textit{closed under quotients} if $G \twoheadrightarrow H$ for $G \in S$ and $H \in \EE$ implies $H \in S$.

The main result of this section is the following.

\begin{thm}
\label{thm:Eichler-simple-alt-def}
$\EE = \bigcup_{n \ge 0} \MNEC^n(C_1)$. That is, a finite group $G$ is Eichler simple if and only if $G \in \MNEC^n(C_1)$ for some $n \ge 0$.
\end{thm}

This allows us to view $\EE$ as a directed graph where there is an edge from $G$ to $H$ if $G \in \MNEC(H)$. This approach will be the basis for computing groups in $\EE$ (see \cref{fig:G_2}).

Before turning to the proof, we will establish a number of properties of Eichler simple groups.

\begin{prop} \label{prop:quotients<=>path}
Let $G\in \EE$ and let $H$ be a finite group. Then there is a quotient $f: G \twoheadrightarrow H$ if and only if $G \in \MNEC^n(H)$ for some $n \ge 0$.
\end{prop}

Taking the graph theoretic perspective on $\EE$, this implies that there is a quotient $G \twoheadrightarrow H$ for $G, H \in \EE$ if and only if there is a directed path in $\EE$ from $G$ to $H$.

\begin{proof}
Let $G \in \EE$. If $G \in \MNEC^n(H)$, then $G \twoheadrightarrow H$. Conversely, suppose $G$ has a quotient $H$. 
If $G \twoheadrightarrow H$ is Eichler, then $G \cong H \in \MNEC^0(H)$ since $G$ is Eichler simple. If not, then $G \twoheadrightarrow H_1$ for some $H_1 \in \MNEC(H)$ by the definition of minimal non-Eichler cover. 
If $G \twoheadrightarrow H_1$ is Eichler, then $G \cong H_1 \in \MNEC(H_1)$. If not, we can repeat this to obtain a sequence of quotients
\[ G \dhxrightarrow[]{} H_{n} \dhxrightarrow[]{\MNEC} \cdots \dhxrightarrow[]{\MNEC} H_1 \dhxrightarrow[]{\MNEC} H_0 = H \]
where, for each $0 \le i < n$, we have $H_{i+1} \in \MNEC(H_i)$. If there exists $n \ge 0$ for which $G \twoheadrightarrow H_n$ is an Eichler quotient, then $G \cong H_n \in \MNEC^n(H)$.
If no such $n$ exists, then the sequence $H_0, H_1, \ldots$ is infinite.
Since $H_{i+1} \twoheadrightarrow H_i$ is non-Eichler, it is proper and so $|H_{i+1}| > |H_i|$. In particular, $|H_i| \to \infty$ as $i \to \infty$. This is a contradiction since $G \twoheadrightarrow H_i$ for all $i$ and so $|H_i| \le |G|$ for all $i$.
\end{proof}

\begin{prop} \label{prop:MNEC-of-ES=ES}
Minimal non-Eichler covers of Eichler simple groups are Eichler simple, i.e. if $H \in \EE$ and $G \in \MNEC(H)$, then $G \in \EE$.	
\end{prop}

The proof of \cref{prop:MNEC-of-ES=ES} will be an application of the following general group theoretic lemma. This will also be used later in the proof of the Fundamental Lemma (\cref{lemma:fundamental-lemma}).

\begin{lemma}[Eichler Pushout Lemma] \label{lemma:EPL}
Let $f_1 : G \twoheadrightarrow H_1$ be a quotient and let $f_2 : G \twoheadrightarrow H_2$ be an Eichler quotient. Then there exists a finite group $\ol{H}$, an Eichler quotient $\ol{f}_1 : H_1 \twoheadrightarrow \ol{H}$ and a quotient $\ol{f}_2 : H_2 \twoheadrightarrow \ol{H}$ such that $\ol{f}_1 \circ f_1 = \ol{f}_2 \circ f_2$. 
i.e. the following diagram commutes:
\[
\begin{tikzcd}
	G \ar[r,twoheadrightarrow] \ar[d,twoheadrightarrow,"\text{\normalfont E}"'] & H_1 \ar[d,twoheadrightarrow,dashed,"\text{\normalfont E}"'] \\
	H_2 \ar[r,twoheadrightarrow,dashed] & \ol{H}
\end{tikzcd}
\]
where \normalfont{E} denotes an Eichler map.	 Furthermore, $H_2 \twoheadrightarrow \ol{H}$ is Eichler if and only if $G \twoheadrightarrow H_1$ is Eichler.
\end{lemma}

\begin{proof}
Let $N_1, N_2 \le G$ be normal subgroups such that $G/N_1 \cong H_1$ and $G/N_2 \cong H_2$. Let $\ol{H} = G/(N_1 \cdot N_2)$, let $\ol{f}_1$ and $\ol{f}_2$ be the natural quotient maps.

We now claim that $\ol{f}_1$ is Eichler. By \cref{prop:relative-eichler-group} it suffices to show that every binary polyhedral quotient $f: H_1 \twoheadrightarrow Q$ (i.e. a quotient with $Q$ binary polyhedral) factors through $\ol{f}_1$. Let $g_1 : H_1 \twoheadrightarrow Q$ be a binary polyhedral quotient. Then $g_1 \circ f_1 : G \twoheadrightarrow Q$ is a binary polyhedral quotient. Since $f_2 : G \twoheadrightarrow H_2$ is Eichler, \cref{prop:relative-eichler-group} implies that $g_1 \circ f_1$ factors through $f_2$, i.e. there exists a binary polyhedral quotient $g_2 : H_2 \twoheadrightarrow Q$ such that $g_1 \circ f_1 = g_2 \circ f_2$. 
Let $F := g_1 \circ f_1 : G \twoheadrightarrow Q$. Then $N_1 = \ker(f_1) \subseteq \ker(F)$ and, since $F= g_2 \circ f_2$ also, we have $N_2 = \ker(f_2) \subseteq \ker(F)$. This implies that $N_1 \cdot N_2 \subseteq \ker(F)$ and so $F$ factors through the quotient map $\ol{f}_1 \circ f_2 : G \twoheadrightarrow G/(N_1 \cdot N_2) =: \ol{H}$, i.e. there exists $q : \ol{H} \twoheadrightarrow Q$ such that $g_1 \circ f_1 = q \circ \ol{f}_1 \circ f_1$. Since $f_1$ is surjective, this implies that $g_1 = q \circ \ol{f}_1$. In particular, $g_1$ factors through $\ol{f}_1$ as required.

For the last part, note that $m_{\H}(G) = m_{\H}(H_2)$ and $m_{\H}(H_1)=m_{\H}(\ol{H})$. Hence $m_{\H}(G)=m_{\H}(H_1)$ if and only if $m_{\H}(H_2)=m_{\H}(\ol{H})$.
\end{proof}

\begin{proof}[Proof of \cref{prop:MNEC-of-ES=ES}]
For convenience, we will use the notation of \cref{lemma:EPL}.
Let $H_1 \in \EE$ and $G \in \MNEC(H_1)$. If $G \twoheadrightarrow H_2$ is an Eichler quotient, then there exists an Eichler quotient $H_1 \twoheadrightarrow \ol{H}$ and a quotient $H_2 \twoheadrightarrow \ol{H}$ such that there is a commutative diagram of quotients as in \cref{lemma:EPL}. Since $H_1 \in \EE$, the map $H_1 \twoheadrightarrow \ol{H}$ is an isomorphism. It follows that $G \twoheadrightarrow H$ factors through the quotient $H_1 \twoheadrightarrow H$ which is non-Eichler by the last part of \cref{lemma:EPL}. Since $G \in \MNEC(H)$, it follows that the map $G \twoheadrightarrow H_1$ is an isomorphism. Hence $G \in \EE$, as required.
\end{proof}

We can now deduce \cref{thm:Eichler-simple-alt-def} from Propositions \ref{prop:quotients<=>path} and \ref{prop:MNEC-of-ES=ES}.

\begin{proof}[Proof of \cref{thm:Eichler-simple-alt-def}]
If $G \in \EE$, then $G \twoheadrightarrow C_1$ and so \cref{prop:quotients<=>path} implies that $G \in \MNEC^n(C_1)$ for some $n \ge 0$. Conversely, if $G \in \MNEC^n(C_1)$ for some $n \ge 0$, then \cref{prop:MNEC-of-ES=ES} implies that $G \in \EE$. Hence $\EE = \bigcup_{n \ge 0} \MNEC^n(C_1)$.	
\end{proof}

We will next show the following. This implies that $\EE = \bigsqcup_{n \ge 0} \MNEC^n(C_1)$ is a disjoint union and, for all $N \ge 1$, the class $S = \bigcup_{n = 0}^N \MNEC^n(C_1) \le \EE$ is closed under quotients.

\begin{prop} \label{prop:MNEC^n(C_1)}
\begin{clist}{(i)}
\item
In $n \ne m$, then $\MNEC^n(C_1) \cap \MNEC^m(C_1)  = \emptyset$.
\item
If $G \in \MNEC^n(C_1)$ and $H \in \MNEC^m(C_1)$ and $G \twoheadrightarrow H$, then $n > m$ or $n=m$ and $G 
\cong H$.	
\end{clist}
\end{prop}

\begin{proof}
We will start by proving (ii).
If $n < m$, then there is a non-Eichler quotient $H \twoheadrightarrow H'$ for some $H' \in \MNEC^n(C_1)$. By composition, this implies that there is a proper quotient $G \twoheadrightarrow H'$ with $G, H' \in \MNEC^n(C_1) = \MNEC(S)$ where $S=\MNEC^{n-1}(C_1)$. This is a contradiction by the definition of $\MNEC(S)$. If $n = m$, the same argument implies that $G \cong H$.

To prove (i), suppose $G \in \MNEC^n(C_1) \cap \MNEC^m(C_1)$ for $n \ne m$. Since $G \twoheadrightarrow G$ and $n \ne m$, applying (ii) twice gives that $n > m$ and $m > n$ which is a contradiction.
\end{proof}

For the purposes of computations in \cref{p:Computations}, we will need to consider Eichler simple groups with no quotients in a given set. 
For sets of groups $S$, $B$, define
\begin{align*}
\EE_{B} &= \{ G \in \EE : \text{$G$ has no quotient in $B$} \} \\
\MNEC_B(S) &= \{ G \in \MNEC(S) : \text{$G$ has no quotient in $B$}\}.
\end{align*}
It follows from \cref{thm:Eichler-simple-alt-def} that $\EE_B = \bigcup_{n \ge 0} \MNEC_B^n(C_1)$. As before, we say a subset $S \subseteq \EE_B$ is \textit{closed under quotients} if $G \twoheadrightarrow H$ for $G \in S$ and $H \in \EE_B$ implies $H \in S$.

\subsection{The Fundamental Lemma}
\label{ss:proof-fundamental-lemma}

The following is the main result of this section. For $S$ a specially chosen collection of Eichler simple groups, this makes it possible to determine the class of $S$-Eichler groups in terms of minimal non-Eichler covers.

\begin{replemma}{lemma:fundamental-lemma}[Fundamental Lemma]
Let $S \subseteq \EE$ be a class of Eichler simple groups which is closed under quotients. 
Then a finite group $G$ is $S$-Eichler if and only if $G$ has no quotient in $\MNEC(S)$.	
\end{replemma}

\begin{proof}
($\Rightarrow$):
Suppose $G$ is $S$-Eichler and has a quotient in $\MNEC(S)$. So there is an Eichler quotient $G \twoheadrightarrow H$ for $H \in S \subseteq \EE$ and a quotient $G \twoheadrightarrow H'$ for $H' \in \MNEC(S)$. By \cref{prop:MNEC-of-ES=ES}, we have $H' \in \EE$. By the Eichler Pushout Lemma (\cref{lemma:EPL}), there exists a finite group $\ol{H}$ and a commutative diagram
\[
\begin{tikzcd}
	G \ar[r,twoheadrightarrow] \ar[d,twoheadrightarrow,"\text{\normalfont E}"'] & H' \ar[d,twoheadrightarrow,"\text{\normalfont E}"'] \\
	H \ar[r,twoheadrightarrow] & \ol{H}
\end{tikzcd}
\]
where \normalfont{E} denotes an Eichler map.	 Since $H' \in \EE$ is Eichler simple, the map $H' \twoheadrightarrow \ol{H}$ is an isomorphism and so there is a quotient map $H \twoheadrightarrow H'$. Since $H \in S$ and $S$ is closed under quotients, this implies that $H' \in S$. This is a contradiction since $H' \in \MNEC(S)$ and $S \cap \MNEC(S) = \emptyset$.

($\Leftarrow$):
Suppose $G$ has no quotient in $\MNEC(S)$ but is not $S$-Eichler. Let $H \in S$ be a group of maximal order such that $G \twoheadrightarrow H$. Since $S$ is closed under quotients, we have $C_1 \in S$ and so such a group $H$ exists. Since $G$ is not $S$-Eichler, $G$ is not $H$-Eichler. By \cref{prop:groups1}, $G$ has a quotient $H' \in \MNEC(H)$.  Since $H' \twoheadrightarrow H$ is a non-Eichler cover, we have that $|H'| > |H|$ and so $H' \not \in S$ since $G$ has no quotients in $S$ or order $>|H|$. This implies that $H'$ has a quotient in $\MNEC(S)$ and so $G$ has a quotient in $\MNEC(S)$, which is a contradiction.
\end{proof}

The following version for $\EE_B$ will is useful from the perspective of computations. This follows immediately from \cref{lemma:fundamental-lemma} since, if $G$ has no quotient in $B$, then $G$ has a quotient in $\MNEC(S)$ if and only if $G$ has a quotient in $\MNEC_B(S)$.	

\begin{corollary}
	\label{cor:boundary-lemma-general}
Let $B \subseteq \EE$ be a subset and let $S \le \EE_B$ be a class which is closed under quotients. If $G$ is a finite group which has no quotients in $B$, then $G$ is $S$-Eichler if and only if $G$ has no quotient in $\MNEC_B(S)$.
\end{corollary}

\section{Algorithms for computing Eichler simple groups}
\label{s:MNEC-algorithms}

We will now develop methods for computing Eichler simple groups, with a view to explicit computations which we implement in GAP \cite{GAP4} and Magma \cite{magma}. In \cref{ss:MNEC-pullback}, we describe the groups in $\MNEC(H)$ in terms of pullbacks. This is the basis for algorithms for computing $\MNEC(H)$ given in \cref{ss:implementation}.
As usual, $G$ and $H$ will always denote finite groups.

\subsection{MNECs as pullbacks} \label{ss:MNEC-pullback}

We will now prove the following. This gives us a method for computing $\MNEC(H)$ for a group $H$.
Define $\wh{\Quot}(G,H) := \Aut(H) \setminus \Quot(G,H) / \Aut(G)$ where $\Aut(G)$ and $\Aut(H)$ act on $\Quot(G,H)$ by pre-composition and post-composition respectively.

\begin{prop} \label{prop:groups2}
Let $H$ be a finite group and let $G \in \MNEC(H)$.
Then there exists a binary polyhedral group $Q$, a finite group $A$ and quotient maps $f_1: H \twoheadrightarrow A$, $f_2 : Q \twoheadrightarrow A$ such that: 
\[ G \cong H \times_{f_1,f_2} Q := \{ (x_1,x_2) \in H \times Q : f_1(x_1) = f_2(x_2) \in A \}. \]

Furthermore, the isomorphism class of $H \times_{f_1,f_2} Q$ only depends on the classes $\ol{f_1} \in \wh{\Quot}(H,A)$ and $\ol{f_2} \in \wh{\Quot}(Q,A)$.
\end{prop}

The proof of \cref{prop:groups2} will consist of the following chain of lemmas.
For a finite group $G$, we say that two quotients $f_1 : G \twoheadrightarrow H_1$ and $f_2 : G \twoheadrightarrow H_2$ are \textit{disjoint} if $\ker(f_1) \cap \ker(f_2) = \{1\}$.

\begin{lemma} \label{lemma:classify-disjoint-quotients} 
Let $G$ be a finite group with disjoint quotients $f_1: G \twoheadrightarrow H_1$ and $f_2: G \twoheadrightarrow H_2$. Let  $N_i = \ker(f_i)$. Then:
\begin{clist}{(i)}
\item $f_1(N_2) \unlhd H_1$ and $f_2(N_1) \unlhd H_2$ are normal subgroups
\item $f_1(N_2) \cong N_2$ and $f_2(N_1) \cong N_1$ as groups
\item $G/(N_1 \times N_2) \cong H_1/f_1(N_2) \cong H_2/f_2(N_1)$
\item
There is a pullback diagram of groups:
\[
\begin{tikzcd}[column sep={6em,between origins}]
	& N_1 \ar[d] \ar[r,equals] & N_1 \ar[d] \\
	N_2 \ar[d,equals] \ar[r] & G \ar[r] \ar[d] & H_2 \ar[d] \\
	N_2 \ar[r] & H_1 \ar[r] & G/(N_1 \times N_2)
\end{tikzcd}
\]
\end{clist}
\end{lemma}

\begin{proof}
(i) This follows since $f_1$, $f_2$ are surjective.

(ii) By the second isomorphism theorem, we have that
\[ f_1(N_2) = (N_1 \cdot N_2)/N_1 \cong  N_2 / (N_1 \cap N_2) = N_2 \]
where the last equality is since $N_1 \cap N_2 = \{1\}$. Similarly for $f_2(N_1)$.

(iii)
By symmetry, it suffices to prove that $G/(N_1 \times N_2) \cong H_1/f_1(N_2)$.
Since $N_1 \cap N_2 = \{1\}$, we have $N_1 \times N_2 = N_1 \cdot N_2 \, \unlhd \, G$. 
By the third isomorphism theorem, this implies that
\[ G/(N_1 \times N_2) = G/(N_1 \cdot N_2) \cong (G/N_1)/((N_1 \cdot N_2)/N_1) =  H_1/f_1(N_2). \]

(iv) The diagram is formed by taking $N_1 \to G$ to be the inclusion map, $N_1 \to H_2$ to be the composition of $N_1 \cong f_2(N_1)$ and inclusion, and $H_1 \to G/(N_1 \times N_2)$ to be the quotient map. Similarly for $N_2$. This is a pullback diagram by (i), (ii) and (iii).
\end{proof}

\begin{lemma} \label{lemma:pullbacks-of-groups}
Let $G$, $H_1$, $H_2$ be finite groups. Then there exists disjoint quotients $G \twoheadrightarrow H_1$ and $G \twoheadrightarrow H_2$ if and only if there exists a group $A$ and quotients $f_1 : H_1 \twoheadrightarrow A$, $f_2 : H_2 \twoheadrightarrow A$ such that
\[ G \cong H_1 \times_{f_1,f_2} H_2 :=  \{ (x_1,x_2) \in H_1 \times H_2 : f_1(x_1) = f_2(x_2) \in A \}.\]
\end{lemma}

\begin{proof}
($\Rightarrow$) This follows directly from \cref{lemma:classify-disjoint-quotients} (iv) and the definition of pullback square of groups. In particular, take $A = G/(N_1 \times N_2)$ with $f_1$, $f_2$ the maps to $A$ in the pullback square.

($\Leftarrow$) Without loss of generality, take $G = \{ (x_1,x_2) \in H_1 \times H_2 : g_1(x_1) = g_2(x_2) \}$ for a finite group $A$ and quotients $g_1 : H_1 \twoheadrightarrow A$, $g_2 : H_2 \twoheadrightarrow A$.
Let $f_1 : G \to H_1$, $(x_1,x_2) \mapsto x_1$. This is surjective since for all $x_1 \in H_1$ there exists $x_2 \in H_2$ such that $g_2(x_2)=g_1(x_1)$ by the surjectivity of $g_2$. Similarly $f_2 : G \twoheadrightarrow H_2$, $(x_1,x_2) \mapsto x_2$ is surjective.
Now $f_1$, $f_2$ are clearly disjoint since $(x_1,x_2) \in \ker(f_1) \cap \ker(f_2)$ implies $x_1=1$ and $x_2=1$.
\end{proof}

The following observation is straightforward but recorded here for convenience.

\begin{lemma} \label{lemma:disjoint=doesnt-factor}
Let $G$ be a finite group and let $f_1 : G \twoheadrightarrow H_1$,  $f_2 : G \twoheadrightarrow H_2$ be disjoint quotients. 
Then $f_1$ factors through $f_2$ if and only if $f_2$ is not proper.	
In particular, $f_1$ and $f_2$ are proper if and only if $f_1$ does not factor through $f_2$ and $f_2$ does not factor through $f_1$.
\end{lemma}

\begin{lemma} \label{lemma:min-non-Eichler}
Let $H$ be a finite group, let $G \in \MNEC(H)$ and let $f : G \twoheadrightarrow H$ be a quotient map. Then there exists $g \in \B(G)$ such that $f$, $g$ are disjoint quotients.
\end{lemma}

\begin{proof}
Since $f : G \to H$ is a non-Eichler, it follows from \cref{prop:relative-eichler-group} that there exists $g \in \B(G)$ which does not factor through $f$. Let $g: G \twoheadrightarrow Q$ for $Q$ a binary polyhedral group.
Now define $\ol G = G / (\ker(f) \cap \ker(g))$. This has quotients $\ol f : \ol G \twoheadrightarrow G$ and $\ol g : \ol G \twoheadrightarrow Q$ which are induced by $f$ and $g$ respectively. Furthermore, $f$ and $g$ factor through $\ol f$ and $\ol g$ respectively. If $\ol f: G \twoheadrightarrow \ol G$ is Eichler then, by the other direction of \cref{prop:relative-eichler-group}, $\ol g$ must factor through $\ol f$. This now implies that $g$ factors through $f$, which is a contradiction.
Hence $\ol f$ is non-Eichler and so $\ol G$ is another non-Eichler cover of $H$. Since $G$ is a minimal non-Eichler cover, it follows that the map $G \twoheadrightarrow \ol G$ is an isomorphism. This implies that $\ker(f) \cap \ker(g) = \{1\}$ and so $f$, $g$ are disjoint quotients. 
\end{proof}

\begin{lemma} \label{lemma:Quot-coset}
Let $H_1, H_2, A$ be finite groups. Suppose $f_1,f_1' \in \Quot(H_1,A)$ and $f_2,f_2' \in \Quot(H_2,A)$. 
\begin{clist}{(i)}
\item
If $[f_1] = [f_1'] \in \Quot(H_1,A) / \Aut(H_1)$ and $[f_2] = [f_2'] \in \Quot(H_2,A) / \Aut(H_2)$, then \\ $H_1 \times_{f_1,f_2} H_2 \cong H_1 \times_{f_1',f_2'} H_2$.
\item
If $\alpha \in \Aut(A)$, then $H_1 \times_{\alpha \circ f_1,f_2} H_2 \cong H_1 \times_{f_1,\alpha^{-1} \circ f_2} H_2$.
\end{clist}
\end{lemma}

\begin{proof}
(i) For $i=1,2$, let $f_i' = f_i \circ \theta_i$ for $\theta_i \in \Aut(H_i)$. Then consider the map
\[ F : H_1 \times_{f_1,f_2} H_2 \to H_1 \times_{f_1',f_2'} H_2, \quad (x_1, x_2) \mapsto (\theta_1(x_1),\theta_2(x_2)). \]
This is well defined since, if $(x_1, x_2) \in H_1 \times_{f_1,f_2} H_2$, then $f_1(x_1)=f_2(x_2)$ and so $f_1'(\theta_1(x_1)) = f_2'(\theta_2(x_2))$. This implies that $(\theta_1(x_1),\theta_2(x_2)) \in H_1 \times_{f_1',f_2'} H_2$. Then check that $F$ is an isomorphism.

(ii) For all $x_1 \in H_1$ and $x_2 \in H_2$, $\alpha(f_1(x_1)) = f_2(x_2)$ if and only if $f_1(x_1) = \alpha^{-1}(f_2(x_2))$. It follows that $H_1 \times_{\alpha \circ f_1,f_2} H_2$ and $H_1 \times_{f_1,\alpha^{-1} \circ f_2} H_2$ coincide as sets, and so are isomorphic.
\end{proof}

\begin{proof}[Proof of \cref{prop:groups2}]
Let $H$ be a finite group and let $G \in \MNEC(H)$. Then there is a quotient $f : G \twoheadrightarrow H$. By \cref{lemma:min-non-Eichler}, there exists a binary polyhedral group $Q$ and a quotient $g : G \twoheadrightarrow Q$ such that $f$ and $g$ are disjoint. By \cref{lemma:pullbacks-of-groups}, this implies that $G \cong H \times_{f_1,f_2} Q$ for some quotients $f_1 : H \twoheadrightarrow A$ and $f_2 : Q \twoheadrightarrow A$, as required. The last part follows directly from \cref{lemma:Quot-coset}.
\end{proof}

\subsection{Algorithms} 
\label{ss:implementation}

We will now give our main algorithm which computes $\MNEC_B(H)$ for a given group $H$ and a specially determined class of groups $B$.
The algorithm takes as input the normal subgroups and quotients of the groups $Q_8$,  $Q_{12}$,  $Q_{16}$,  $Q_{20}$, $\wt T$, $\wt O$ and $\wt I$. For convenience later on, we list them in \cref{fig:BPG-quotients}. This follows, for example, from the tables in GroupNames \cite{groupnames}.

\begin{figure}[h]
\begin{center}
\begin{tabular}{|c|c|c|}
\hline
Group & Normal subgroups & Quotients \\ \hline
$Q_8$ & $C_1$, $C_2$, $C_4$, $Q_8$ & $C_1$, $C_2$, $C_2^2$, $Q_8$ \\ \hline
$Q_{12}$ & $C_1$, $C_2$, $C_3$, $C_6$, $Q_{12}$ & $C_1$, $C_2$, $C_4$, $D_6$, $Q_{12}$ \\ \hline
$Q_{16}$ & $C_1$, $C_2$, $C_4$, $C_8$, $Q_8$, $Q_{16}$ & $C_1$, $C_2$, $C_2^2$, $D_8$, $Q_{16}$ \\ \hline
$Q_{20}$ & $C_1$, $C_2$, $C_5$, $C_{10}$, $Q_{20}$ & $C_1$, $C_2$, $C_4$, $D_{10}$, $Q_{20}$ \\ \hline
$\widetilde{T}$ & $C_1$, $C_2$, $Q_8$, $\wt T$ & $C_1$, $C_3$, $A_4$, $\wt T$ \\ \hline
$\widetilde{O}$ & $C_1$, $C_2$, $Q_8$, $\wt T$, $\wt O$ & $C_1$, $C_2$, $D_6$, $S_4$, $\wt O$ \\ \hline
$\widetilde{I}$ & $C_1$, $C_2$, $\wt I$ & $C_1$, $A_5$, $\wt I$ \\ \hline
\end{tabular}
\end{center}
\caption{Normal subgroups and quotients of small binary polyhedral groups} \label{fig:BPG-quotients}
\end{figure}

Let $B_{\BPG} = \{Q_{36},Q_{60},Q_{100}\} \sqcup \{Q_{4p} : \text{$p \ge 7$}\} \sqcup \{Q_{8p} : \text{$p \ge 3$}\} \sqcup \{Q_{16p} : \text{$p \ge 3$} \} \sqcup \{Q_{2^n} : n \ge 4\}$, where $p$ is prime and $n$ is an integer.
This is the smallest class of groups for which a finite group $G$ has a quotient in $B_{\BPG}$ if and only if $G$ has a quotient of the form $Q_{4n}$ for some $n \ge 6$. Since $Q_{4n} \in \EE$ for all $n \ge 2$, we have $B_{\BPG} \subseteq \EE$.
Let $B$ be such that $B_{\BPG} \subseteq B \subseteq \EE$.

\begin{algorithm} \label{alg:MNEC(G)-original}
Input a finite group $H$ and a class $B$ such that $B_{\BPG} \subseteq B \subseteq \EE$.
Output $\MNEC_B(H)$.
Define 
\[ \BPG = \{ Q_8, Q_{12}, Q_{16}, Q_{20}, \widetilde{T}, \widetilde{O}, \widetilde{I}\}, \quad \BPG^{N} = \{C_1, C_2, C_3, C_4, C_5, C_6, C_8, Q_8, C_{10}, \widetilde{T}\}.\]
By \cref{fig:BPG-quotients}, $\BPG \cup \BPG^N$ (resp. $\BPG^N$) is the set of normal subgroups (resp. proper normal subgroups) of the groups in $\BPG$.
Note that $\Quot(Q)$ for each $Q \in \BPG$ has been computed already in \cref{fig:BPG-quotients}.
	The steps are as follows:
\begin{clist}{(i)}
\item
For each $N \in \BPG^N$, compute the extensions of $H$ by $N$, i.e. the groups $G$ such that there is an extension $1 \to N \to G \to H \to 1$. Let $\mathcal{S}_N$ denote this class of groups. For each $N \in \BPG \setminus \BPG^N = \{Q_{12},Q_{16},Q_{20},\wt O, \wt I\}$, let $\mathcal{S}_N = \{ H \times N\}$.

\item
Let $\mathcal{S} = \bigcup_{N} \mathcal{S}_{N}$ where $N \in \BPG \cup \BPG^N$.
Compute $m_{\H}(G)$ for $G \in \mathcal{S} \sqcup \{H\}$ using \cref{prop:m_H-formula}.
Remove all groups $G \in \mathcal{S}$ such that $m_{\H}(G) = m_{\H}(H)$.
Then remove all groups which have others as proper quotients. Then remove all groups which have quotients in $B$. This gives a class of groups $\mathcal{S}'$ which coincides with $\MNEC_B(H)$.
\end{clist}	
\end{algorithm}

\begin{remark}
We lose no generality in assuming that $B \supseteq \{Q_{4n} : n \ge 6\}$ rather than $B \supseteq B_{\BPG}$. However, this does make a slight difference to the speed of the algorithm in practice.
\end{remark}

\begin{proof}[Proof of correctness]
We claim that the algorithm terminates and has output $\mathcal{S}' = \MNEC_B(H)$. The algorithm terminates since, for fixed $H$ and $N$ a soluble group, there is a finite time algorithm for determining the groups $G$ such that there is an extension $1 \to N \to G \to H \to 1$ using group cohomology. 
The case where $N$ is soluble suffices since all the groups in $\BPG^N$ are soluble (see GroupNames \cite{groupnames}).
We use \cref{prop:m_H-formula} to compute $m_{\H}(G)$. This amounts to computing the character table of $G$ and the Frobenius-Schur indicators of each character, both of which can be done in finite time.

We now claim that the algorithm computes $\MNEC_B(H)$, i.e. that $\MNEC_B(H) = \mathcal{S}'$.
In order to prove this, let $\mathcal{S}'' \subseteq \mathcal{S}$ be the subset consisting of those $G \in \mathcal{S}$ such that $m_{\H}(G)>m_{\H}(H)$ and which have no quotients in $B$.
So we have $\mathcal{S}' \subseteq \mathcal{S}'' \subseteq \mathcal{S}$.
The order of removing groups from $\mathcal{S}'$ with others as proper quotients and removing those with quotients in $B$ does not matter, and so 
 $\mathcal{S}'$ is obtained from $\mathcal{S}''$ by removing groups with proper quotients.

We will start by showing that $\MNEC_B(H) \subseteq \mathcal{S}''$.
 Let $G \in \MNEC_B(H)$. Then there is a quotient $f : G \twoheadrightarrow H$ and, by \cref{lemma:min-non-Eichler}, there is a quotient $f : G \twoheadrightarrow Q$ for some binary polyhedral group $Q$ such that $f$ and $g$ are disjoint. Since $G$ has no quotients in $B$ and so no quotients of the form $Q_{4n}$ for $n \ge 6$, we must have $Q \in \BPG$. 
By \cref{lemma:classify-disjoint-quotients} (iv), this implies that there is an extension of groups
$ 1 \to N \to G \to  H \to 1 $
where $N \le Q$ is a normal subgroup. In particular, $N \in \BPG \cup \BPG^N$. If $N \in \BPG \setminus \BPG^N$, then we must have $N \cong Q$ and so \cref{lemma:classify-disjoint-quotients} (iii) implies that $G \cong H \times N$. It follows that $G \in \mathcal{S}_N$ and so $G \in \mathcal{S}$. Since $G \in \MNEC_B(H)$, we have $m_{\H}(G) > m_{\H}(H)$ and $G$ has no quotients in $B$. Hence $G \in \mathcal{S}''$.

We will next show that $\mathcal{S}' \subseteq \MNEC_B(H)$. Let $G \in \mathcal{S}'$, and so $G \in \mathcal{S}_N$ for some $N \in \BPG \cup \BPG^N$. This implies that there is an extension
$ 1 \to N \to G \to H \to 1$.
Since $m_{\H}(G) > m_{\H}(H)$ and $G$ has no quotients in $B$, it follows that $G$ is a non-Eichler cover of $H$. Then $G$ has a quotient $G' \in \MNEC_B(H)$. Since $\MNEC_B(H) \subseteq \mathcal{S}''$, we have $G' \in \mathcal{S}''$. Since $G \in \mathcal{S}'$, this implies that $G \cong G'$ and so $G \in \MNEC_B(H)$.

Finally, we will complete the proof by showing that $\MNEC_B(H) \subseteq \mathcal{S}'$.
Let $G \in \MNEC_B(H)$. Then $G \in \mathcal{S}''$. 
Since $G$ is finite, we can keep quotienting until we arrive at a group with no proper quotients in $\mathcal{S}''$.
That is, $G$ has a quotient $G' \in \mathcal{S}'$. Then $G' \in \MNEC_B(H)$. Since $G$ is a minimal non-Eichler cover, we have $G \cong G'$ and so $G \in \mathcal{S}'$.
\end{proof}

We will now discuss the implementation of \cref{alg:MNEC(G)-original} in Magma and GAP.
\begin{clist}{(i)}
\item
In order to determine the extensions of $H$ by $N$ when $N \in \BPG^N$ (which is a soluble group), we use the Magma function \texttt{ExtensionsOfSolubleGroup}. We then filter the groups up to isomorphism using \texttt{IsIsomorphic}. Note that we could not do this for all $N \in \BPG \cup \BPG^N$ since $\wt I$ is not soluble.
\item
We compute $\mathcal{S}$ by filtering up to isomorphism once more. We compute $m_{\H}(G)$ for $G \in \mathcal{S} \sqcup \{H\}$ as in
\cref{prop:m_H-formula}. We compute the characters $\chi_V : G \to \C^{\times}$ of the irreducible complex representations using the GAP functions \texttt{CharacterTable} and \texttt{Irr}. We then compute the degree and Frobenius-Schur indicator of each representation using the GAP functions \texttt{Degree} and \texttt{Indicator}. 
Using this, we filter out groups which are not non-Eichler covers. We filter out groups which have others as proper quotients or have quotients in $B$ using the function \texttt{Homomorphisms} which can be used to determine when surjective homomorphisms exist.	
\end{clist}

We obtain the following more general algorithm as an immediate consequence.

\begin{algorithm} \label{alg:MNEC(S)}
Input $S \le \EE$ a finite closed subgraph and a class $B$ such that $B_{\BPG} \subseteq B \subseteq \EE$.
 Output $\MNEC_B(S)$.
	The steps are as follows:
\begin{clist}{(i)}
\item
For each $G \in S$, compute $\MNEC_B(G)$ using \cref{alg:MNEC(G)-original}.
\item
Let $\mathcal{S} = \bigcup_{G \in S} \MNEC_B(G)$. Remove all groups in $S$. Remove all groups which have others as proper quotients. This gives a class of groups $\mathcal{S}'$ which coincides with $\MNEC_B(S)$.
\end{clist}
\end{algorithm}

\begin{proof}[Proof of correctness]
This follows immediately from \cref{prop:MNEC-basic-prop} and the correctness of \cref{alg:MNEC(G)-original}. Note that the algorithm terminates because $S$ is finite.
\end{proof}

Observe that \cref{alg:MNEC(S)} terminates since $S$ is finite. 
When $S$ is infinite, for example $S =\{\wt T^n \times \wt I^m : n,m \ge 0\}$, \cref{alg:MNEC(S)} does not terminate. We discuss this case in \cref{s:groups-exceptional}.

\part{Explicit computations}
\label{p:Computations}

The aim of this part is to do the explicit computations which feed into our main classification results.
In \cref{s:groups-exceptional}, we will compute $\MNEC_B(\wt T^n \times \wt I^m)$.
We will then use this in \cref{s:results}, combined with computer calculations, to obtain our main results.

\section{MNECs for the infinite family $\widetilde{T}^n \times \widetilde{I}^m$} 
\label{s:groups-exceptional}

The aim of this section will be to compute $\MNEC_B$ for the infinite family $\wt T^n \times \wt I^m$. 
Throughout this section we will fix $B =  \{Q_{4n} : n \ge 6\}$ so that all groups in $B$ fail $\SFC$.

Let $C_3 = \langle t \mid t^3\rangle$ and $Q_8 = \langle x,y \mid x^2y^{-2},yxy^{-1}x\rangle$ and note that $\wt T \cong Q_8 \rtimes_\varphi C_3$ where $\varphi(t) = \theta \in \Aut(Q_8)$ and $\theta : Q_8 \to Q_8$, $x \mapsto y$, $y \mapsto xy$ is an automorphism of order three (see, for example, GroupNames \cite{groupnames}). By taking this as our definition for $\wt T$, we obtain a canonical quotient map $f : \wt T \twoheadrightarrow C_3$.
For $1 \le k \le n$, define $f_k = (f, \cdots, f, 0, \cdots, 0): \wt T^n \twoheadrightarrow C_3$ where there are $k$ copies of $f$.
Define
\[ Q_8 \rtimes_{(k)} \wt T^n := Q_8 \rtimes_{i \circ f_k} \wt T^n \]
where $i : C_3 \to \Aut(Q_8)$, $t \mapsto \theta$. 
We have $Q_8 \rtimes_{(k)} \wt T^n \cong \wt T^{n-k} \times (Q_8 \rtimes_{(k)} \wt T^k)$.
For brevity, we write $Q_8 \rtimes \wt T^n := Q_8 \rtimes_{(n)} \wt T^n$. 
The following is the main result of this section.

\begin{thm} \label{thm:MNEC-calculation-main}
Let $n, m \ge 0$ with $(n,m) \ne (0,0)$. Then 
\[ \MNEC_B(\wt T^n \times \wt I^m) = \{\,\wt T^{n+1} \times \wt I^m, \,\, \wt T^n \times \wt I^{m+1}, \,\, \underbrace{\wt T^{n-k} \times \wt I^m \times (Q_8 \rtimes \wt T^k)}_{k=1, \cdots, n},\,\, \wt T^n \times \wt I^m \times C_2 \}. \]	
Hence $\MNEC_B(\{\wt T^n \times \wt I^m\}_{n,m}) = \{\wt T \times C_2, \wt I \times C_2\} \sqcup \{Q_8 \rtimes \wt T^k : k \ge 1 \}$ where $n,m$ ranges across the values $n,m \ge 0$, $(n,m) \ne (0,0)$.	
\end{thm}

The proof will involve a mixture of two approaches. The first is based on classifying group extensions using group cohomology. This was used to calculate $\MNEC_B$ in Algorithm \ref{alg:MNEC(G)-original}, and is the basis of the Magma function \texttt{ExtensionsOfSolubleGroup} which we used in our implementation.
The second is based on pullbacks of group quotients using \cref{prop:groups2}.

In \cref{ss:group-cohomology}, we dicuss group cohomology and group extensions. In \cref{ss:quot-hat}, we compute $\wh{\Quot}(\wt T^n \times \wt I^m,A)$ for various finite groups $A$, which is needed for the pullback approach. In \cref{ss:proof-of-T^nxI^m}, we combine the two approaches to completely determine $\wt \MNEC_B(\wt T^n \times \wt I^m)$. We then use the results on the non-existence of quotients from \cref{ss:non-existence-quotients} to obtain $\MNEC_B(\wt T^n \times \wt I^m)$.

\subsection{Cohomology and group extensions} \label{ss:group-cohomology}

The following can be found in Eilenberg-Maclane \cite{EM47}. 
Following the notation of \cite[Section IV.6]{Br82}, suppose we have an extension
\[ 1 \to N \to G \to H \to 1.\]
Since $N \, \unlhd \, G$, there is a homomorphism $\varphi : G \to \Aut(N)$ given by conjugation. Let $\Inn(N)$ denote the set of inner automorphisms of $N$ and $\Out(N) = \Aut(N)/\Inn(N)$ the outer automorphism group. Since $\varphi(N) = \Inn(N)$, $\varphi$ induces a map $\psi : H \to \Out(N)$.
If the extension is split, $\psi$ lifts to a map $\bar{\psi} : H \to \Aut(N)$. Let $\mathcal{E}(H,N)$ denote the equivalence class of extensions of the above form and let $\mathcal{E}(H,N,\psi)$ denote the those extensions in $\mathcal{E}(H,N)$ which give rise to $\psi$. 

In \cite{EM47}, Eilenberg-Maclane showed that $\mathcal{E}(H,N,\psi)$ can be computed in terms of group cohomology as follows. First note that, if $Z(N)$ is the centre of $N$, then the restriction map $\Res: \Aut(N) \to \Aut(Z(N))$ has $\Res(\Inn(N)) = \{ \id_{Z(N)}\}$ and so induces a map $\Res : \Out(N) \to \Aut(Z(N))$. Hence if $\psi : H \to \Out(N)$, then the composition 
\[ \bar{\psi} = \Res \circ \psi: H \to \Aut(Z(N))\] 
induces a $\Z H$-module structure on $Z(N)$ which we denote by $Z(N)_\psi$. The following can be found in \cite[Theorems 6.6, 6.7]{Br82}.
We often refer to (i) and (ii) as the primary and secondary obstructions for the extensions $\mathcal{E}(H,N,\psi)$.

\begin{thm} \label{thm:extension-general}
Let $H$ and $N$ be groups, let $\psi : H \to \Out(N)$ be a map and let $Z(N)_\psi$ denote the centre of $N$ which is a $\Z H$-module via $\bar{\psi}$. Then
\begin{enumerate}[\normalfont (i)]
\item
There is an obstruction in $H^3(H,Z(N)_\psi)$ which vanishes if and only if $\mathcal{E}(H,N,\psi) \ne \emptyset$
\item
If $\mathcal{E}(H,N,\psi)$ is non-empty, then there is a bijection 
\[\mathcal{E}(H,N,\psi) \cong H^2(H,Z(N)_\psi).\]
\end{enumerate}
\end{thm}

This result takes the following simpler form when $N$ is abelian (see \cite[Section IV.3]{Br82}).

\begin{corollary} \label{cor:extension-abelian}
Let $H$ and $N$ be groups with $N$ abelian, let $\psi : H \to \Aut(N)$ be a map and let $N_\psi$ denote the centre of $N$ which is a $\Z H$-module via $\bar{\psi}$. Then there is a bijection
\[\mathcal{E}(H,N,\psi) \cong H^2(H,N_\psi).\]
\end{corollary}

The following will be used to calculate $\MNEC_B$ in the following section.

\begin{prop} \label{prop:H^2=0}
Let $n, m \ge 0$. Then
$H^2(\wt T^n \times \wt I^m; \F_2) =0.
$\end{prop}

\begin{proof}
Since $\widetilde{T} \cong \SL_2(\F_3)$ and $\widetilde{I} \cong \SL_2(\F_5)$, we can use \cite[Corollary 6.18]{AM94} to deduce that $H^*(\widetilde{T}; \F_2) = H^*(\widetilde{I}; \F_2) = 0$ for $* = 1,2$. By the K\"{u}nneth formula for group cohomology \cite[Exercise 6.1.10]{We94}, we have that $H^*(\widetilde{T}^n \times \widetilde{I}^m;\F_2)=0$ for $\ast = 1,2$.
\end{proof}

\subsection{Classifying surjective group homomorphisms} \label{ss:quot-hat}

Motivated by \cref{prop:groups2}, we now calculate $\wh{\Quot}(\wt T^n \times \wt I^m,A)$ for certain finite groups $A$.
Let $\BPG = \{ Q_8, Q_{12}, Q_{16}, Q_{20}, \widetilde{T}, \widetilde{O}, \widetilde{I}\}$.
By \cref{fig:BPG-quotients}, the non-trivial proper quotients of the groups in $\BPG$ are
\[ \BPG^{*} = \{C_2, C_3, C_4, C_2^2, D_6, D_8, D_{10}, A_4, S_4, A_5\}.\]
Recall from \cref{ss:MNEC-pullback} that, for groups $G$ and $H$, we let $\Quot(G,H)$ denote the set of quotient maps $f: G \twoheadrightarrow H$ and $\Quot(G)$ denote the set of quotients of $G$, i.e. $\{ H : \Quot(G,H) \ne \emptyset\}$.
We also let $\wh{\Quot}(G,H) := \Aut(H) \setminus \Quot(G,H) / \Aut(G)$.

Let $q : Q_8 \rtimes \wt T^k \twoheadrightarrow \wt T^k$ denote the standard surjection induced by the semidirect product. As before, let $f : \wt T = Q_8 \rtimes C_3 \twoheadrightarrow C_3$ and let $f_r = (f, \cdots, f, 0, \cdots, 0): \wt T^k \to C_3$ where $0 \le r \le k$ and there are $r$ copies of $f$.
Similarly, let $g : \wt T = Q_8 \rtimes C_3 \twoheadrightarrow C_2^2 \rtimes C_3 \cong A_4$ and let $g_r = (g, \cdots, g, 0, \cdots, 0): \wt T^k \to A_4$ where $0 \le r \le k$. Finally, let $h : \wt I = \SL_2(\F_5) \twoheadrightarrow \PSL_2(\F_5) \cong A_5$ and $h_r = (h, \cdots, h, 0, \cdots, 0) : \wt I^m \to A_5$ where $0 \le r \le m$.

\begin{prop} \label{prop:homs-TxI}
Let $n,m  \ge 0$. Then
$\Quot(\wt T^n \times \wt I^m) \cap \BPG^{*} \subseteq \{C_3, A_4, A_5\}
$. Furthermore:
\begin{clist}{(i)}
\item
If $n \ge 1$, then 
$\wh{\Quot}(\wt T^n \times \wt I^m,C_3) = \{(f_r,0) : 1 \le r \le n\}$.
\item
If $m \ge 1$, then
$\wh{\Quot}(\wt T^n \times \wt I^m,A_5) = \{(0,h_t) : 1 \le t \le m\}$.
\end{clist}
\end{prop}

We will begin by establishing the following useful lemma. The case $k=0$ will be used in the proof of \cref{prop:homs-TxI}, and the general case will be used later in the proof of \cref{thm:MNEC-calculation-main}.

\begin{lemma} \label{lemma:TxIxQ_8:T^k->>C_2}
Let $n,m,k \ge 0$. Then $\wt T^n \times \wt I^m \times (Q_8 \rtimes \wt T^k) \nsurj C_2$.	
\end{lemma}

\begin{proof}
Note that, if $G$ and $H$ are finite groups with no $C_2$ quotients, then $G \times H$ has no $C_2$ quotients. This follows from
the fact that $G$ has a $C_2$ quotient if and only if the abelianisation $G^{\text{ab}}$ has even order, and $(G \times H)^{\text{ab}} \cong G^{\text{ab}} \times H^{\text{ab}}$.
Since $\wt T$ and $\wt I$ have no $C_2$ quotients (by \cref{fig:BPG-quotients}), $\wt T^n \times \wt I^m$ has no $C_2$ quotients.
It remains to show that $Q_8 \rtimes \wt T^k$ has no $C_2$ quotients for $k \ge 1$.

We will identify $Q_8 = \langle x,y \mid yxy^{-1}=x^{-1}\rangle$. Recall that $Q_8 \rtimes \wt T^k$ has action $\varphi : \wt T^k \twoheadrightarrow C_3 \le \Aut(Q_8)$, where the image is generated by $\theta : x \mapsto y$, $y \mapsto xy$.
Suppose that $f : Q_8 \rtimes \wt T^k \to C_2$ is a quotient map. Since $\wt T^k$ has no $C_2$ quotients, $f\mid_{\wt T^k} : \wt T^k \to C_2$ is trivial. Since $f$ is non-trivial, this implies that $f\mid_{Q_8} : Q_8 \to C_2$ is a quotient map. We can act by automorphisms of $Q_8$, which extend to $Q_8 \rtimes \wt T^k$, to assume that $\ker(f\mid_{Q_8}) = \langle y \rangle$. This gives $\ker(f) = \langle y, \wt T^k \rangle$. Let $z \in \wt T^k$ be such that $\varphi(z)=\theta$. Then $zyz^{-1} = \theta(y)=xy$ and so $x=zyz^{-1}y^{-1} \in \ker(f)$. This implies that $\ker(f) = \langle x,y,\wt T^n \rangle = Q_8 \rtimes \wt T^k$, which 
is a contradiction since we assumed $f$ was non-trivial.
\end{proof}

\begin{proof}[Proof of \cref{prop:homs-TxI}]
By \cref{lemma:TxIxQ_8:T^k->>C_2}, $\wt T^n \times \wt I^m$ has no $C_2$ quotient. The first part now follows since $\{C_3, A_4, A_5\}$ are precisely the groups in $\BPG^*$ with no $C_2$ quotients.

(i) From now on, let $G = \wt T^n \times \wt I^m$.
Let $f : G \to C_3$ be a quotient map. Then, for each copy of $\wt I$, we have $f \mid_{\wt I} : \wt I \to C_3$ which is trivial since $\wt I$ has no $C_3$ quotients (by \cref{fig:BPG-quotients}).
It therefore suffices to work in the case $m=0$. For each copy of $\wt T$, we have $f \mid_{\wt T} : \wt T \to C_3$.
By GroupNames \cite{groupnames}, we can deduce that
$\Hom(\widetilde{T},C_3) = \{ 1, \pm f \}$.
Since $\widetilde{T} \le \H_{\R}$, the involution on $\H_{\R}$ induces an automorphism which maps $f \mapsto -f$. Since automorphisms can also permute the copies of $\wt T^n$, we can act by $\Aut(\wt T^n) \le \Aut(G)$ to get that $f\mid_{\wt T^n} = f_k$ for some $k$.
Hence $f = (f_r,0)$ for some $0 \le r \le n$, $0 \le s \le k$. Since $f$ is surjective, we have $(r,s) \ne (0,0)$.

(ii) Let $f : G \to A_5$ be a quotient map. For each copy of $\wt I$, we have $f\mid_{\wt I} : \wt I \to A_5$. By GroupNames \cite{groupnames}, this is trivial or the unique quotient map $h$. By acting by $\Aut(\wt I^m) \le \Aut(G)$, we can therefore assume that $f \mid_{\wt I^m} = h_t$ for some $t$. For each copy of $\wt T$, we have $f \mid_{\wt T} : \wt T \to A_5$. This is trivial since $A_5$ is simple and $\wt T$ has no quotient $A_5$ (by \cref{fig:BPG-quotients}). This implies that $f = (0,h_t)$ for $0 \le t \le m$. Since $f$ is surjective, we have $t \ne 0$.
\end{proof}

\subsection{Non-existence of quotients} \label{ss:non-existence-quotients}

We now show that non-trivial quotients do not exist between certain groups in $\wt \MNEC(\wt T^n \times \wt I^m)$.
We will use the following version of Goursat's lemma.

\begin{lemma} \label{lemma:goursat}
Let $G_1, G_2$ be groups. If $H_1 \le G_1$, $H_2 \le G_2$ are subgroups, $N_1 \unlhd H_1$, $N_2 \unlhd H_2$ are normal subgroups and $\varphi : H_1/N_1 \to H_2/N_2$ is an isomorphism, then we can define a subgroup
\[ N_{\varphi} = \{ (x,y) \in H_1 \times H_2 : \varphi(xN_1)=yN_2 \} \le G_1 \times G_2. \]

Every subgroup of $G_1 \times G_2$ is of the form $N_{\varphi}$ for some $H_1,H_2,N_1,N_2$ and $\varphi$.
Furthermore, if $N_{\varphi}$ is a normal subgroup, then $H_1 \unlhd G_1$ and $H_2 \unlhd G_2$ are normal.
\end{lemma}

\begin{proof}
The main part can be found in \cite[Theorem 5.5.1]{Hall18}. Suppose $N_{\varphi}$ is a normal subgroup. For $i=1,2$, $H_i = p_i(N_{\varphi})$ where $p_i : G_1 \times G_2 \twoheadrightarrow G_i$ is projection map. This implies that $H_i \unlhd G_i$ is a normal subgroup since its is the image of a normal subgroup under a surjective homomorphism.
\end{proof}

We will begin by considering the quotients of $\wt T^n \times \wt I^m$. 

\begin{lemma} \label{lemma:Quot(T^n)}
Let $n \ge 1$. If $N \unlhd \wt T^n$ is a normal subgroup and $|N| = 2^a 3^b$, then $a \ge 3b$.
\end{lemma}

\begin{proof}
We will prove this by induction on $n$. It is true for $n=1$ since, by GroupNames \cite{groupnames}, the normal subgroups of $\wt T$ have orders $1$, $2$, $8$ and $24=2^3 \cdot 3$. Suppose it is true for all $k < n$ and let $N 
\unlhd \wt T^n$ be a normal subgroup. By Goursat's Lemma (\cref{lemma:goursat}), we have that $N \cong N_{\varphi}$ where $\varphi : H_1/ N_1 \to H_2/N_2 =: A$ is an isomorphism and $N_1 \unlhd H_1 \unlhd \wt T^{n-1}$, $N_2 \unlhd H_2 \unlhd \wt T$ are normal subgroups.
It follows that $|N_{\varphi}| = |H_1||N_2|$. 
By the inductive hypothesis, we know that $|H_1| = 2^{a}3^{b}$ where $a \ge 3b$.
By \cref{fig:BPG-quotients}, we must have $H_2 \in \{1,C_2,Q_8,\wt T\}$ and so $N_2 \in \{1,C_2, C_4, Q_8,\wt T\}$ and $|N_2| \in \{1,2,4,8,24\}$.
It follows that $|N_{\varphi}| = 2^{a'}3^{b'}$ where $(a',b')=(a+d,b)$ for $d \in \{0,1,2,3\}$ or $(a',b')=(a+3,b+1)$. In each case, $a \ge 3b$ implies that $a' \ge 3b'$. This completes the proof.
\end{proof}

\begin{prop} \label{prop:T^n->>Q_8:T^k}
Let $n,m \ge 0$ and $k \ge 1$. Then $\wt T^n \times \wt I^m \nsurj Q_8 \rtimes \wt T^k$.
\end{prop}

\begin{proof}
Suppose $f : \wt T^n \times \wt I^m \twoheadrightarrow Q_8 \rtimes \wt T^k$ is a quotient map. Then the image $f(\wt I)$ of each copy of $\wt I$ is isomorphic to $1$, $A_5$ or $\wt I$ (by \cref{fig:BPG-quotients}). Since $5 \mid |A_5|, |\wt I|$ and $5 \nmid |Q_8 \rtimes \wt T^k|$, $f(\wt I) \not \cong A_5$ or $\wt I$ and so we must have $f(\wt I)=1$ for all copies of $\wt I$. It therefore suffices to consider the case where $m = 0$, so that $f : \wt T^n \twoheadrightarrow Q_8 \rtimes \wt T^k$ and so we necessarily have that $n \ge k+1$ by the orders of the groups. 
Then $N = \ker(f)$ is a normal subgroup of $\wt T^n$ with $|N| = \frac{1}{8} |\wt T|^{n-k} = 2^{3(n-k)-3} 3^{n-k}$. This contradicts \cref{lemma:Quot(T^n)}.
\end{proof}

We will now move on to the larger family $\wt T^n \times \wt I^m \times (Q_8 \rtimes \wt T^k)$.

\begin{lemma} \label{lemma:Quot(Q_8:T^k)}
Let $k \ge 1$. If $N \unlhd Q_8 \rtimes \wt T^k$ is a normal subgroup and $|N|  = 2^a3^b > 1$, then $a \ne 3b$.
\end{lemma}

\begin{proof}
Suppose for contradiction that $Q_8 \rtimes \wt T^{k}$ has a normal subgroup $N$ of order $2^{3t}3^t$ where $t \ge 1$. Analogously to Goursat's Lemma (\cref{lemma:goursat}), let $N_1 = N \cap Q_8$ and let $H_2$ denote the projection of $N$ to $\wt T^{k}$. Then $N_1 \unlhd Q_8$, $H_2 \unlhd \wt T^{k}$ are normal subgroups. For each $y \in H_2$, consider $S_y = \{x \in Q_8 : (x,y) \in N\}$ where $(x,y) \in Q_8 \rtimes \wt T^{k}$. Then $S_y = N_1x$ for any $x \in S_y$ and so $|N| = |N_1||H_2|$. Since $N_1 \unlhd Q_8$, we have $|N_1| \in \{1,2,4,8\}$. Since $H_2 \unlhd \wt T^{k}$, \cref{lemma:Quot(T^n)} implies that $|H_2|=2^a3^b$ for $a \ge 3b$. Since $|N|=2^{3t}3^t$, this now implies that $|N_1|=1$ and $|N|=|H_2|$. In particular, the projection map $p_2 : Q_8 \rtimes \wt T^{k} \twoheadrightarrow \wt T^{k}$ restricts to an isomorphism $p_2 \mid_N : N \to H_2$.

Now suppose $(x,y) \in N$. Since $N$ is normal, we have that $(z,1)(x,y)(z^{-1},1) \in N$ for all $z \in Q_8$. We have $(z,1)(x,y)(z^{-1},1) = (zx\varphi_y(z^{-1}),y) \in N$ where $\varphi_{(\cdot)} : \wt T^{k} \twoheadrightarrow C_3 \le  \Aut(Q_8)$ is the action defined by the semidirect product.
Since $p_2 \mid_N$ is an isomorphism, this implies that $x=zx\varphi_y(z^{-1})$ and so $\varphi_y(z) =x^{-1}zx$ for all $z \in Q_8$ and so $\varphi_y \in \Inn(Q_8)$. Since $\Inn(Q_8) \cong Q_8/Z(Q_8) \cong C_2^2$ and $|\IM(\varphi_{(\cdot)})|=3$, we have that $\IM(\varphi_{(\cdot)}) \cap \Inn(Q_8) = \{\id\}$. It follows that $\varphi_y = \id$ and so $y \in \ker(\varphi_{(\cdot)} : \wt T^{k} \to \Aut(Q_8))$. In particular, we have
\[ N \cong H_2 \unlhd \ker(\varphi_{(\cdot)} : \wt T^{k} \to \Aut(Q_8)) \cong \ker(f_k : \wt T^{k} \twoheadrightarrow C_3) \]
where $f_k = (f,\cdots, f)$ is as defined at the start of this section. Projection onto the first $k-1$ coordinates gives a split surjection $\ker(f_k) \twoheadrightarrow \wt T^{k-1}$ with kernel isomorphic to $Q_8$. The action of $\wt T^{k-1}$ on $Q_8$ is given by $\varphi' = (i \circ f_{k-1}) : \wt T^{k-1} \twoheadrightarrow C_3 \le \Aut(Q_8)$. That is, $\ker(f_k : \wt T^{k} \twoheadrightarrow C_3) \cong Q_8 \rtimes \wt T^{k-1}$. Hence we have a normal subgroup $H_2 \unlhd Q_8 \rtimes \wt T^{k-1}$ of order $2^{3t}3^t$.
We can repeat this argument $k$ times to get that there is a normal subgroup $N' \unlhd Q_8$ of order $2^{3t}3^{t}$. This is a contradiction since $3 \nmid |Q_8|$.	
\end{proof}

\begin{prop} \label{prop:Q_8:T^n->>Q_8:T^m}
Let $n,m \ge 0$ and $k, r \ge 1$ with $k \ne r$. Then $\wt T^n \times \wt I^m \times (Q_8 \rtimes \wt T^k) \nsurj Q_8 \rtimes \wt T^r$.
\end{prop}

\begin{proof}
By the same argument as given for \cref{prop:T^n->>Q_8:T^k}, 
we can restrict to the case where $m=0$ and so there is a quotient map 
$f : Q_8 \rtimes_{(k)} \wt T^{n+k} \twoheadrightarrow Q_8 \rtimes \wt T^r$ 
where $Q_8 \rtimes_{(k)} \wt T^{n+k} \cong \wt T^n \times (Q_8 \rtimes \wt T^k)$. This implies that $Q_8 \rtimes_{(k)} \wt T^{n+k}$ has a normal subgroup $N$ of order $2^{3t}3^t$ where $t = n+k-r \ge 1$. 
Let $N_1 = N \cap Q_8$ and let $H_2$ denote the projection of $N$ to $\wt T^{n+k}$, so that $N_1 \unlhd Q_8$, $H_2 \unlhd \wt T^{n+k}$ are normal subgroups.
By an identical argument to the one given in \cref{lemma:Quot(Q_8:T^k)}, we get that
\[ N \cong H_2 \unlhd \ker(\varphi_{(\cdot)} : \wt T^{n+k} \twoheadrightarrow C_3) \cong \wt T^n \times \ker(f_k : \wt T^{k} \twoheadrightarrow C_3) \cong Q_8 \rtimes_{(k-1)} \wt T^{n+k-1}. \]

We can repeat this argument $k$ times to get that $N$ projects to a normal subgroup $N' \unlhd \wt T^n \times Q_8$ of order $2^{3t}3^{t}$. By Goursat's Lemma (\cref{lemma:goursat}), we get that $|N'| = |H_1'||N_2'|$ where $H_1' \unlhd \wt T^n$ and $N_2' \le Q_8$. By \cref{lemma:Quot(T^n)}, we have that $|H_1'| = 2^a3^b$ where $a \ge 3b$. This implies that $|N_2'|=1$ and it follows that $N' \unlhd \wt T^n$. It follows that
\[ Q_8 \rtimes \wt T^r = f(\wt T^n \times (Q_8 \rtimes \wt T^k)) \cong (\wt T^n/N') \times (Q_8 \rtimes \wt T^k). \]
However, this implies that $Q_8 \rtimes \wt T^r$ has a normal subgroup of order $2^{3(n-t)}3^{n-t}$ induced by $\wt T^n/N'$. This contradicts \cref{lemma:Quot(Q_8:T^k)}.
\end{proof}

\subsection{Proof of \cref{thm:MNEC-calculation-main}} \label{ss:proof-of-T^nxI^m}

Let $H= \wt T^n \times \wt I^m$ and $G \in \MNEC_B(H)$. By Proposition \ref{prop:groups2}, there exists $Q \in \BPG$, a finite group $A$ and representatives $f_1 \in \wh{\Quot}(H,A)$, $f_2 \in \wh{\Quot}(Q,A)$ such that $G \cong H \times_{f_1,f_2} Q$.
If $A = 1$, then $G \cong H \times Q$. If $Q \cong \wt T$ or $\wt I$, then we obtain $\wt T^{n+1} \times \wt I^m$ and $\wt T^n \times \wt I^{m+1}$ respectively.
If $Q \ne \wt T, \wt I$, then $Q \twoheadrightarrow C_2$ (by \cref{fig:BPG-quotients}) and so there is a proper quotient $G \twoheadrightarrow H \times C_2$. Since $H \times C_2 \twoheadrightarrow H$ is non-Eichler, this implies that $G$ is not minimal which is a contradiction.
If $A \cong Q$, then $G \cong H$ which is again a contradiction. 
If $A \ne 1, Q$, then $A \in \Quot(H) \cap \BPG^* \subseteq \{C_3,A_4,A_5\}$ by \cref{prop:homs-TxI}. We will now consider each case in turn.

If $A = C_3$, then $Q = \wt T$ (by \cref{fig:BPG-quotients}). We have $\wh{\Quot}(H,C_3) = \{(f_k,0)\}_{k=1}^n$ (by \cref{prop:homs-TxI}) and $\wh{\Quot}(\wt T,C_3) = \{f\}$.
This implies that
\[ G \cong H \times_{(f_k,0),f} \wt T \cong \wt T^{n-k} \times \wt I^m \times (\wt T^k \times_{f_k,f} \wt T). \]
Projection gives a split surjection $\wt T^k \times_{f_k,f} \wt T \twoheadrightarrow \wt T^k$ with kernel isomorphic to $\ker(\wt T \twoheadrightarrow C_3) \cong Q_8$.
The action of $\wt T^k$ on $Q_8$ coincides with the action $\varphi : \wt T^k \twoheadrightarrow C_3 \le \Aut(Q_8)$ defined previously and so $\wt T^k \times_{f_k,f} \wt T \cong Q_8 \rtimes_{\varphi} \wt T^k$. Hence $G \cong \wt T^{n-k} \times \wt I^m \times (Q_8 \rtimes \wt T^k)$ for some $1 \le k \le n$.

If $A = A_4$, then $Q = \wt T$ (by \cref{fig:BPG-quotients}). Since $\ker(\wt T \twoheadrightarrow A_4) \cong C_2$, \cref{lemma:classify-disjoint-quotients} (iv) implies that
\[ 1 \to C_2 \to G \to H \to 1. \]
Since $\Aut(C_2) = 1$, there is a unique map $\psi : H \to \Aut(C_2)$. By \cref{cor:extension-abelian} and \cref{prop:H^2=0}, we have
$\mathcal{E}(H,C_2) \cong H^2(H,\F_2) = 0$.
This implies that $\mathcal{E}(H,C_2)$ contains only the trivial extension, and so $G \cong H \times C_2$. 
If $A = A_5$, then $Q = \wt I$ (by \cref{fig:BPG-quotients}). Since $\ker(\wt I \twoheadrightarrow A_5) \cong C_2$, we have that $G \cong H \times C_2$ by the same argument.

Conversely, all groups obtained are non-Eichler covers of $\wt T^n \times \wt I^m$ and it can be shown that each group has no quotient in $B$. No groups are pairwise isomorphic: they each have distinct orders except the groups $\wt T^{n-k} \times \wt I^m \times (Q_8 \rtimes \wt T^k)$ for $k=1, \cdots, n$ which are distinct by \cref{prop:Q_8:T^n->>Q_8:T^m}. No groups have others as proper quotients: $\wt T^n \times \wt I^m \times C_2$ does not quotient onto any other groups since it has the smallest order, and no groups have quotient $\wt T^n \times \wt I^m \times C_2$ since they do not have $C_2$ quotients (by \cref{lemma:TxIxQ_8:T^k->>C_2}). Similarly, by the order of the groups, none of the groups $\wt T^{n-k} \times \wt I^m \times (Q_8 \rtimes \wt T^k)$ can have quotient $\wt T^{n+1} \times \wt I^m$ or $\wt T^n \times \wt I^{m+1}$, and the converse follows from \cref{prop:T^n->>Q_8:T^k}.

This completes the computation of $\MNEC_B(\wt T^n \times \wt I^m)$. 
It follows immediately that
\[ \MNEC_B(\{\wt T^n \times \wt I^m\}_{n,m}) = (\wt{\MNEC}_B(\{\wt T^n \times \wt I^m\}_{n,m}))^{\QF} = (\{\wt T \times C_2, \wt I \times C_2\} \sqcup \{Q_8 \rtimes \wt T^k : k \ge 1 \})^{\QF}. \]
By \cref{prop:Q_8:T^n->>Q_8:T^m}, we have that $Q_8 \rtimes \wt T^k \nsurj Q_8 \rtimes \wt T^r$ for all $r \ne k$, and \cref{lemma:TxIxQ_8:T^k->>C_2} implies that $Q_8 \rtimes \wt T^k \nsurj \wt T \times C_2, \wt I \times C_2$.
Hence $\MNEC_B(\{\wt T^n \times \wt I^m\}_{n,m}) = \{\wt T \times C_2, \wt I \times C_2\} \sqcup \{Q_8 \rtimes \wt T^k : k \ge 1 \}$.

\section{Final results} \label{s:results}

We will now discuss the various partial classification results which we obtain in light of the calculations done both via computer, which can be found in Appendix \ref{s:tables}, and for the infinite families which were dealt with in \cref{s:groups-exceptional}.

\subsection{Proof of \cref{thm:main-group-theory}}

We will now prove the following result from the introduction. We will make frequent use of the computer calculations which are presented in Appendix \ref{s:tables}.

\begin{reptheorem}{thm:main-group-theory}
Let $G$ be a finite group. Then the following are equivalent:
\begin{clist}{(i)}
\item 
$G$ has an Eichler quotient $H$ of the form
\[ C_1, \, Q_8, \, Q_{12}, \, Q_{16}, \, Q_{20}, \, \widetilde{T}, \, \widetilde{O}, \, \widetilde{I} \quad \text{or} \quad \widetilde{T}^n \times \widetilde{I}^m \text{ for $n, m \ge 0$}\]
\item 
$G$ has no quotient $H$ of the form:
\begin{clist}{(a)}
\item
$Q_{4n}$ for $n \ge 6$
\item 
$Q_8 \rtimes \widetilde{T}^n$ for $n \ge 1$
\item
$Q_8 \times C_2$, \, $Q_{12} \times C_2$, \, $Q_{16} \times C_2$, \, $Q_{20} \times C_2$, \, $\widetilde{T} \times C_2$, \, $\widetilde{O} \times C_2$, \, $\widetilde{I} \times C_2$, \, $G_{(32,14)}$, $G_{(36,7)}$, $G_{(64,14)}$, $G_{(96,66)}$, $G_{(100,7)}$, $G_{(384, 18129)}$, $G_{(1152,155476)}$.
\end{clist}
\end{clist}
\end{reptheorem}

\begin{proof}
Let $S$ denote the groups $H$ listed in (i) above, i.e. the groups in \cref{eq:BPGs+TxI}.
The seven binary polyhedral groups listed, which we write as $\BPG$, are in $\MNEC(C_1)$ (by \cref{lemma:minimal-BPG}) and $\wt T^n \times \wt I^m \in \MNEC^{n+m}(C_1)$ (by \cref{thm:MNEC-calculation-main}). In particular, we have $S \subseteq \EE$. 

To see that $S$ is closed under quotients, suppose that there exists $G \in S$ and $H \in \EE \,\setminus\, S$ with $G \twoheadrightarrow H$. The only reasonable option is $G = \wt T^n \times \wt I^m$ for some $n, m \ge 0$.
Since $\wt T^n \times \wt I^m \nsurj C_2$ for $n,m \ge 0$ (by \cref{prop:homs-TxI}), it follows that $H \nsurj C_2$ and so $H \nsurj Q_8, \, Q_{12}, \, Q_{16}, \, Q_{20}, \, \widetilde{O}$. In particular, $H \twoheadrightarrow \wt T$ or $\wt I$. Let $a, b \ge 0$ be such that $H \twoheadrightarrow \wt T^a \times \wt I^b$ and $a+b$ is maximal. Then $H \not \in S$ and $H \in \EE$ implies that $H \twoheadrightarrow H' \in \MNEC(\wt T^a \times \wt I^b)$. By assumption, $H \nsurj \wt T^{a+1} \times \wt I^b$, $\wt T^a \times \wt I^{b+1}$, $C_2$ and so $H' \cong \wt T^{a-k} \times \wt I^b \times (Q_8 \rtimes \wt T^k)$ for some $1 \le k \le a$ by \cref{thm:MNEC-calculation-main}. This implies that $\wt T^n \times \wt I^m$ has a quotient $Q_8 \rtimes \wt T^k$, which is a contradiction by \cref{prop:T^n->>Q_8:T^k}.

By the Fundamental Lemma (\cref{lemma:fundamental-lemma}), a finite group $G$ has an Eichler quotient in $S$ if and only if $G$ has no quotient in $\MNEC(S)$. 
Let $B = \{Q_{4n} : n \ge 6\} \subseteq \EE$. If $G$ has an Eichler quotient in $S$, then $G$ has no quotient in $B$. It follows that a finite group $G$ has an Eichler quotient in $S$ if and only if $G$ has no quotient in $B$ or $\MNEC_{B}(S)$. It remains to determine
\[ \MNEC_{B}(S) = \left\{ (\MNEC_{B}(\BPG) \cup \MNEC_{B}(\{\wt T^n \times \wt I^m\}_{n,m})) \, \setminus \, (\BPG \cup \{\wt T^n \times \wt I^m\}_{n,m})  \right\}^{\QF} \]
where  $n,m \ge 0$ and $(n,m) \ne (0,0)$, and equality follows from \cref{prop:MNEC-basic-prop}.
 
Firstly, we have $\MNEC_{B}(\BPG) = \MNEC_{B}(\G_1) = \G_2$. By \cref{fig:G_2} (or the tables in Appendix \ref{s:tables}), this consists of the groups listed in (ii) (c) as well as the groups $Q_8 \rtimes \wt T$,  $\wt T^2$,  $\wt T \times \wt I$ and $\wt I^2$.

Secondly, by \cref{thm:MNEC-calculation-main}, we have
$ \MNEC_{B'}(\{\wt T^n \times \wt I^m\}_{n,m}) = \{\, Q_8 \rtimes \wt T^k : k \ge 1 \,\} $
where $B' = B \cup \{\wt T \times C_2\} \cup \{\wt I \times C_2\}$. In particular, we have
\[ \MNEC_{B}(S) = \{ (c) \cup \{\, Q_8 \rtimes \wt T^k : k \ge 1 \,\}\}^{\QF}   \]
where (c) denote the sets of groups listed in (ii) (c). Firstly, note that $Q_8 \rtimes \wt T^k$ has no quotient in (c) since $Q_8 \rtimes \wt T^k$ has no $C_2$ quotient (by \cref{lemma:TxIxQ_8:T^k->>C_2}) and the groups in (c) all have $C_2$ quotients (this follows from the BPG quotients arising in \cref{fig:G_2}). The groups in (c) have no $Q_8 \rtimes \wt T$ quotient since $Q_8 \rtimes \wt T \in \G_2$ (so they would have been filtered out otherwise). Since $|Q_8 \rtimes \wt T^k| \ge 4608$ for $k \ge 2$, they have no $Q_8 \rtimes \wt T^k$ quotients for any $k \ge 1$. Hence $\MNEC_{B}(S)$ consists of the groups in (c) and the groups $Q_8 \rtimes \wt T^k$ for $k \ge 1$. This completes the proof of \cref{thm:main-group-theory}.
\end{proof}

We can similarly deduce the following result from the introduction.

\begin{reptheorem}{thm:TxC2-group-theory}
Let $G$ be a finite group such that $G \twoheadrightarrow \wt T \times C_2$. Then the following are equivalent:
\begin{clist}{(i)}	
\item
	$G$ has an Eichler quotient $\wt T \times C_2$
\item	
	$G$ has no quotient of the form 
	$Q_{4n}$ for $n \ge 6$,
	$\wt T \times C_2^2$, $\wt T \times Q_{12}$, $(Q_8 \rtimes \wt T) \times C_2$, $\wt T \times Q_{20}$, $\wt T \times \wt O$ or $\wt T^2 \times C_2$.
\end{clist}	
\end{reptheorem}

\begin{proof}
By \cref{prop:groups1}, it suffices to compute $\MNEC_{B}(H)$ where $B = \{Q_{4n} : n \ge 6\}$. By \cref{fig:G_2}, this consists of the groups $\wt T \times C_2^2$, $\wt T \times Q_{12}$, $(Q_8 \rtimes \wt T) \times C_2$, $\wt T \times Q_{20}$, $\wt T \times \wt O$ and $\wt T^2 \times C_2$.
\end{proof}

\subsection{Further partial classifications}
\label{ss:classification-partial}

We now consider two additional partial classifications.

\begin{thm} \label{thm:classification-order<=383}
Let $G$ be a finite group with $|G| \le 383$. Precisely one of the following holds.
\begin{clist}{(A)}
\item
$G$ has a quotient of the form 
$Q_{4n}$ for $n \ge 6$, 
$Q_8 \times C_2$,  $Q_{12} \times C_2$, $Q_{16} \times C_2$, $Q_{20} \times C_2$, $\widetilde{T} \times C_2^2$, $\widetilde{O} \times C_2$, $G_{(32,14)}$, $G_{(36,7)}$, $G_{(64,14)}$ or $G_{(100,7)}$ (in which case $\Z G$ does not have $\SFC$)
\item
$G \cong Q_8 \rtimes \wt T$, $\widetilde{I} \times C_2$ or $\wt T \times Q_{12}$
\item
$G$ has an Eichler quotient $H$ where $H$ is one of the following groups:
\begin{clist}{(i)}
\item
$C_1, \, Q_8, \, Q_{12}, \, Q_{16}, \, Q_{20}, \, \widetilde{T}, \, \widetilde{O}, \, \widetilde{I}$ (in which case $\Z G$ has $\PC$)
\item
$\wt T \times C_2$, $G_{(96,66)}$.
\end{clist}
\end{clist}
\end{thm}

\begin{proof}
Let $S$ be the list of groups in \cref{eq:BPGs+TxI} as well as $\{\wt T \times C_2, G_{(96,66)}, Q_8 \rtimes \wt T, \wt I \times C_2, \wt T \times Q_{12}\}$, so that $S \subseteq \EE$ is closed under quotients. Let $B$ be the quotient filter of the sets of groups listed in \cref{thm:main-group-theory} (ii) (a)-(c) as well as $\MNEC_{B}(H)$ for each $H \in \{\wt T \times C_2, G_{(96,66)}, Q_8 \rtimes \wt T, \wt I \times C_2, \wt T \times Q_{12}\}$ where $B = \{Q_{4n} : n \ge 6\}$. 
Define $\ol{S} = S \cap \{G : |G| \le 383\}$ and $\ol{B} = B \cap \{G : |G| \le 383\}$.
If $G$ is a finite group such that $|G| \le 383$ then, by the Fundamental Lemma (\cref{lemma:fundamental-lemma}) and \cref{thm:main-group-theory}, either $G$ has an Eichler quotient in $\ol{S}$ or a quotient in $\ol{B}$ (but not both). 
Next note that $\ol{S}$ consists of the groups in (B) and (C) (i)-(ii), and $\ol{B}$ consists of the groups listed in (A). This requires that $\MNEC_{B}(G_{(96,66)}) \cap \{G : |G| \le 383\} = \emptyset$, which follows by a computer calculation.
Finally note that the groups in (B) have order $\ge 383/2$ and so, if $G$ has a quotient $H$ in (B), then $G \cong H$.
\end{proof}

This reduces the complete classification of when PC holds for groups of order $\le 383$ to checking PC (resp. SFC) for $Q_8 \rtimes \wt T$, $\wt I \times C_2$ and $\wt T \times Q_{12}$, and checking a weak version of PC lifting (resp. SFC lifting) for $\wt T \times C_2$ and $G_{(96,66)}$. 
By results of Bley-Hofmann-Johnston \cite{BHJ24} (see \cref{thm:BHJ-part-1}), the groups $Q_8 \rtimes \wt T$, $\wt I \times C_2$ and $\wt T \times Q_{12}$ have SFC. By \cref{thm:TxC2}, $\wt T \times C_2$ has SFC lifting. 
By calculating group extensions, we can show that $G_{(192,183)}$ is the unique group $G$ with $|G| \le 383$ which has $G_{(96,66)}$ as a proper Eichler quotient. By \cref{thm:BHJ-part-1}, $G_{(192,183)}$ has SFC. Hence we obtain the following, which is one of the main results in \cite{BHJ24}:

\begin{corollary}[Bley-Hofmann-Johnston] \label{cor:|G|<=383}
Let $G$ be a finite group with $|G| \le 383$. Then the following are equivalent:
\begin{clist}{(i)}
\item
$\Z G$ has stably free cancellation
\item
$G$ has no quotient of the form 
$Q_{4n}$ for $n \ge 6$, 
$Q_8 \times C_2$,  $Q_{12} \times C_2$, $Q_{16} \times C_2$, $Q_{20} \times C_2$, $\widetilde{T} \times C_2^2$, $\widetilde{O} \times C_2$, $G_{(32,14)}$, $G_{(36,7)}$, $G_{(64,14)}$ or $G_{(100,7)}$.	
\end{clist}
\end{corollary}

\begin{remark}
Note that we have given a proof that (ii) implies (i) using a slightly different approach to the one taken in \cite{BHJ24}. Whilst they used the classification of groups of order $\le 383$ (using GAP's Small Groups library \cite{BEO02,GAP4}), we proved the result as a consequence of \cref{thm:main-group-theory}. 
\end{remark}

We also have the following for $m_{\H}(G)$ bounded. The proof is similar to that of \cref{thm:classification-order<=383} except that we use the computations of $m_{\H}(G)$ given in Appendix \ref{s:tables}. We also use that $Q_8 \times C_2$ and $Q_{12} \times C_2$ are the only groups in \cref{thm:main-group-theory} (ii) (c) with $m_{\H}(G) \le 2$.

\begin{thm} \label{thm:classification-mH<=2}
Let $G$ be a finite group with $m_{\H}(G) \le 2$. Precisely one of the following holds.
\begin{clist}{(A)}
\item
$G$ has a quotient of the form $Q_8 \times C_2$ or $Q_{12} \times C_2$ (in which case $\Z G$ does not have $\PC$)
\item
$G$ has an Eichler quotient $H$ where $H$ is one of the following groups:
\begin{clist}{(i)}
\item
$C_1, \, Q_8, \, Q_{12}, \, Q_{16}, \, Q_{20}, \, \widetilde{T}, \, \widetilde{O}, \, \widetilde{I}$ (in which case $\Z G$ has $\PC$)
\item
$\wt T \times C_2$, $Q_8 \rtimes \wt T$, $\wt T^2$.
\end{clist}
\end{clist}
\end{thm}

This reduces the complete classification of when PC holds for such groups to checking PC for one group $Q_8 \rtimes \wt T$ (which has SFC by  \cref{thm:BHJ-part-1}) and checking PC lifting for at most two groups $\wt T \times C_2$ and $Q_8 \rtimes \wt T$.
If it was known that $Q_8 \rtimes \wt T$ had SFC lifting then, similarly to \cref{cor:|G|<=383}, we would have that a finite group $G$ with $m_{\H}(G) \le 2$ has SFC if and only if $G$ has no quotient of the form $Q_8 \times C_2$ or $Q_{12} \times C_2$.

In fact, since $m_{\H}(H) = 2$ for $H = Q_8 \times C_2, \, Q_{12} \times C_2, \, \wt T \times C_2, \, Q_8 \rtimes \wt T, \, \wt T^2$, we obtain:

\begin{corollary}
Let $G$ be a finite group with $m_{\H}(G) \le 1$. Then $\Z G$ has $\PC$.		
\end{corollary}

This generalises \cite[Theorem C]{Ni18} which established the analogous result for $\SFC$.

\subsection{Proof of \cref{thm:main-prescribed}}
\label{ss:classification-prescribed-quotient}

We will now prove the following result from the introduction.

\begin{reptheorem}{thm:main-prescribed}
Let $G$ be a finite group such that $G \twoheadrightarrow C_2^2$. Then the following are equivalent:
\begin{clist}{(i)}
\item
$\Z G$ has projective cancellation
\item
$\Z G$ has stably free cancellation
\item
$G$ has no quotients of the form $Q_{4n}$ for $n \ge 6$ even, $Q_{4n} \times C_2$ for $n=2,4$ or $n \ge 3$ odd, $\wt T \times C_2^2$, $\wt O \times C_2$, $\wt I \times C_2^2$, $G_{(32,14)}$ or $G_{(64,14)}$
\item
$G$ has an Eichler quotient of the form $C_2^2$, $Q_8$ or $Q_{16}$.
\end{clist}
\end{reptheorem}

For a class of groups $A$, define $\EE(A) = \bigcup_{n \ge 0} \MNEC^n(A)$ which we will view as a directed graph with edge set $\{(G,H) : G \in \MNEC(H)\}$ where $(G,H)$ is an edge from $G$ to $H$. 
We refer to the groups in $\EE(A)$ as \textit{Eichler simple groups over $A$}.

Similarly to Appendix \ref{ss:main-diagram}, we start by making a number of definitions. Let $\G_0(C_2^2) = \{C_2^2\}$ and $\G_1(C_2^2) = \MNEC(\G_0(C_2^2))$. If $B = \{ Q_{4n} : \text{ $n \ge 6$ }\}$, then our GAP program can compute $\{G \in \G_1(C_2^2) : \text{ $G$ has no quotient in $B$ }\}$. Note that $Q_{4n} \twoheadrightarrow C_2$ for all $n \ge 6$, but $Q_{4n} \twoheadrightarrow C_2^2$ if and only if $n$ is even. If $G \twoheadrightarrow C_2^2, Q_{4n}$ for $n$ odd, then $G \twoheadrightarrow C_2^2 \times_{C_2} Q_{4n} \cong Q_{4n} \times C_2$ since we can choose generators for $C_2^2$ such that the quotient map $C_2^2 \twoheadrightarrow C_2$ is projection. Combining this with the computations gives $\G_1(C_2^2)$, which can be found in \cref{fig:G_2(C_2^2)}. (Note that technically the groups in $Q_{4n}$ and $Q_{4n} \times C_2$ for $n \ge 6$ are not all in $\G_1(C_2^2)$, but $(\,\cdot\,)^{\QF}$ of this class.)

Let $B_1(C_2^2) = \G_1(C_2^2) \, \setminus \, \{Q_8, Q_{16}\}$. By \cref{fig:G_2} (or Appendix \ref{s:tables}), these are the groups in $\G_1(C_2^2)$ which fail PC (and which fail SFC). Let $\G_2(C_2^3) = \MNEC_{B_1(C_2^2)}(\G_1(C_2^2))$. Again by \cref{fig:G_2}, we have that $\G_2(C_2^2)$ consists of the four groups $Q_8 \times C_2$, $G_{(32,14)}$, $G_{16} \times C_2$, $G_{(64,14)}$ listed in \cref{fig:G_2(C_2^2)}.
All four groups fail PC (and SFC) by \cref{fig:G_2}.

\captionsetup{aboveskip=8pt}

\begin{figure}[h]
\def\layersep{4cm}
\begin{center}
\hspace{4mm}
\begin{tikzpicture}[draw=black!75, node distance=\layersep]
\definecolor{BLUE}{HTML}{318CE7}
\definecolor{RED}{HTML}{	E32636}
    \tikzstyle{neuron}=[fill=white,draw=black!75,minimum size=20pt,inner sep=0pt,rounded corners]
    \tikzstyle{bbneuron}=[pattern=horizontal lines, pattern color=BLUE, draw=black!75,minimum size=20pt,inner sep=0pt,rounded corners]
    \tikzstyle{bneuron}=[fill=BLUE!65, draw=black!75,minimum size=20pt,inner sep=0pt,rounded corners]
    \tikzstyle{rneuron}=[fill=red!50, draw=black!75,minimum size=20pt,inner sep=0pt,rounded corners]
    \tikzstyle{annot} = [text width=4em, text centered]

    \node[bneuron] (I) at (0,-2.6) {$C_2^2$};

\def\sep{0.8}
\def\shift{-0.4cm}
        \path[yshift=1cm]node[bneuron] (H-1) at (\layersep,\shift-\sep*1 cm) {\text{ $\quad Q_8 \quad$ }};
        \path[yshift=1cm]node[bneuron] (H-2) at (\layersep,\shift-\sep*2 cm) {\text{ $\quad Q_{16} \quad$ }};
        \path[yshift=1cm]node[rneuron] (H-6) at (\layersep,\shift-\sep*3 cm) {\text{ $\wt T \times C_2^2$ }};        
        \path[yshift=1cm]node[rneuron] (H-5) at (\layersep,\shift-\sep*4 cm) {\text{ $\wt O \times C_2$ }};
        \path[yshift=1cm]node[rneuron] (H-7) at (\layersep,\shift-\sep*5 cm) {\text{ $\wt I \times C_2^2$ }};
        \path[yshift=1cm]node[rneuron] (H-8) at (\layersep,\shift-\sep*6 cm) {\text{ $Q_{4n}$,  $\substack{\text{\normalsize $n \ge 6$} \\ \text{\normalsize even}}$ }};
        \path[yshift=1cm]node[rneuron] (H-9) at (\layersep,\shift-\sep*7 cm) {\text{ $Q_{4n} \times C_2$, $\substack{\text{\normalsize $n \ge 3$} \\ \text{\normalsize odd}}$ }};

\node (I) at ($(I)+(0:-0.1)$) {};

\foreach \x in {1,2,5,6,7,8,9} \node (H-\x) at ($(H-\x) +(0:-0.55) $) {}; 

\node (H-8)  at ($(H-8) +(0:-0.2) $) {}; 
\node (H-9)  at ($(H-9) +(0:-0.6) $) {}; 

\begin{scope}[on background layer]
    \foreach \source in {1}
        \foreach \dest in {1,2,5,6,7,8,9}
            \path (I) edge (H-\dest);
\end{scope}

\foreach \x in {1,2,5,6,7} \node (H-\x) at ($(H-\x) +(0:1.125) $) {}; 

\def\sep{1.45}
\def\shiftt{0.5cm}
        \path[yshift=0.5cm]node[rneuron] (L-1) at (2*\layersep,\shiftt-\sep*1 cm) {\text{ $Q_8 \times C_2$ }};
        \path[yshift=0.5cm]node[rneuron] (L-2) at (2*\layersep,\shiftt-\sep*2 cm) {\text{ $\, G_{(32,14)} \,$ }};
        \path[yshift=0.5cm]node[rneuron] (L-3) at (2*\layersep,\shiftt-\sep*3 cm) {\text{ $Q_{16} \times C_2$ }};
        \path[yshift=0.5cm]node[rneuron] (L-4) at (2*\layersep,\shiftt-\sep*4 cm) {\text{ $\, G_{(64,14)} \,$ }};   

\foreach \x in {1,...,4} \node (L-\x) at ($(L-\x) +(0:-0.5) $) {}; 

\begin{scope}[on background layer]
	\path (H-1) edge (L-1) edge (L-2);
	\path (H-2) edge (L-2) edge (L-3) edge (L-4);
\end{scope}
      	
    \node (g0) at (0,0.6) {};
    \node[annot,right of=g0] (g1) {$\G_1(C_2^2)$};
    \node[annot,right of=g1] (g2) {$\G_2(C_2^2)$};
\end{tikzpicture}
\end{center}
\caption{Graphical structure of $\G_0(C_2^2)$, $\G_1(C_2^2)$ and $\G_2(C_2^2)$}
\label{fig:G_2(C_2^2)}
\end{figure}
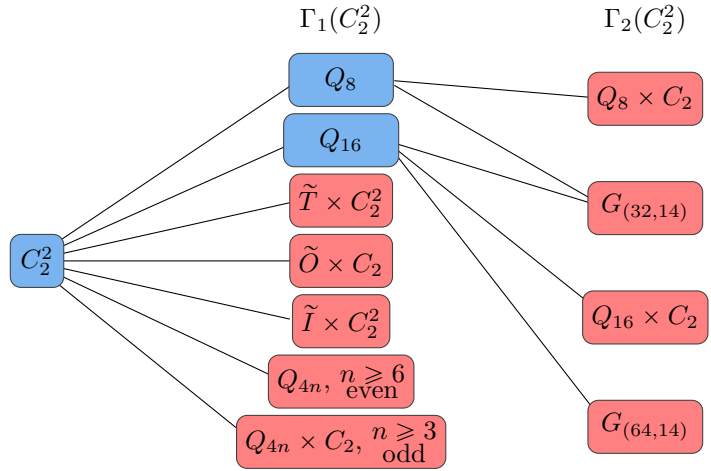

\begin{proof}
First note that (iv) $\Rightarrow$ (i) $\Rightarrow$ (ii) $\Rightarrow$ (iii) where the first arrow follows from \cref{thm:main-cancellation} and the third follows from \cref{fig:G_2}. It therefore remains to prove that (iii) implies (iv). We will prove the contrapositive.
The Fundamental Lemma (\cref{lemma:fundamental-lemma}) as stated only applies to $\EE$ and not $\EE(C_2^2)$, but we will use the same idea that went into its proof.

Let $G$ be a finite group such that $G \twoheadrightarrow C_2^2$ but which is not $H$-Eichler for $H=C_2^2$, $Q_8$ or $Q_{16}$, i.e. (iv) is not satisfied. 
Suppose for contradiction that (iii) is satisfied, i.e. $G$ has no quotient which is one of the groups listed in (iii).
Since $G$ is not $C_2^2$-Eichler, \cref{prop:groups1} implies that $G \twoheadrightarrow H' \in \MNEC(C_2^2) = \G_1(C_2^2)$. By \cref{fig:G_2(C_2^2)}, all the groups in $\G_1(C_2^2)$ are listed in (iii) except $Q_8$, $Q_{16}$, and so $H' = Q_8$, $Q_{16}$. By assumption, $G$ is not $H'$-Eichler and so \cref{prop:groups1} implies that $G \twoheadrightarrow H'' \in \MNEC(H')$. Since $G$ has no quotient in $B_1(C_2^2)$, we have that $H'' \in \MNEC_{B_1(C_2^2)}(H')$ and so $H''$ must have a quotient in $\MNEC_{B_1(C_2^2)}(\{Q_8,Q_{16}\}) = \G_2(C_2^2)$. By \cref{fig:G_2(C_2^2)}, all these groups are listed in (iii), which is a contradiction. Hence (iii) is not satisfied.
\end{proof}

\subsection{Groups with periodic cohomology} \label{ss:periodic-cohomology}

Recall that, if $P$ is a (finitely generated) projective $\Z G$-module, we say that $[P] \in \wt K_0(\Z G)$ has \textit{cancellation} if $P \oplus \Z G \cong P' \oplus \Z G$ implies $P \cong P'$ for all projective $\Z G$-modules $P'$.
Let $D(\Z G) \le \wt K_0(\Z G)$ denote the kernel group.

We say that a finite group $G$ has \textit{periodic cohomology} if, for some $k \ge 1$, its Tate cohomology groups satisfy $\hat{H}^i(G;\Z) = \hat{H}^{i+k}(G;\Z)$ for all $i \in \Z$. We will now prove the following.

\begin{thm} \label{thm:main-periodic-cohomology-later}
Let $G$ be a finite group with periodic cohomology. Then $\Z G$ has $\pc$ if and only if $m_{\H}(G) \le 2$. Furthermore, if $P$ is a projective $\Z G$-module, then:
\begin{clist}{(i)}
\item If $m_{\H}(G) \le 2$, then $[P]$ has cancellation
\item If $m_{\H}(G) = 3$, then:
\begin{clist}{(a)}
\item If the Sylow $2$-subgroup of $G$ is cyclic, then $[P]$ has non-cancellation
\item If $[P] \in D(\Z G)$, then $[P]$ has non-cancellation
\end{clist}
\item If $m_{\H}(G) \ge 4$, then $[P]$ has non-cancellation.
\end{clist}
\end{thm}

Recall from \cref{ss:relative-Eichler} that $\mathcal{B}(G)$ is the set of equivalence classes, denote by $\sim$, of quotients $f: G \twoheadrightarrow H$ where $H$ is a binary polyhedral group.
If $G$ has periodic cohomology and $f_1 , f_2 \in \mathcal{B}(G)$, then \cite[Corollary 1.4]{Ni20b} implies that $f_1 \simeq f_2 $ if and only if $\IM(f_1) \cong \IM(f_2)$.
This shows that there is a bijection $\mathcal{B}(G) \cong \{ H \text{ a binary polyhedral group} : G \twoheadrightarrow H \}$.
Given this, we will often write $H \in \mathcal{B}(G)$ when there exists $f: G \twoheadrightarrow H$ with $f \in \mathcal{B}(G)$.
 In order to determine $\mathcal{B}(G)$, it suffices to determine the set of maximal binary polyhedral quotients $\mathcal{B}_{\text{\normalfont max}}(G)$ which are the subset containing those $f \in \mathcal{B}(G)$ such that $f$ does not factor through any other $g \in \mathcal{B}(G)$.

\begin{proof} It suffices to prove (i)-(iii) since this implies that $\Z G$ has $\PC$ if and only if $m_{\H}(G) \le 2$.
We begin by showing the result for $G$ a binary polyhedral group (BPG). The BPGs with $m_{\H}(G) \le 2$ are $Q_8$, $Q_{12}$, $Q_{16}$, $Q_{20}$, $\wt T$, $\wt O$, $\wt I$ which all have $\PC$. Recall that $m_{\H}(Q_{4n}) = \lfloor n/2 \rfloor$.
The BPGs with $m_{\H}(G)=3$ are $Q_{24}$ and $Q_{28}$.
For $Q_{24}$, the Sylow $2$-subgroup is $Q_8$ (which is non-cyclic). The fact that $[P]$ has non-cancellation whenever $[P] \in D(\Z Q_{24})$ follows from close inspection of \cite[Theorem III]{Sw83}.
For $Q_{28}$, the Sylow $2$-subgroup is $C_4$ and $\Z Q_{28}$ has non-cancellation in every class by \cite[Theorem I]{Sw83}.
The BPGs with $m_{\H}(G) \ge 4$ are $Q_{4n}$ for $n \ge 7$. Again by \cite[Theorem I]{Sw83}, $\Z Q_{4n}$ has non-cancellation in every class for $n \ge 7$.

We will now prove the general case by verifying it in three separate cases according to $\# \mathcal{B}_{\text{max}}(G)$. 
First, if $\# \mathcal{B}_{\text{max}}(G) = 0$, then \cref{prop:relative-eichler-group} implies that $G$ is Eichler and so has $\PC$.
 
Next consider the case $\# \mathcal{B}_{\text{max}}(G) = 1$. Let $f : G \twoheadrightarrow H$ denote this quotient, with $H$ a BPG. By \cref{prop:relative-eichler-group}, we have $m_{\H}(G) = m_{\H}(H)$. If $m_{\H}(G) \le 2$, then $m_{\H}(H) \le 2$ and so $H$ is of the form $Q_8, Q_{12}, Q_{16}, Q_{20}, \wt T, \wt O, \wt I$ and $G$ has $\PC$ by \cref{thm:main-cancellation}. This proves (i) in this case.
Now assume $m_{\H}(G) \ge 3$ and let $P$ be a projective $\Z G$-module. If $P' = f_\#(P)$ and $[P']$ has non-cancellation, then $[P]$ has non-cancellation by \cref{thm:LF-COR}. Hence, to verify \cref{thm:main-periodic-cohomology-later}, it suffices to show that cancellation fails for $[P']$ in the cases (ii) and (iii).
 Suppose $m_{\H}(G) = m_{\H}(H) =3$. If the Sylow $2$-subgroup $\Syl_2(G)$ is cyclic, then $\Syl_2(H)$ is a quotient of $\Syl_2(G)$ and so is cyclic. Since the result holds for BPGs, this implies that $[P']$ has non-cancellation. If $[P] \in D(\Z G)$ then, since $f_\#$ induces a map $f_\# : D(\Z G) \to D(\Z H)$, we have $P' = f_\#(P) \in D(\Z H)$. Similarly, this implies that $[P']$ has non-cancellation which completes the proof of (ii). If $m_{\H}(G) = m_{\H}(H) \ge 4$, then $[P']$ has non-cancellation by the result for BPGs and this completes the proof of (iii) in this case.

Finally, we consider the case $\# \mathcal{B}_{\text{max}}(G) > 1$.
By \cite[Theorems 1.5 \& 1.10]{Ni20b}, we know that $\# \mathcal{B}_{\text{max}}(G) = 2$ or $3$. It can be shown (see the proof of \cite[Theorem 1.10]{Ni20b}) that $\mathcal{B}_{\text{max}}(G) = \{ Q_{8m_1}, \cdots, Q_{8m_b}\}$ where the $m_i$ are odd coprime and 
\[ m_{\H}(G) =
\begin{cases}
(m_1+m_2)-1, & \text{if $b=2$}\\
(m_1+m_2+m_3)-2, & \text{if $b=3$}
\end{cases}
\]
depending on the two cases that arise. We will assume that $m_1 < \cdots < m_b$. 
Note that $m_{\H}(G) =3$ only when $b=2$, $m_1 = 1$ and $m_2=3$. In this case, $G$ has a quotient $Q_{24}$ and so $\Syl_2(G)$ is non-cyclic. Furthermore, if $P$ is a projective $\Z G$ module with $[P] \in D(\Z G)$, then $[P]$ has non-cancellation by \cref{thm:LF-COR} since its image lies in $D(\Z Q_{24})$ under the extension of scalars map. In all other cases, we have $m_{\H}(G) \ge 5$ and $G$ must have a quotient $Q_{8m}$ for some $m \ge 5$. By the BPG case, we have that $\Z Q_{8m}$ has non-cancellation in every class. By \cref{thm:LF-COR} again, $\Z G$ has non-cancellation in every class. This completes the proof.
\end{proof}

In contrast to (ii), we have the following which follows from \cite[Theorem III]{Sw83}.

\begin{prop}
The quaternion group $Q_{24}$ of order $24$ has $m_{\H}(Q_{24})=3$ and a (non-cyclic) Sylow $2$-subgroup $Q_8$. Furthermore, $[\Z G]$ has non-cancellation but there exists a projective $\Z Q_{24}$ module $P$ with $[P] \not \in D(\Z Q_{24})$ for which $[P]$ has cancellation.	
\end{prop}

\FloatBarrier
\appendix

\section{Computation of Eichler simple groups}
\label{s:tables}

\label{ss:main-diagram}

Let $\EE = \bigcup_{n \ge 0} \MNEC^n(\{C_1\})$ and let $B \subseteq \EE$ denote the subset consisting of the groups:
\begin{clist}{(i)}
\item
$Q_{4n}$ for $n \ge 6$
\item
$Q_8 \times C_2$, \, $Q_{12} \times C_2$, \, $Q_{16} \times C_2$, \, $Q_{20} \times C_2$, \, $\widetilde{T} \times C_2^2$, \, $\widetilde{O} \times C_2$, \, $\widetilde{I} \times C_2^2$
\item
$G_{(32,14)}$, $G_{(36,7)}$, $G_{(64,14)}$, $G_{(100,7)}$.
\end{clist}
These are the groups in $\EE$ which, by \cite{Sw83,Ch86,BHJ24}, are currently known to fail PC and have no proper quotients in $\EE$ which fail PC.
Let $B_1 = \{\,Q_{4n} : n \ge 6\,\}$ and, for each $n \ge 2$, let $B_{n} = B_{n-1} \cup (B \cap \MNEC^n(\{C_1\})$. For each $n \ge 1$, define $\G_n = \MNEC^n_{B_n}(\{C_1\})$ and $\G = \bigcup_{n \ge 1} \G_n$. 
We have $B_2 = B \,\setminus\, \{\,\wt T \times C_2^2\,\}$ and $B_n = B$ for all $n \ge 3$.
We compute $\G_0$, $\G_1$ and $\G_2$ completely as well as those groups $G \in \G_3$ which have a quotient $H \in \G_2$ currently known to have PC, i.e. $\wt T \times C_2$ and $\wt T^n \times \wt I^m$ for $n+m=2$. We present the list of groups in both a diagram and a table.

\captionsetup{belowskip=-15pt}

\begin{figure}[h]
\def\layersep{4cm}
\begin{center}
\hspace{-16mm}
\begin{tikzpicture}[draw=black!75, node distance=\layersep]
\definecolor{BLUE}{HTML}{318CE7}
\definecolor{RED}{HTML}{	E32636}
    \tikzstyle{neuron}=[fill=white,draw=black!75,minimum size=20pt,inner sep=0pt,rounded corners]
    \tikzstyle{bbneuron}=[pattern=horizontal lines, pattern color=BLUE, draw=black!75,minimum size=20pt,inner sep=0pt,rounded corners]
    \tikzstyle{bneuron}=[fill=BLUE!65, draw=black!75,minimum size=20pt,inner sep=0pt,rounded corners]
    \tikzstyle{rneuron}=[fill=red!50, draw=black!75,minimum size=20pt,inner sep=0pt,rounded corners]
    \tikzstyle{annot} = [text width=4em, text centered]

    \node[bneuron] (I) at (1.5,-8) {$C_1$};

\def\sep{1.8}
        \path[yshift=1cm]node[bneuron] (H-1) at (\layersep,-\sep*1 cm) {$Q_8$};
        \path[yshift=1cm]node[bneuron] (H-4) at (\layersep,-\sep*2 cm) {$Q_{12}$};
        \path[yshift=1cm]node[bneuron] (H-2) at (\layersep,-\sep*3 cm) {$Q_{16}$};
        \path[yshift=1cm]node[bneuron] (H-3) at (\layersep,-\sep*4 cm) {$Q_{20}$};        
        \path[yshift=1cm]node[bneuron] (H-6) at (\layersep,-\sep*5 cm) {$\wt T$};        
        \path[yshift=1cm]node[bneuron] (H-5) at (\layersep,-\sep*6 cm) {$\wt O$};
        \path[yshift=1cm]node[bneuron] (H-7) at (\layersep,-\sep*7 cm) {$\wt I$};
        \path[yshift=1cm]node[rneuron] (H-8) at (\layersep,-\sep*8 cm) {\text{ $Q_{4n}$, $n \ge 6$ }};

\foreach \x in {1,...,8} \node (H-\x) at ($(H-\x) +(0:-0.2) $) {}; 

\begin{scope}[on background layer]
    \foreach \source in {1}
        \foreach \dest in {1,...,8}
            \path (I) edge (H-\dest);
\end{scope}

\foreach \x in {1,...,7} \node (H-\x) at ($(H-\x) +(0:0.4) $) {}; 

\def\sep{0.8}
        \path[yshift=0.5cm]node[rneuron] (L-1) at (2*\layersep,-\sep*1 cm) {\text{ $Q_8 \times C_2$ }};
        \path[yshift=0.5cm]node[rneuron] (L-7) at (2*\layersep,-\sep*2 cm) {\text{ $Q_{12} \times C_2$ }};
        \path[yshift=0.5cm]node[rneuron] (L-2) at (2*\layersep,-\sep*3 cm) {\text{ $G_{(32,14)}$ }}; 
        \path[yshift=0.5cm]node[rneuron] (L-3) at (2*\layersep,-\sep*4 cm) {\text{ $Q_{16} \times C_2$ }};
        \path[yshift=0.5cm]node[rneuron] (L-8) at (2*\layersep,-\sep*5 cm) {\text{ $\,G_{(36,7)}\,$ }}; 
        \path[yshift=0.5cm]node[rneuron] (L-5) at (2*\layersep,-\sep*6 cm) {\text{ $Q_{20} \times C_2$ }};
       \path[yshift=0.5cm]node[bneuron] (L-13) at (2*\layersep,-\sep*7 cm) {\text{ $\,\wt T \times C_2\,$ }};
        \path[yshift=0.5cm]node[rneuron] (L-4) at (2*\layersep,-\sep*8 cm) {\text{ $G_{(64,14)}$ }}; 
        \path[yshift=0.5cm]node[neuron] (L-9) at (2*\layersep,-\sep*9 cm) {\text{ \hspace{11.5mm} }}; 
       \path[yshift=0.5cm]node[bbneuron] (L-9) at (2*\layersep,-\sep*9 cm) {\text{ $G_{(96,66)}$ }}; 
        \path[yshift=0.5cm]node[rneuron] (L-10) at (2*\layersep,-\sep*10 cm) {\text{ $\wt O \times C_2$ }};
        \path[yshift=0.5cm]node[rneuron] (L-6) at (2*\layersep,-\sep*11 cm) {\text{ $G_{(100,7)}$ }}; 
        \path[yshift=0.5cm]node[neuron] (L-14) at (2*\layersep,-\sep*12 cm) {\text{ \hspace{12.25mm} }}; 
       \path[yshift=0.5cm]node[bbneuron] (L-14) at (2*\layersep,-\sep*12 cm) {\text{ $\,Q_8 \rtimes \wt T\,$ }};
        \path[yshift=0.5cm]node[bbneuron] (L-18) at (2*\layersep,-\sep*13 cm) {\text{ $\wt I \times C_2$ }};
         \path[yshift=0.5cm]node[neuron] (L-11) at (2*\layersep,-\sep*14 cm) {\text{ $G_{(384,18129)}$ }}; 
        \path[yshift=0.5cm]node[bneuron] (L-15) at (2*\layersep,-\sep*15 cm) {\text{ $\quad\wt T^2\quad$ }};        
         \path[yshift=0.5cm]node[neuron] (L-12) at (2*\layersep,-\sep*16 cm) {\text{ $G_{(1152, 155476)}$ }}; 
        \path[yshift=0.5cm]node[bneuron] (L-16) at (2*\layersep,-\sep*17 cm) {\text{ $\,\,\wt T \times \wt I\,\,$ }};
        \path[yshift=0.5cm]node[bneuron] (L-17) at (2*\layersep,-\sep*18 cm) {\text{ $\quad\wt I^2\quad$ }};

\foreach \x in {1,...,18} \node (L-\x) at ($(L-\x) +(0:-0.5) $) {}; 

\begin{scope}[on background layer]
	\path (H-1) edge (L-1) edge (L-2);
	\path (H-2) edge (L-2) edge (L-3) edge (L-4);
	\path (H-3) edge (L-5) edge (L-6);
	\path (H-4) edge (L-7) edge (L-8) edge (L-9);
	\path (H-5) edge (L-9) edge (L-10) edge (L-11) edge (L-12);
	\path (H-6) edge (L-13) edge (L-14) edge (L-15) edge (L-16);
	\path (H-7) edge (L-16) edge (L-17) edge (L-18);
\end{scope}

\foreach \x in {1,...,18} \node (L-\x) at ($(L-\x) +(0:1.08) $) {}; 

\def\sep{1.1}
        \path[yshift=0.5cm]node[rneuron] (M-3) at (3*\layersep,-\sep*1 cm) {\text{ $\,\,\wt T \times C_2^2\,\,$ }};       
        \path[yshift=0.5cm]node[neuron] (M-4) at (3*\layersep,-\sep*2 cm) {\text{ \hspace{13.75mm} }}; 
        \path[yshift=0.5cm]node[bbneuron] (M-4) at (3*\layersep,-\sep*2 cm) {\text{ $\,\wt T \times Q_{12}\,$ }};         
        \path[yshift=0.5cm]node[neuron] (M-7) at (3*\layersep,-\sep*3 cm) {\text{ $(Q_8 \rtimes \wt T) \times C_2$ }};     
        \path[yshift=0.5cm]node[neuron] (M-5) at (3*\layersep,-\sep*4 cm) {\text{ \hspace{12.6mm} }}; 
        \path[yshift=0.5cm]node[bbneuron] (M-5) at (3*\layersep,-\sep*4 cm) {\text{ $\wt T \times Q_{20}$ }};
        \path[yshift=0.5cm]node[neuron] (M-6) at (3*\layersep,-\sep*5 cm) {\text{ $\,\,\wt T \times \wt O\,\,$ }};
        \path[yshift=0.5cm]node[neuron] (M-8) at (3*\layersep,-\sep*6 cm) {\text{ $\wt T^2 \times C_2$ }};        
        \path[yshift=0.5cm]node[neuron] (M-11) at (3*\layersep,-\sep*7 cm) {\text{ $Q_8 \rtimes \wt T^2$ }};
        \path[yshift=0.5cm]node[neuron] (M-9) at (3*\layersep,-\sep*8 cm) {\text{ $\wt T \times (Q_8 \rtimes \wt T)$ }};   
        \path[yshift=0.5cm]node[bneuron] (M-12) at (3*\layersep,-\sep*9 cm) {\text{ $\quad\wt T^3\quad$ }};
        \path[yshift=0.5cm]node[neuron] (M-10) at (3*\layersep,-\sep*10 cm) {\text{ $\wt I \times (Q_8 \rtimes \wt T)$ }};               
        \path[yshift=0.5cm]node[bneuron] (M-13) at (3*\layersep,-\sep*11 cm) {\text{ $\,\wt T^2 \times \wt I\,$ }};
        \path[yshift=0.5cm]node[bneuron] (M-14) at (3*\layersep,-\sep*12 cm) {\text{ $\,\wt T \times \wt I^2\,$ }};
        \path[yshift=0.5cm]node[bneuron] (M-15) at (3*\layersep,-\sep*13 cm) {\text{ $\quad\wt I^3\quad$ }};
       
\foreach \x in {3,...,15} \node (M-\x) at ($(M-\x) +(0:-0.6) $) {}; 

\node (M-3) at ($(M-3) +(0:-0.1) $) {}; 
\node (M-7) at ($(M-7) +(0:-0.5) $) {};
 
\begin{scope}[on background layer]        
	\path (L-13) edge (M-3) edge (M-4) edge (M-5) edge (M-6) edge (M-7) edge (M-8);
	\path (L-15) edge (M-8) edge (M-9) edge (M-11) edge (M-12) edge (M-13);
	\path (L-16) edge (M-10) edge (M-13) edge (M-14);
	\path (L-17) edge (M-14) edge (M-15);
	
	\path (H-4) edge [dashed] (M-4);
	\path (H-3) edge [dashed] (M-5);
	\path (H-5) edge [dashed] (M-6);
\end{scope}       	
       	    	
    \node (g0) at (0,0.6) {};
    \node[annot,right of=g0] (g1) {$\G_1$};
    \node[annot,right of=g1] (g2) {$\G_2$};
    \node (g3) at (12.05,0.6) {$\G_3$ \text{(partial)}};
\end{tikzpicture}
\end{center}
\caption{Graphical structure of $\G_0$, $\G_1$, $\G_2$ and some of $\G_3$. Key: Blue (PC), Red (fails PC), White/Blue (SFC with PC unknown), White (No information).
}
\label{fig:G_2}
\end{figure}
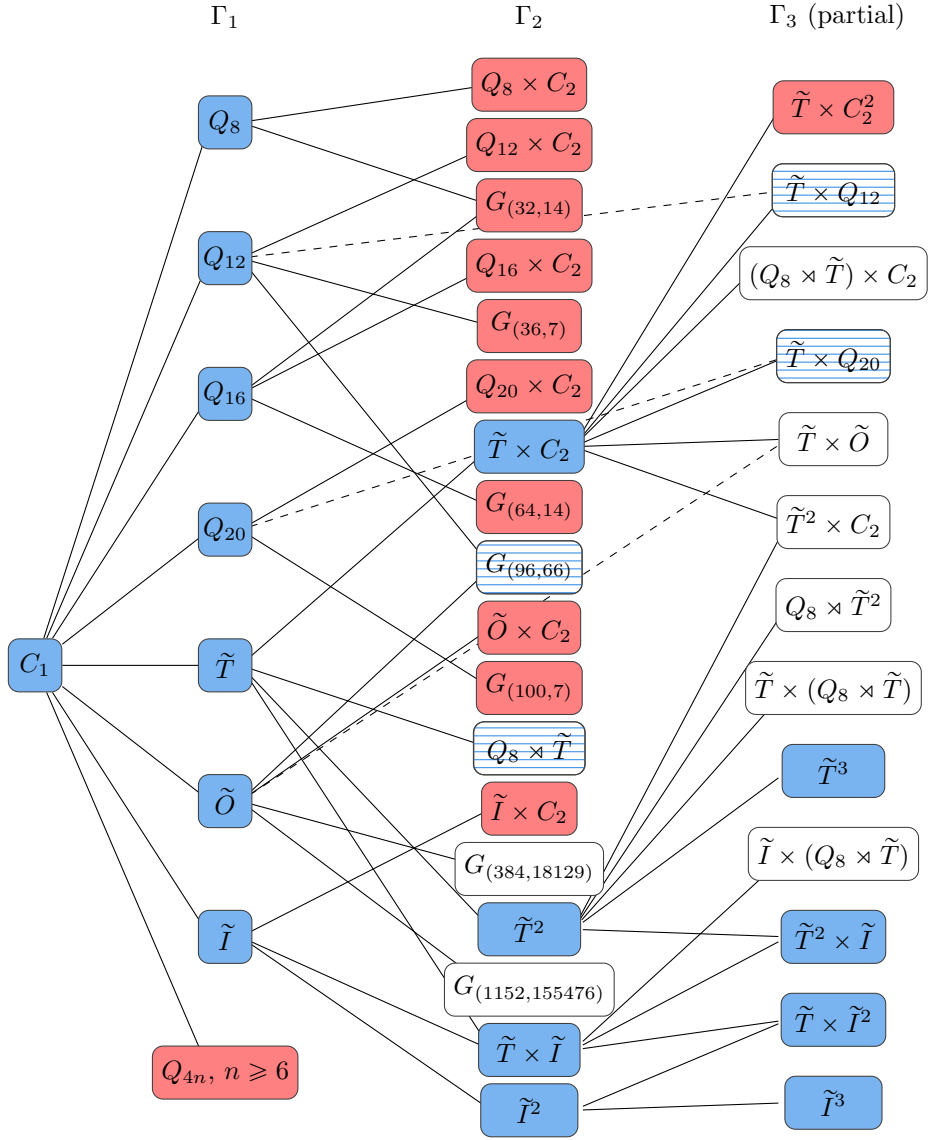

\subsection{Diagram of groups}

In \cref{fig:G_2}, we sketch $\mathcal{G} = \G_0 \sqcup \G_1 \sqcup \G_2 \sqcup \G_3'$ where $\G_3' \subseteq \G_3$ is the subset defined above. If $G, H \in \mathcal{G}$ and $G \in \MNEC(H)$, we draw a line from $G$ to $H$. Three of the lines are dashed to indicate that they pass between groups in $\G_1$ and $\G_3$. For example, there is an edge from $\wt T \times Q_{12} \in \G_3$ to $Q_{12} \in \G_1$.

\FloatBarrier
\subsection{Table of groups}
  
The column headings are as follows.
\begin{clist}{1.}
\item 
\textbf{MNEC ID.}
We name each group by an MNEC ID $(n,m)$ which refers to the $m$th group in $\G_n$. The corresponding group is denoted by $\G_{(n,m)}$.
The ordering is chosen so that $|\G_{(n,m)}| \le |\G_{(n,m')}|$ for $m \le m'$, but otherwise follows no specific pattern. For each $n \le 3$, the groups $\G_{(n,1)}, \G_{(n,2)}, \ldots$ coincide with groups in the diagram listed from top to bottom.

\item
\textbf{Group ID / Order.} Where possible, we will list the Small Group ID $(n,m)$ which refers to the $m$th group of order $n$ in GAP's Small Groups library \cite{BEO02,GAP4}. The corresponding group is denoted by $G_{(n,m)}$. Otherwise, we just list the order of the group.
\item
\textbf{Edges To.} For a group $G \in \G_n$, we list the MNEC IDs of all groups $H \in \G_{m}$ for which $G \in \MNEC(H)$, i.e. they are connected by an edge in $\EE$. Note that this determines all quotients between groups in $\mathcal{G}$ since they correspond to directed paths in $\EE$ (see \cref{prop:quotients<=>path}).
\item
\textbf{m}$_{\H}$($\bm G$). We compute $m_{\H}(G)$ using \cref{prop:m_H-formula}.
\item
\textbf{Description.}
For groups which are direct products of previously defined groups, we list the groups (see \cref{s:groups-exceptional} for the definition of $Q_8 \rtimes \wt T^n$).
We give descriptions to the groups (32,14), (36,7), (64,14), (96,66), (100,7), (384, 18129) and (1152,155476).
\item
\textbf{Cancellation.} 
We record whether $G$ is known to have PC, known to fail SFC, or is known to have SFC but with PC currently unknown. If we have no information, we write a dash.
\end{clist}

\begin{center}
\begin{longtable}{|c|c|c|c|Sc|c|}
\hline
MNEC ID & Group ID/Order & Edges To & $m_{\H}(G)$ & Description & Cancellation \\ \hline 
$(1,1)$ & $(8,4)$ & $(0,1)$ & 1 & $Q_8$ & PC \\ \hline
$(1,2)$ & $(12,1)$ & $(0,1)$ & 1 & $Q_{12}$ & PC \\ \hline
$(1,3)$ & $(16,9)$ & $(0,1)$ & 2 & $Q_{16}$ & PC \\ \hline
$(1,4)$ & $(20,1)$ & $(0,1)$ & 2 & $Q_{20}$ & PC \\ \hline
$(1,5)$ & $(24,3)$ & $(0,1)$ & 1 & $\wt T$ ($=\SL_2(\F_3)$) & PC \\ \hline
$(1,6)$ & $(48,28)$ & $(0,1)$ & 2 & $\wt O$ ($=\text{CSU}_2(\F_3)$) & PC \\ \hline
$(1,7)$ & $(120,5)$ & $(0,1)$ & 2 & $\wt I$ ($=\SL_2(\F_5)$) & PC \\ \hline
$(2,1)$ & $(16,12)$ & $(1,1)$ & 2 & $Q_8 \times C_2$ & Fails SFC \\ \hline
$(2,2)$ & $(24,7)$ & $(1,2)$ & 2 & $Q_{12} \times C_2$ & Fails SFC \\ \hline
$(2,3)$ & $(32,14)$ & $(1,1), (1,3)$ & 3 & $C_4 \cdot Q_8$ & Fails SFC \\ \hline
$(2,4)$ & $(32,41)$ & $(1,3)$ & 4 & $Q_{16} \times C_2$  & Fails SFC \\ \hline
$(2,5)$ & $(36,7)$ & $(1,2)$ & 4 & $C_3 \rtimes Q_{12}$ & Fails SFC \\ \hline
$(2,6)$ & $(40,7)$ & $(1,4)$ & 4 & $Q_{20} \times C_2$ & Fails SFC \\ \hline
$(2,7)$ & $(48,32)$ & $(1,5)$ & 2 & $\wt T \times C_2$ & PC  \\ \hline
$(2,8)$ & $(64,14)$ & $(1,3)$ & 4 & $C_4 \cdot Q_{16}$ & Fails SFC \\ \hline
$(2,9)$ & $(96,66)$ & $(1,2), (1,6)$ & 3 & $Q_8 \rtimes Q_{12}$ & SFC (PC?) \\ \hline
$(2,10)$ & $(96,188)$ & $(1,6)$ & 4 & $\wt O \times C_2$ & Fails SFC \\ \hline
$(2,11)$ & $(100,7)$ & $(1,4)$ & 12 & $C_5 \rtimes Q_{20}$ & Fails SFC \\ \hline
$(2,12)$ & $(192,1022)$ & $(1,5)$ & 2 & $Q_8 \rtimes \wt T$ & SFC (PC?) \\ \hline
$(2,13)$ & $(240,94)$ & $(1,7)$ & 4 & $\wt I \times C_2$ & SFC (PC?) \\ \hline
$(2,14)$ & $(384,18129)$ & $(1,6)$ & 4 & $Q_8 \rtimes \wt O$ & - \\ \hline
$(2,15)$ & $(576,5128)$ & $(1,5)$ & 2 & $\wt T^2$ & PC \\ \hline
$(2,16)$ & $(1152,155476)$ & $(1,6)$ & 4 & $\wt T \times \wt O$ & -\\ \hline
$(2,17)$ & $2880$ & $(1,5), (1,7)$ & 3 & $\wt T \times \wt I$ & PC \\ \hline
$(2,18)$ & $14400$ & $(1,7)$ & 4 & $\wt I^2$ & PC \\ \hline
$(3,1)$ & $(96,198)$ & $(2,7)$ & 4 & $\wt T \times C_2^2$ & Fails SFC \\ \hline
$(3,2)$ & $(288,409)$ & $(1,2), (2,7)$ & 3 & $\wt T \times Q_{12}$ & SFC (PC?) \\ \hline
$(3,3)$ & $(384,18228)$ & $(2,7), (2,12)$ & 4 & $(Q_8 \rtimes \wt T) \times C_2$ & - \\ \hline
$(3,4)$ & $(480,266)$ & $(1,4), (2,7)$ & 4 & $\wt T \times Q_{20}$ & SFC (PC?) \\ \hline
$(3,5)$ & $(1152,155456)$ & $(1,6), (2,7)$ & 4 & $\wt T \times \wt O$ & - \\ \hline
$(3,6)$ & $(1152,156570)$ & $(2,7), (2,15)$ & 4 & $\wt T^2 \times C_2$ & - \\ \hline
$(3,7)$ & $4608$ & $(2,15)$ & 3 & $Q_8 \rtimes \wt T^2$ & - \\ \hline
$(3,8)$ & $4608$ & $(2,12), (2,15)$ & 3 & $\wt T \times (Q_8 \rtimes \wt T)$ & - \\ \hline
$(3,9)$ & $13824$ & $(2,15)$ & 3 & $\wt T^3$ & PC \\ \hline
$(3,10)$ & $23040$ & $(2,12), (2,17)$ & 4 & $\wt I \times (Q_8 \rtimes \wt T)$ & - \\ \hline
$(3,11)$ & $69120$ & $(2,15), (2,17)$ & 4 & $\wt T^2 \times \wt I$ & PC \\ \hline
$(3,12)$ & $345600$ & $(2,17), (2,18)$ & 5 & $\wt T \times \wt I^2$ & PC \\ \hline
$(3,13)$ & $1728000$ & $(2,18)$ & 6 & $\wt I^3$ & PC \\ \hline
\end{longtable}
\end{center}

\FloatBarrier
\bibliography{biblio.bib}

\providecommand{\bysame}{\leavevmode\hbox to3em{\hrulefill}\thinspace}
\providecommand{\MR}{\relax\ifhmode\unskip\space\fi MR }
\providecommand{\MRhref}[2]{%
  \href{http://www.ams.org/mathscinet-getitem?mr=#1}{#2}
}
\providecommand{\href}[2]{#2}
\begin{thebibliography}{MOV83}

\bibitem[AM04]{AM94}
A.~Adem and R.~J. Milgram, \emph{Cohomology of finite groups}, Grundlehren der
  mathematischen Wissenschaften, vol. 309, Springer-Verlag Berlin Heidelberg,
  2004.

\bibitem[BCP97]{magma}
W~Bosma, J.~Cannon, and C.~Playoust, \emph{The magma algebra system {I}: the
  user language}, J. Symbolic Comput. \textbf{24} (1997), 235--265, See also
  \href{http://magma.maths.usyd.edu.au/magma/}{\texttt{http://magma.maths.usyd.edu.au/magma/}}.

\bibitem[BEO02]{BEO02}
Hans~Ulrich Besche, Bettina Eick, and E.~A. O'Brien, \emph{A millennium
  project: constructing small groups}, Internat. J. Algebra Comput. \textbf{12}
  (2002), no.~5, 623--644.

\bibitem[BHJ24]{BHJ24}
Werner Bley, Tommy Hofmann, and Henri Johnston, \emph{Determination of the
  stably free cancellation property for orders}, 2024, arXiv:2407.02294.

\bibitem[Bro82]{Br82}
K.~S. Brown, \emph{Cohomology of groups}, Springer-Verlag New York Inc, 1982.

\bibitem[Bur22]{Bu22}
David Burrell, \emph{On the number of groups of order 1024}, Comm. Algebra
  \textbf{50} (2022), no.~6, 2408--2410.

\bibitem[BW05]{BW05}
F.~R. Beyl and N.~Waller, \emph{A stably free nonfree module and its relevance
  for homotopy classification, case {$Q_{28}$}}, Algebr. Geom. Topol.
  \textbf{5} (2005), 899--910.

\bibitem[BW08]{BW08}
\bysame, \emph{Examples of exotic free complexes and stably free nonfree
  modules for quaternion groups}, Algebr. Geom. Topol. \textbf{8} (2008),
  1--17.

\bibitem[Che86]{Ch86}
H.~F. Chen, \emph{The locally free cancellation property of the group ring {$\Z
  G$}}, PhD Thesis, University of Illinois at Urbana-Champaign (1986).

\bibitem[CR81]{CR81}
C.~W. Curtis and I.~Reiner, \emph{Methods of representation theory: with
  applications to finite groups and orders}, vol.~1, Wiley Classics Library,
  1981.

\bibitem[CR87]{CR87}
\bysame, \emph{Methods of representation theory: With applications to finite
  groups and orders}, vol.~2, Wiley Classics Library, 1987.

\bibitem[Dok]{groupnames}
T.~Dokchitser, \emph{Groupnames},
  https://people.maths.bris.ac.uk/~matyd/GroupNames/.

\bibitem[Eic37]{Ei37}
M.~Eichler, \emph{\"{U}ber die idealklassenzahl total defmiter
  quaternionalgebren}, Math. Z. \textbf{43} (1937), 102--109.

\bibitem[EM47]{EM47}
S.~Eilenberg and S.~MacLane, \emph{Cohomology theory in abstract groups. {II}.
  {G}roup extensions with a non-{A}belian kernel}, Ann. of Math. (2)
  \textbf{48} (1947), 326--341.

\bibitem[Fr{\"{o}}73]{Fr73}
A.~Fr{\"{o}}hlich, \emph{The picard group of noncommutative rings, in
  particular of orders}, Trans. Amer. Math. Soc. \textbf{180} (1973), 1--45.

\bibitem[Fr{\"{o}}75]{Fr75}
\bysame, \emph{Locally free modules over arithmetic orders}, J. Reine Angew.
  Math \textbf{274/275} (1975), 112--138.

\bibitem[FRU74]{FRU74}
A.~Fr{\"{o}}hlich, I.~Reiner, and S.~Ullom, \emph{Class groups and picard
  groups of orders}, Proc. London Math. Soc. \textbf{29} (1974), no.~3,
  405--434.

\bibitem[GAP]{GAP4}
The GAP~Group, \emph{{GAP -- Groups, Algorithms, and Programming, Version
  4.13.0, 2024}}.

\bibitem[GR88]{GR88}
W.~H. Gustafson and K.~W. Roggenkamp, \emph{A {M}ayer-{V}ietoris sequence for
  {P}icard groups, with applications to integral group rings of dihedral and
  quaternionic groups}, Ill. Jour. of Math. \textbf{32} (1988), no.~3,
  375--406.

\bibitem[Hal18]{Hall18}
M.~Hall, \emph{The theory of groups}, Dover Publications, 2018.

\bibitem[HK88a]{HK88}
I.~Hambleton and M.~Kreck, \emph{On the classification of topological
  4-manifolds with finite fundamental group}, Math. Annalen \textbf{280}
  (1988), 85--104.

\bibitem[HK88b]{HK88b}
Ian Hambleton and Matthias Kreck, \emph{Smooth structures on algebraic surfaces
  with cyclic fundamental group}, Invent. Math. \textbf{91} (1988), no.~1,
  53--59.

\bibitem[Jac68]{Ja68}
H.~Jacobinski, \emph{Genera and decompositions of lattices over orders}, Acta
  mathematica \textbf{121} (1968), 1--29.

\bibitem[JM80]{JM80}
S.~Jajodia and B.~Magurn, \emph{Surjective stability of units and simple
  homotopy type}, J. Pure Appl. Algebra \textbf{18} (1980), 45--58.

\bibitem[Joh03]{Jo03-book}
F.~E.~A. Johnson, \emph{Stable modules and the {D(2)}-problem}, London Math.
  Soc. Lecture Note Ser., vol. 301, Cambridge University Press, 2003.

\bibitem[Mil26]{Mi26}
G.~A. Miller, \emph{Subgroups of {I}ndex {$p^2$} {C}ontained in a {G}roup of
  {O}rder {$p^m$}}, Amer. J. Math. \textbf{48} (1926), no.~4, 253--256.

\bibitem[Mil71]{Mi71}
J.~Milnor, \emph{Introduction to algebraic {K}-theory}, Ann. of Math. Stud.,
  vol.~72, Princeton University Press, 1971.

\bibitem[MOV83]{MOV83}
B.~Magurn, R.~Oliver, and L.~Vaserstein, \emph{Units in {W}hitehead groups of
  finite groups}, J. Algebra \textbf{84} (1983), 324--360.

\bibitem[Nic20]{Ni20b}
J.~Nicholson, \emph{Cancellation for {$(G,n)$}-complexes and the {S}wan
  finiteness obstruction}, Int. Math. Res. Not., to appear (2020),
  arXiv:2005.01664.

\bibitem[Nic21a]{Ni18}
\bysame, \emph{A cancellation theorem for modules over integral group rings},
  Math. Proc. Cambridge Philos. Soc. \textbf{171} (2021), no.~2, 317--327.

\bibitem[Nic21b]{Ni19}
\bysame, \emph{On {CW}-complexes over groups with periodic cohomology}, Trans.
  Amer. Math. Soc. \textbf{374} (2021), no.~9, 6531--6557.

\bibitem[Nic21c]{Ni21b}
\bysame, \emph{Projective modules over integral group rings and {W}all's {D}2
  problem}, Ph.D. thesis, University College London, 2021.

\bibitem[Oli88]{Ol88}
R.~Oliver, \emph{Whitehead groups of finite groups}, London Math. Soc. Lecture
  Note Ser., vol. 132, Cambridge University Press, 1988.

\bibitem[RU74]{RU74}
I.~Reiner and S.~Ullom, \emph{A {M}ayer-{V}ietoris sequence for class groups},
  J. of Algebra \textbf{31} (1974), 305--342.

\bibitem[Ser77]{Se77}
Jean-Pierre Serre, \emph{Linear representations of finite groups}, Graduate
  Texts in Mathematics, Vol. 42, Springer-Verlag, New York-Heidelberg, 1977,
  Translated from the second French edition by Leonard L. Scott.

\bibitem[SV19]{SV19}
D.~Smertnig and J.~Voight, \emph{Definite orders with locally free
  cancellation}, Trans. of the London Math. Soc. \textbf{6} (2019), 53--86.

\bibitem[Swa60a]{Sw60-II}
R.~G. Swan, \emph{Induced representations and projective modules}, Ann. of Math
  \textbf{71} (1960), no.~2, 552--578.

\bibitem[Swa60b]{Sw60-I}
\bysame, \emph{Periodic resolutions for finite groups}, Ann. of Math
  \textbf{72} (1960), no.~2, 267--291.

\bibitem[Swa62]{Sw62}
\bysame, \emph{Projective modules over group rings and maximal orders}, Ann. of
  Math. (2) \textbf{76} (1962), 55--61.

\bibitem[Swa80]{Sw80}
\bysame, \emph{Strong approximation and locally free modules}, Ring Theory and
  Algebra III, Proceedings of the third Oklahoma Conference 3, 1980,
  pp.~153--223.

\bibitem[Swa83]{Sw83}
\bysame, \emph{Projective modules over binary polyhedral groups}, J. Reine
  Angew. Math \textbf{342} (1983), 66--172.

\bibitem[Tay81]{Ta81}
M.~J. Taylor, \emph{On {F}r{\"{o}}hlich's conjecture for rings of integers of
  tame extensions}, Invent. Math. \textbf{63} (1981), no.~1, 41--79.

\bibitem[Wal74]{Wa74}
C.~T.~C. Wall, \emph{Norms of units in group rings}, Proc. London Math. Soc.
  (3) \textbf{29} (1974), 593--632.

\bibitem[Wei94]{We94}
C.~A. Weibel, \emph{An introduction to homological algebra}, Cambridge Studies
  in Advanced Mathematics, vol.~38, Cambridge University Press, 1994.

\end{thebibliography}
\bibliographystyle{amsalpha}

\end{document}